\documentclass[pdflatex,sn-mathphys-num]{sn-jnl}

\makeatletter
\Gm@hbodyfalse
\Gm@vbodyfalse
\makeatother
\usepackage{amssymb,mathtools,mathrsfs}
\usepackage{amsmath}
\usepackage{booktabs,array,tabularx}
\usepackage{algorithm,algpseudocode}
\algrenewcommand\algorithmicrequire{\textbf{Input:}}
\usepackage{enumitem,microtype}
\usepackage{xcolor}
\usepackage{xurl,bookmark}
\hypersetup{
	hypertexnames=false,
	colorlinks=true,
	linkcolor=black,
	citecolor=black,
	urlcolor=black
}
\numberwithin{equation}{section}

\theoremstyle{thmstyleone}
\newtheorem{theorem}{Theorem}[section]
\newtheorem{lemma}{Lemma}[section]
\newtheorem{proposition}{Proposition}[section]
\newtheorem{corollary}{Corollary}[section]

\theoremstyle{thmstylethree}
\newtheorem{assumption}{Assumption}[section]
\newtheorem{definition}{Definition}[section]

\theoremstyle{thmstyletwo}
\newtheorem{remark}{Remark}[section]

\newcommand{\R}{\mathbb R}

\newcommand{\ip}[2]{\langle #1,#2\rangle}
\newcommand{\norm}[1]{\lVert #1\rVert}
\DeclareMathOperator{\dist}{dist}
\DeclareMathOperator{\supp}{supp}
\DeclareMathOperator{\diam}{diam}

\DeclareMathOperator*{\argmin}{argmin}

\newcommand{\rev}[1]{{\color{black}#1}}
\newenvironment{revision}
{\begingroup\color{black}}
{\endgroup}

\allowdisplaybreaks[1]
\author[1]{\fnm{Minhao} \sur{Zhang}}
\email{zhangminhao@shu.edu.cn}

\author*[1]{\fnm{Zi} \sur{Xu}}
\email{xuzi@shu.edu.cn}

\author[2]{\fnm{Anthony Man-Cho} \sur{So}}
\email{manchoso@se.cuhk.edu.hk}

\affil[1]{Department of Mathematics, College of Sciences,
	Shanghai University, Shanghai, 200444, China}
\affil[2]{Department of Systems Engineering and Engineering Management, The Chinese University of Hong Kong, HKSAR, China}

\begin{document}
	
	\title[]{
		A Single-Loop Method with Optimal Leading-Order Oracle
		Complexity for Optimization and Game Stationarity
		in Nonconvex--Concave Minimax Optimization
	}
	
	\abstract{
	We develop a single-loop first-order framework for smooth
	nonconvex--concave minimax optimization, with matching
	leading-order upper and lower bounds for optimization
	stationarity and game stationarity in a projected
	zero-respecting oracle model.
	To establish the upper bound, we propose a single-loop
	projected damped extragradient method, which couples
	projected extragradient updates with dual momentum and a
	moving proximal center.
	Each iteration requires at most three gradient evaluations,
	without the need to solve nested subproblems.
	For game stationarity, the method requires a
	leading oracle complexity
	$\mathcal{O}(L^{3/2}D_Y^{1/2}\Delta_\phi\varepsilon^{-5/2})$
	for returning a deterministic best-iterate whose
	first-order residual is below $\varepsilon$.
	Here, $L$ is the gradient Lipschitz constant
	of the objective function, $D_Y$ bounds
	the diameter of the dual feasible set, and
	$\Delta_\phi$ is the value-function gap that
	depends on the initial point.
	For optimization stationarity, the method requires a
	leading oracle complexity of
	$\mathcal{O}(L^2D_Y\Delta_\phi\varepsilon^{-3})$
	for returning a uniformly selected iterate whose
	expected squared Moreau-envelope gradient norm is at most $\varepsilon^2$.
	We also establish matching lower bounds of
	$\Omega(L^{3/2}D_Y^{1/2}\Delta_\phi\varepsilon^{-5/2})$
	for game stationarity and
	$\Omega(L^2D_Y\Delta_\phi\varepsilon^{-3})$ for optimization
	stationarity over projected zero-respecting first-order methods.
	Under dual strong concavity, the same framework achieves
	a leading first-order oracle complexity of
	$\mathcal{O}(\sqrt{\kappa}\,L\Delta_\phi\varepsilon^{-2})$
	under both stationarity criteria, where $\kappa$ is the dual
	condition number.
	Matching lower bounds hold for both criteria in the
	same oracle model.
	To the best of our knowledge, this is the first
	work to close the oracle complexity gap for smooth
	nonconvex--concave minimax optimization under both stationarity
	criteria by simultaneously improving the upper and lower bounds
	in the projected zero-respecting oracle model.

	}
	
	
	\maketitle
	
	\section{Introduction}\label{sec:introduction}
	
	We study the deterministic nonconvex--concave (NC--C) minimax
	problem
	\begin{equation}\tag{P}\label{eq:problem}
		\min_{x\in X}\max_{y\in Y}f(x,y),
	\end{equation}
	where $X$ is nonempty, closed, and convex, $Y$ is nonempty,
	compact, and convex, and $f$ has an $L$-Lipschitz continuous
	gradient.
	The objective is possibly nonconvex in $x$ and concave in $y$.
	Smooth instances arise in distributionally robust
	learning~\cite{rafique2022weakly}, group robust
	learning~\cite{sagawa2020distributionally}, and learning with
	average top-$k$ losses~\cite{fan2017topk}.
	In these applications, the dual variable weights samples or
	groups, or selects large losses, while the primal variable
	parametrizes a possibly nonconvex model.
	An important algorithmic question is whether optimal
	stationarity complexity can be attained by a single-loop
	first-order method, without repeatedly solving auxiliary
	optimization problems.
	
	We consider two stationarity criteria.
	Optimization stationarity (OS) measures the gradient of a
	Moreau envelope of the value function
	$\phi(x)=\max_{y\in Y}f(x,y)$, extended by $+\infty$ outside $X$.
	Game stationarity (GS) measures the joint first-order residual of $f$.
	These criteria describe different aspects of stationarity:
	OS concerns the minimization of the value function, whereas
	GS concerns the first-order conditions of the minimax objective
	in both variables.
	They should also be distinguished from the value-gradient
	and projected-gradient-mapping criteria used in earlier work.
	Formal definitions are given in Section~\ref{sec:preliminaries}.
	Throughout this introduction, complexity refers to first-order
	oracle complexity.
	Unless displayed explicitly, bounds suppress fixed smoothness,
	domain, and initialization quantities;
	$\widetilde{\mathcal{O}}$ additionally suppresses logarithmic factors.
	
	\paragraph{Related work.}
	
	In the nonconvex--strongly concave (NC--SC) setting, let $\mu>0$
	be the dual strong concavity parameter and write $\kappa=L/\mu$.
	For an unconstrained primal variable, two-timescale gradient
	descent ascent (TS-GDA) attains
	$\mathcal{O}(\kappa^2\varepsilon^{-2})$ complexity for
	value-gradient stationarity~\cite{lin2020gda,lin2025ttgda}.
	Under a dual Polyak--\L{}ojasiewicz condition on unconstrained
	domains, Smoothed-AGDA attains
	$\mathcal{O}(\kappa\varepsilon^{-2})$ complexity for GS,
	with a terminal refinement giving the same leading order
	for value-gradient stationarity~\cite{yang2022smoothed}.
	Minimax-PPA and Catalyst-EG/OGDA attain
	$\widetilde{\mathcal{O}}(\sqrt{\kappa}\varepsilon^{-2})$
	complexity
	\rev{for value-gradient stationarity in the corresponding
		unconstrained-primal settings}
	under strong concavity~\cite{lin2020ppa,zhang2021ncsc}.
	The single-loop mirror descent ascent method MDA also reports
	$\mathcal{O}(\sqrt{\kappa}\varepsilon^{-2})$ complexity under
	its prescribed mirror geometry and mirror-gradient-mapping
	criterion~\cite{huang2021mda}.
	Lower bounds of
	$\Omega(\sqrt{\kappa}L\Delta_\phi\varepsilon^{-2})$
	are known for value-gradient stationarity in the corresponding
	unconstrained oracle models~\cite{li2021lower,zhang2021ncsc},
	where $\Delta_\phi=\phi(x_0)-\inf_{x\in X}\phi(x)$.
	
	The general NC--C setting is more difficult because the value
	function need not be smooth.
	Multi-loop methods address this difficulty by solving
	regularized auxiliary problems.
	Prox-DIAG~\cite{thekumparampil2019efficient} and
	Minimax-PPA~\cite{lin2020ppa} attain
	$\widetilde{\mathcal{O}}(\varepsilon^{-3})$ complexity for OS.
	Minimax-PPA also gives
	$\widetilde{\mathcal{O}}(\varepsilon^{-5/2})$
	complexity under its game-stationarity formulation.
	Li et al.~\cite{li2026smoothing} develop Perturbed Smoothed FOAM,
	which attains $\widetilde{\mathcal{O}}(\varepsilon^{-3})$ for OS
	and $\widetilde{\mathcal{O}}(\varepsilon^{-5/2})$ for GS.
	\rev{These methods use inner iterations to solve
		regularized auxiliary problems.}
	
	Single-loop methods instead use a fixed number of elementary
	first-order updates per iteration, but their representative
	NC--C guarantees have been weaker.
	TS-GDA has an $\mathcal{O}(\varepsilon^{-6})$ OS bound under
	an additional primal Lipschitz
	assumption~\cite{lin2020gda,lin2025ttgda}.
	AGP attains $\mathcal{O}(\varepsilon^{-4})$ under a projected
	primal--dual gradient-mapping criterion on compact convex
	domains~\cite{xu2023agp}.
	The analysis of Li et al.~\cite{li2026smoothing} establishes
	$\mathcal{O}(\varepsilon^{-4})$ complexity for Smoothed GDA
	under both OS and GS.
	Their Perturbed Smoothed GDA improves the GS bound to
	$\mathcal{O}(\varepsilon^{-3})$, while retaining an
	$\mathcal{O}(\varepsilon^{-4})$ OS bound.
	Thus, under both criteria, a complexity gap remains between
	these single-loop methods and methods that rely on auxiliary
	solvers.
	
	These comparisons require attention to the stationarity
	criterion and the reported output.
	In particular, a small projected-gradient mapping does not
	in general imply a small normal-cone residual at the same
	point.
	A conversion may require a projected output and an additional
	gradient evaluation.
	\rev{We therefore distinguish the criteria stated in the
		original results and do not implicitly identify them with
		the constrained residual used here.}
	
	Lower bounds provide a complementary perspective.
	Concurrent work establishes an $\Omega(\varepsilon^{-3})$
	OS lower bound for smooth NC--C problems:
	Wu et al.~\cite{wu2026lower} consider zero-respecting first-order
	methods, while Pan et al.~\cite{pan2026optimal} consider
	general deterministic first-order methods.
	Both results measure OS through the Moreau-envelope gradient
	of the value function.
	For GS, the lower bounds in Li et al.~\cite{li2026smoothing}
	concern specific GDA variants.
	To the best of our knowledge, no prior work establishes a
	matching $\Omega(\varepsilon^{-5/2})$ lower bound for the
	class of projected zero-respecting first-order methods.
	Consequently, the available GS upper bounds do not by
	themselves establish the optimal complexity.
	
	These observations motivate the following question:
	\begin{quote}
		\emph{Can a single-loop first-order method attain optimal
			oracle complexity under both OS and GS, with matching
			lower bounds certifying its optimality?}
	\end{quote}
	
	\paragraph{Our contributions.}
	
	We develop a single-loop framework that answers this question
	affirmatively within the projected zero-respecting first-order
	oracle class.
	\rev{Here optimality concerns the leading complexity terms
		in the stated accuracy ranges for fixed problem and
		initialization bounds, with the additive initialization
		cost stated separately.}
	The framework attains leading complexities
	$\mathcal{O}(\varepsilon^{-3})$ for OS and
	$\mathcal{O}(\varepsilon^{-5/2})$ for GS without
	the need to solve nested subproblems or introduce
	multiplicative logarithmic factors
	\rev{in these leading terms}.
	Matching lower bounds establish its optimality under both
	criteria, including what is, to our knowledge, the first
	matching problem-class lower bound for GS in the stated
	oracle model.
	The framework also extends to dual strong concavity.
	Our contributions are detailed below, where $D_Y$ denotes
	an upper bound on the dual diameter.
	
	\begin{enumerate}[leftmargin=*]
		\item
		We develop Moving-Center Extragradient with Momentum
		(MC-EGM), a projected damped extragradient framework
		for smooth NC--C minimax optimization.
		The method combines projected prediction--correction steps,
		dual momentum, and a moving proximal center, requiring
		at most three gradient evaluations per iteration
		without the need to solve nested subproblems.
		Its analysis relies on a new joint potential function that
		couples descent of a regularized envelope with tracking
		of its saddle point.
		A suitably designed warm-up mechanism sharpens the
		initialization dependence using the same elementary
		updates, without resetting the state or
		repeatedly solving inner subproblems.
		
		\item
		For GS, the warm-started method requires a
		leading first-order oracle complexity of
		$\mathcal{O}(L^{3/2}D_Y^{1/2}\Delta_\phi\varepsilon^{-5/2})$
		for returning a deterministic best-iterate whose
		first-order residual is below $\varepsilon$,
		up to an additive lower-order warm-up cost
		(Theorem~\ref{thm:end-to-end}).
		This bound matches the accuracy exponent of the
		best-known multi-loop bounds while removing their
		multiplicative logarithmic factors from the leading term.
		We establish, to our knowledge, the first
		matching problem-class GS lower bound for smooth
		NC--C minimax optimization:
		projected zero-respecting first-order methods require
		$\Omega(L^{3/2}D_Y^{1/2}\Delta_\phi\varepsilon^{-5/2})$
		oracle calls in the worst case
		(Theorem~\ref{thm:lb-gs}).
		The leading terms have the same dependence on the
		displayed problem parameters in their common accuracy range.
		
		\item
		For OS, the method requires a leading
		first-order oracle complexity of
		$\mathcal{O}(L^2D_Y\Delta_0\varepsilon^{-3})$
		for returning a uniformly selected iterate whose
		expected squared Moreau-envelope gradient norm is at most $\varepsilon^2$,
		where $\Delta_0$ depends on the initial value-function
		gap and initial gradients.
		Using the aforementioned warm-up mechanism, we can sharpen
		the leading term to
		$\mathcal{O}(L^2D_Y\Delta_\phi\varepsilon^{-3})$,
		up to an additive lower-order cost
		(Theorem~\ref{thm:end-to-end}).
		We also prove a matching lower bound of
		$\Omega(L^2D_Y\Delta_\phi\varepsilon^{-3})$
		for projected zero-respecting first-order methods,
		including randomized output rules with a fixed query budget
		(Theorem~\ref{thm:lb-os}
		and Corollary~\ref{cor:lb-randomized-outputs}).
		The lower bound therefore applies to the expected
		squared residual criterion used in the OS upper bound.
		
		\item
		Under dual strong concavity, the same framework requires
		no dual perturbation and achieves leading
		first-order oracle complexity
		$\mathcal{O}(\sqrt{\kappa}L\Delta_0\varepsilon^{-2})$
		under both OS and GS, where $\kappa$ is the dual condition number.
		With the same warm-up mechanism, the leading complexity
		becomes
		$\mathcal{O}(\sqrt{\kappa}L\Delta_\phi\varepsilon^{-2})$,
		up to an additive initialization cost
		(Theorem~\ref{thm:end-to-end}).
		We establish matching lower bounds directly for
		the Moreau-envelope OS criterion and the constrained
		GS criterion in
		Theorems~\ref{thm:lb-ncsc-os}
		and~\ref{thm:lb-ncsc-gs}.
		These results do not rely on identifying either criterion
		with value-gradient stationarity.
	\end{enumerate}

	\begin{revision}
		Table~\ref{tab:perturbed-comparison} compares representative
		NC--C complexity bounds, including their problem and
		initialization parameters.
		The MC-EGM entries report leading terms for fixed data
		with $\Delta_\phi>0$ as $\varepsilon\downarrow0$;
		the warm-up cost is additive and is not included in
		that row.
		The complete budgets and output guarantees are given
		in Theorem~\ref{thm:end-to-end}.
	\end{revision}
	
	%
	%
	
	\begin{table}[!ht]
		\caption{First-order oracle complexity bounds for NC--C
			minimax optimization, including problem and initialization
			parameters.
			\rev{Criterion qualifications and the interpretation
				of the leading terms are given below.}}
		\label{tab:perturbed-comparison}
		\small
		\setlength{\tabcolsep}{3pt}
		\renewcommand{\arraystretch}{1.65}
		\begin{tabularx}{\textwidth}{
				@{}
				>{\raggedright\arraybackslash}p{0.24\textwidth}
				*{2}{>{\centering\arraybackslash}X}
				@{}
			}
			\toprule
			Algorithm
			& \shortstack{Optimization\\stationarity}
			& \shortstack{Game\\stationarity}\\
			\midrule
			
			\multicolumn{3}{@{}l}{\textit{Single-loop algorithms}}\\
			
			TS-GDA~\cite{lin2020gda,lin2025ttgda}
			& $\displaystyle \mathcal{O}\!\left(
			\frac{L^3L_f^2D_Y^2\Delta_\phi}{\varepsilon^6}
			\right)$
			& \rev{---}\\[3pt]
			
			Smoothed GDA~\cite{li2026smoothing}
			& $\displaystyle \mathcal{O}\!\left(
			\frac{L^3D_Y^2\Delta_{\Psi_2}}{\varepsilon^4}
			\right)$
			& $\displaystyle \mathcal{O}\!\left(
			\frac{L^3D_Y^2\Delta_{\Psi_2}}{\varepsilon^4}
			\right)$\\[3pt]
			
			Perturbed GDA~\cite{xu2023agp,li2026smoothing}
			& $\displaystyle \mathcal{O}\!\left(
			\frac{L^5D_Y^4\Delta_{\Psi_1}}{\varepsilon^6}
			\right)$
			& $\displaystyle \mathcal{O}\!\left(
			\frac{L^3D_Y^2\Delta_{\Psi_1}}{\varepsilon^4}
			\right)$\\[3pt]
			
			Perturbed Smoothed GDA~\cite{li2026smoothing}
			& $\displaystyle \mathcal{O}\!\left(
			\frac{L^3D_Y^2\Delta_{\Psi_2}}{\varepsilon^4}
			\right)$
			& $\displaystyle \mathcal{O}\!\left(
			\frac{L^2D_Y\Delta_{\Psi_2}}{\varepsilon^3}
			\right)$\\[3pt]
			
			\textbf{MC-EGM (Algorithm~\ref{alg:main})}
			& {\boldmath$\displaystyle \mathcal{O}\!\left(
				\frac{L^2D_Y\rev{\Delta_0}}{\varepsilon^3}
				\right)$}
			& {\boldmath$\displaystyle \mathcal{O}\!\left(
				\frac{L^{3/2}D_Y^{1/2}\rev{\Delta_0}}
				{\varepsilon^{5/2}}
				\right)$}\\[3pt]
			
			\textbf{MC-EGM with warm-up
				\rev{(Algorithms~\ref{alg:warmup} and~\ref{alg:main})}}
			& {\boldmath$\displaystyle \mathcal{O}\!\left(
				\frac{L^2D_Y\Delta_\phi}{\varepsilon^3}
				\right)$}
			& {\boldmath$\displaystyle \mathcal{O}\!\left(
				\frac{L^{3/2}D_Y^{1/2}\Delta_\phi}
				{\varepsilon^{5/2}}
				\right)$}\\[3pt]
			
			\midrule
			\multicolumn{3}{@{}l}{\textit{Multi-loop algorithms}}\\
			
			Prox-DIAG~\cite{thekumparampil2019efficient}
			& $\displaystyle \widetilde{\mathcal{O}}\!\left(
			\frac{L^2D_Y\Delta_\phi}{\varepsilon^3}
			\right)$
			& \rev{---}\\[3pt]
			
			Minimax-PPA~\cite{lin2020ppa}
			& $\displaystyle \widetilde{\mathcal{O}}\!\left(
			\frac{L^2D_Y\Delta_\phi}{\varepsilon^3}
			\right)$
			& $\displaystyle \widetilde{\mathcal{O}}\!\left(
			\frac{L^{3/2}D_Y^{1/2}\Delta_\phi}
			{\varepsilon^{5/2}}
			\right)$\rev{$^\dagger$}\\[3pt]
			
			Perturbed Smoothed FOAM~\cite{li2026smoothing}
			& $\displaystyle \widetilde{\mathcal{O}}\!\left(
			\frac{L^2D_Y\Delta_{p_0}}{\varepsilon^3}
			\right)$
			& $\displaystyle \widetilde{\mathcal{O}}\!\left(
			\frac{L^{3/2}D_Y^{1/2}\Delta_{p_0}}
			{\varepsilon^{5/2}}
			\right)$\\[3pt]
			
			\midrule
			\textbf{Lower bounds (this paper)}
			& {\boldmath$\displaystyle \Omega\!\left(
				\frac{L^2D_Y\Delta_\phi}{\varepsilon^3}
				\right)$}
			& {\boldmath$\displaystyle \Omega\!\left(
				\frac{L^{3/2}D_Y^{1/2}\Delta_\phi}
				{\varepsilon^{5/2}}
				\right)$}\\
			\bottomrule
		\end{tabularx}
		\par\smallskip
		\begin{minipage}{\textwidth}
			\footnotesize\raggedright
			\textit{Note.}
			$L_f$ is the primal Lipschitz constant required
			by TS-GDA.
			$\Delta_{\Psi_1}$ and $\Delta_{\Psi_2}$ are the
			method-specific initial potential function gaps
			in~\cite[Definition~3.1]{li2026smoothing};
			$\Delta_{p_0}$ is the initial regularized-envelope
			gap in~\cite[Theorem~6.1]{li2026smoothing}.
			\rev{The OS/GS entries for Perturbed GDA refer to
				the analysis in~\cite{li2026smoothing}, rather than
				identifying the original AGP gradient-mapping
				criterion with GS.
				$^\dagger$The Minimax-PPA entry uses its
				game-stationarity formulation and associated
				domain assumptions.
				A dash indicates that no GS bound is asserted
				here, not the absence of such a guarantee.
				The MC-EGM rows report leading terms; its OS
				output has an expected squared Moreau-gradient
				guarantee, and its GS output is deterministic.
				The lower bounds hold in the projected
				zero-respecting model, with
				$\Delta_\phi=\Theta(\Delta)$ on the explicit
				chain instances.}
		\end{minipage}
	\end{table}
	
	\paragraph{Technical overview.}
	
We introduce quadratic regularization to obtain a strongly
convex--strongly concave saddle-point problem. The primal
regularization penalizes the squared distance from the primal
variable to a reference point, called the proximal center.
Updating this center drives optimization progress but also shifts
the saddle point that the primal--dual iterates must track.
The main challenge is to control the resulting tracking error
while updating the center at every iteration, without repeatedly
solving the regularized problem to a prescribed accuracy.

Our key observation is that the center update contributes a descent
term to the estimate for the regularized value function. This term
helps absorb the additional tracking error caused by the same
update. We capture this relationship through a joint potential
function that accounts for both the regularized value and the
tracking error. For a fixed center, projected extragradient updates
with dual momentum control the tracking error despite the imbalance
between the primal and dual curvatures. Together, this control and
the descent term explain why the center and the primal--dual
variables can be updated within a single loop. The analysis also
accounts for the normal vectors generated by projection onto the
feasible sets, so the argument applies to constrained problems.

	The lower bounds use different mechanisms for the two criteria.
	The OS construction delays successive primal coordinates
	through dual blocks and maintains a Moreau-gradient
	barrier before the last coordinate is revealed.
	For GS, the barrier must hold at every feasible dual point,
	including boundary points where normal vectors may cancel
	gradient components.
	The GS construction therefore couples an alternating
	primal--dual chain with constraints for which a small
	residual forces the relevant norm constraints to be inactive.
	\rev{The remaining sign constraints then allow the
		terminal signal to yield a uniform residual barrier.}
	
	\noindent\textbf{Organization.}
	Section~\ref{sec:preliminaries} gives the assumptions
	and stationarity criteria.
	Section~\ref{sec:algorithm} presents the algorithm and its
	end-to-end complexity guarantees.
	Section~\ref{sec:analysis} proves the potential-function
	estimates and the warm-up refinement.
	Section~\ref{sec:lower-bound} establishes matching OS and GS
	lower bounds for projected zero-respecting first-order methods.
	Section~\ref{sec:experiments} presents numerical illustrations.
	
	\noindent\textbf{Notation.}
	We use $\ip{\cdot}{\cdot}$ and $\norm{\cdot}$ for the Euclidean
	inner product and norm, respectively, with
	$\norm{(x,y)}^2=\norm{x}^2+\norm{y}^2$ on product spaces.
	For matrices, $\norm{\cdot}$ denotes the spectral norm.
	For a nonempty closed convex set $C\subseteq\R^d$, let
	$\Pi_C(w):=\argmin_{u\in C}\frac12\norm{u-w}^2$
	denote the Euclidean projection onto $C$, and let
	$N_C(u):=\{\eta\in\R^d:\ip{\eta}{a-u}\le0
	\text{ for all }a\in C\}$
	denote its normal cone at $u\in C$.
	For a nonempty set $S$, define
	$\dist(w,S):=\inf_{u\in S}\norm{w-u}$ and
	$\diam(S):=\sup_{u,v\in S}\norm{u-v}$.
	The support of a vector $u$ is
	$\supp(u):=\{i:u_i\ne0\}$.
	We write $\mathbb E$ for expectation over all randomness.

	\section{\rev{Preliminaries}}\label{sec:preliminaries}
	\label{sec:assumptions}
	
	We use the following assumptions throughout the analysis.
	
	\begin{assumption}\label{ass:main}
		The following conditions hold.
		\begin{enumerate}[label=(\roman*),leftmargin=*]
			\item
			$X\subseteq\R^{d_x}$ is nonempty, closed, and convex,
			and $Y\subseteq\R^{d_y}$ is nonempty, compact, and
			convex, with $\diam(Y)\le D_Y$ for some $D_Y>0$.
			After translating the dual coordinates if necessary,
			we assume that $0\in Y$.
			
			\item
			The function $f$ is continuously differentiable on
			an open neighborhood of $X\times Y$, and its gradient
			is $L$-Lipschitz continuous for some $L>0$:
			\begin{equation}\label{eq:smooth}
				\norm{\nabla f(x,y)-\nabla f(x',y')}
				\le L\norm{(x-x',y-y')},
				\qquad (x,y),(x',y')\in X\times Y.
			\end{equation}
			
			\item
			For some $\mu\in[0,L]$, the function $f(x,\cdot)$ is
			$\mu$-strongly concave on $Y$ for every $x\in X$:
			\begin{equation}\label{eq:strong-concavity}
				f(x,y')\le f(x,y)
				+\ip{\nabla_yf(x,y)}{y'-y}
				-\frac{\mu}{2}\norm{y'-y}^2,
				\qquad y,y'\in Y.
			\end{equation}
			
			\item
			The extended value function
			\begin{equation}\label{eq:value-function}
				\phi(x):=
				\begin{cases}
					\max_{y\in Y}f(x,y),&x\in X,\\
					+\infty,&x\notin X,
				\end{cases}
				\qquad
				\phi_{\inf}:=\inf_{x\in X}\phi(x)
			\end{equation}
			satisfies $\phi_{\inf}>-\infty$.
		\end{enumerate}
	\end{assumption}
	
	No boundedness assumption is imposed on $X$.
	The case $\mu=0$ covers general NC--C problems.
	The normalization $0\in Y$ implies $\norm{y}\le D_Y$ for
	every $y\in Y$; dual regularization is understood in these
	translated coordinates.
	
	Under Assumption~\ref{ass:main}, $\phi$ is proper, lower
	semicontinuous, and $L$-weakly convex.
	Indeed, compactness of $Y$ and continuity of $f$ imply
	continuity of $\phi$ relative to $X$.
	Moreover, $f(\cdot,y)+(L/2)\norm{\cdot}^2$ is convex on $X$
	for every $y\in Y$, and taking the pointwise maximum
	preserves convexity.
	For $\lambda>L$, define
	\begin{equation}\label{eq:original-envelope}
		\Phi_\lambda(z):=\min_{x\in X}
		\left\{\phi(x)+\frac{\lambda}{2}\norm{x-z}^2\right\},
		\qquad
		\bar x_\lambda(z):=\argmin_{x\in X}
		\left\{\phi(x)+\frac{\lambda}{2}\norm{x-z}^2\right\}.
	\end{equation}
	The minimizer is unique for every $z\in\R^{d_x}$, and
	$\Phi_\lambda$ is continuously differentiable with
	$\nabla\Phi_\lambda(z)=\lambda(z-\bar x_\lambda(z))$.
	
	\begin{definition}\label{def:optimization}\label{def:game}
		Let $\varepsilon>0$.
		\begin{enumerate}[label=(\roman*),leftmargin=*]
			\item
			A point $x\in X$ is an $\varepsilon$-optimization
			stationary ($\varepsilon$-OS) point if
			$\norm{\nabla\Phi_{2L}(x)}\le\varepsilon$.
			
			\item
			A pair $(x,y)\in X\times Y$ is an
			$\varepsilon$-game stationary ($\varepsilon$-GS)
			point if $R(x,y)\le\varepsilon$, where
			\begin{equation}\label{eq:gs-residual-definition}
				R(x,y):=\left[
				\dist^2\!\left(0,\nabla_xf(x,y)+N_X(x)\right)
				+\dist^2\!\left(0,-\nabla_yf(x,y)+N_Y(y)\right)
				\right]^{1/2}.
			\end{equation}
		\end{enumerate}
	\end{definition}
	
		For a random primal output $x^{\rm out}$, we state the OS
		guarantee as
		$\mathbb E\|\nabla\Phi_{2L}(x^{\rm out})\|^2
		\le\varepsilon^2$.
		This is an expected squared residual guarantee, not a
		pathwise assertion that every realization is an
		$\varepsilon$-OS point.
		The GS output constructed below is deterministic and
		satisfies $R(x^{\rm out},y^{\rm out})\le\varepsilon$.
	
	\section{A single-loop first-order method}\label{sec:algorithm}
	
	We develop a single-loop method that couples saddle-point
	tracking with descent of a regularized value function.
	Following the two-sided regularization framework
	of~\cite{li2026smoothing}, for a center $z\in X$ define
	\begin{align}
		G(x,y;z)
		&:=f(x,y)+\frac{r_x}{2}\norm{x-z}^2
		-\frac{r_y}{2}\norm{y}^2,
		\label{eq:regularized-subproblem}\\
		p(z)&:=\min_{x\in X}\max_{y\in Y}G(x,y;z),
		\label{eq:envelope}
	\end{align}
	where $r_x>L$, $r_y\ge0$, and $\mu+r_y>0$.
	Here $\mu\ge0$ denotes the dual strong concavity parameter,
	with $\mu=0$ in the general NC--C setting.
	The function $G(\cdot,\cdot;z)$ is
	$\mu_x$-strongly convex in $x$ and
	$\mu_y$-strongly concave in $y$, where
	$\mu_x=r_x-L$ and $\mu_y=\mu+r_y$.
	It therefore has a unique saddle point on $X\times Y$,
	denoted by $(x^\star(z),y^\star(z))$.
	Moreover,
	$\inf_{z\in X}p(z)
	=\inf_{x\in X}\max_{y\in Y}
	\{f(x,y)-r_y\norm{y}^2/2\}$.
	Thus, primal regularization preserves the optimal value of
	the dual-perturbed problem; when $r_y=0$, this value coincides
	with that of~\eqref{eq:problem}.
	
	A direct approach would approximately solve the regularized
	saddle-point problem at each center before updating the center,
	leading to nested iterations.
	Our approach instead takes a single projected
	prediction--correction step and moves the center immediately.
	The resulting method, Moving-Center Extragradient with Momentum
	(MC-EGM), combines three ingredients:
	extragradient updates to track the moving saddle point,
	damped dual momentum to accommodate unequal primal and dual
	curvatures, and projection-generated normal vectors to account
	for the constraints.
	The key design is to couple these updates so that envelope
	descent and saddle-point tracking can be controlled jointly,
	rather than enforcing a separate inner accuracy.
	Algorithm~\ref{alg:main} gives the complete iteration.
	
	\begin{algorithm}[H]
		\caption{Single-Loop Moving-Center Extragradient with Momentum
			(MC-EGM)}
		\label{alg:main}
		\begin{algorithmic}[1]
			\Require $x_0,z_0\in X$, $y_0\in Y$,
			$\xi_0\in N_X(x_0)$, $n_0\in N_Y(y_0)$,
			$v_0\in\R^{d_y}$, and an integer $T\ge1$;
			$r_x>L$, $r_y\ge0$ with $\mu+r_y>0$,
			$h>0$, $0<\alpha<1$, $\rho>0$, and $0<\beta\le1$.
			\For{$t=0,\ldots,T-1$}
			\State $\widetilde x_t
			=\Pi_X\!\left(
			x_t-h[\nabla_xG(x_t,y_t;z_t)+\xi_t]\right)$.
			\label{alg:core-start}
			\State $\widetilde y_t
			=\Pi_Y\!\left(
			y_t+h[\nabla_yG(x_t,y_t;z_t)-n_t+v_t]\right)$.
			\State $\widetilde v_t
			=\alpha v_t+\rho
			\nabla_yG(\widetilde x_t,\widetilde y_t;z_t)$.
			\State $x_{t+1}
			=\Pi_X\!\left(x_t-h
			\nabla_xG(\widetilde x_t,\widetilde y_t;z_t)\right)$.
			\State $\xi_{t+1}
			=(x_t-x_{t+1})/h
			-\nabla_xG(\widetilde x_t,\widetilde y_t;z_t)$.
			\State $y_{t+1}
			=\Pi_Y\!\left(
			y_t+h[\nabla_yG(\widetilde x_t,\widetilde y_t;z_t)
			+\widetilde v_t]\right)$.
			\State $\displaystyle n_{t+1}
			=\frac{(y_t-y_{t+1})/h
				+\nabla_yG(\widetilde x_t,\widetilde y_t;z_t)
				+\widetilde v_t}{1+\rho}$.
			\State $v_{t+1}
			=\alpha v_t+\rho[
			\nabla_yG(x_{t+1},y_{t+1};z_t)-n_{t+1}]$.
			\label{alg:core-end}
			\State $z_{t+1}=z_t+\beta(x_{t+1}-z_t)$.
			\EndFor
			\State \rev{For OS, return $z_\tau$, where $\tau$ is
				selected independently and uniformly from
				$\{0,\ldots,T-1\}$.}
			\State \rev{For GS, return $(x_{t_\star+1},y_{t_\star+1})$,
				where $t_\star$ minimizes the certificate
				$S_t$ in \eqref{eq:certificate};
				break ties by the smallest index.}
		\end{algorithmic}
	\end{algorithm}
	
	At the current center $z_t$, the predictor
	$(\widetilde x_t,\widetilde y_t)$ uses the primal residual
	$\nabla_xG(x_t,y_t;z_t)+\xi_t$ and the momentum-adjusted
	dual residual $\nabla_yG(x_t,y_t;z_t)-n_t+v_t$.
	Gradients evaluated at the predicted pair then determine
	the corrected iterate $(x_{t+1},y_{t+1})$.
	Only one prediction--correction step is taken before the
	center is updated.
	The connection between this tracking step and minimization
	of the envelope $p$ follows from
	\eqref{eq:envelope-gradient}:
	\begin{equation}\label{eq:center-inexact-descent}
		z_{t+1}
		=z_t+\beta(x_{t+1}-z_t)
		=z_t-\frac{\beta}{r_x}\nabla p(z_t)
		+\beta\bigl(x_{t+1}-x^\star(z_t)\bigr).
	\end{equation}
	Hence, the center update is an inexact gradient step on
	$p$, with stepsize $\beta/r_x$ and error
	$\beta(x_{t+1}-x^\star(z_t))$.
	If $z_t,x_{t+1}\in X$, convexity of $X$ and
	$0<\beta\le1$ ensure that $z_{t+1}\in X$.
	The analysis does not require the tracking error to fall
	below a prescribed tolerance at each iteration.
	Instead, the joint potential function in
	Section~\ref{sec:analysis} controls its evolution together
	with envelope descent.
	This joint control is what permits the center to move
	after a single primal--dual step.
	
	The dual momentum addresses the imbalance between the
	primal and dual curvatures of the regularized problem.
	Under the parameter choices in Section~\ref{sec:analysis},
	$\mu_x=L$, whereas $\mu_y$ may be substantially smaller;
	momentum is therefore introduced only in the dual variable.
	The current momentum $v_t$ enters the predictor, and a
	provisional momentum $\widetilde v_t$ enters the correction.
	The momentum retained for the next iteration is then
	formed using the constrained dual residual:
	\begin{align}
		\widetilde v_t
		&=\alpha v_t+
		\rho\nabla_yG(\widetilde x_t,\widetilde y_t;z_t),
		\notag\\
		v_{t+1}
		&=\alpha v_t+\rho[
		\nabla_yG(x_{t+1},y_{t+1};z_t)-n_{t+1}].
		\label{eq:algorithm-momentum-explanation}
	\end{align}
	Here $\alpha\in(0,1)$ damps the previous momentum, while
	$\rho>0$ determines the strength of the residual feedback.
	The distinction between the provisional and retained
	momenta is deliberate: the former participates in the
	correction projection, whereas the latter incorporates
	the normal vector generated by that projection.
	This construction draws on the treatment of unequal
	curvatures in FOAM~\cite{kovalev2022optimal} and the
	unconstrained damped extragradient mechanism
	of~\cite{zhangxu2026scsc}.
	In contrast to the procedure for solving subproblems using FOAM
	in~\cite{li2026smoothing}, these updates are incorporated
	directly into each moving-center iteration.
	
	The normal-cone feedback makes these updates compatible
	with constrained saddle-point conditions.
	In addition to enforcing feasibility, the correction
	projections generate the normal vectors used in the next
	predictor and in the momentum recursion.
	\rev{Indeed, if $u=\Pi_C(a)$, then $a-u\in N_C(u)$.}
	Applying this characterization gives
	\begin{align}
		\xi_{t+1}
		&=\frac{x_t-x_{t+1}}{h}
		-\nabla_xG(\widetilde x_t,\widetilde y_t;z_t)
		\in N_X(x_{t+1}),\notag\\
		(1+\rho)n_{t+1}
		&=\frac{y_t-y_{t+1}}{h}
		+\nabla_yG(\widetilde x_t,\widetilde y_t;z_t)
		+\widetilde v_t
		\in N_Y(y_{t+1}).
		\label{eq:algorithm-normal-inclusions}
	\end{align}
	Since $N_Y(y_{t+1})$ is a cone and $1+\rho>0$, the
	second inclusion also gives $n_{t+1}\in N_Y(y_{t+1})$.
	Thus, the required normal-cone inclusions are maintained
	without an additional optimization problem or a
	normal-cone oracle.
	For any feasible initial pair, the choices
	$\xi_0=n_0=0$ satisfy the initial inclusions, and one
	may also set $v_0=0$.
	
	This feedback is particularly relevant at boundary
	solutions, where the constrained dual residual
	$\nabla_yG-n$ may vanish although $\nabla_yG$ does not.
	Using the constrained residual in the retained momentum
	therefore aligns the recursion with the saddle-point
	optimality conditions rather than with unconstrained
	gradient vanishing.
	The primal and dual normal vectors also provide the
	stationarity certificates used in
	Section~\ref{sec:analysis}.
	
	These three mechanisms require only explicit vector
	updates, gradient evaluations, and projections.
	Each iteration uses four projections and at most three
	evaluations of $\nabla f$.
	With gradient reuse, only two new evaluations are needed
	after the initial query: the gradient at the corrected
	iterate is retained for the next predictor, while the
	center-dependent quadratic term is updated explicitly.
	The fixed-center warm-up uses the same primal--dual
	updates with the center held fixed; the moving-center
	phase then starts from the resulting state without a
	reset.

\subsection{End-to-end complexity guarantees}
\label{sec:main-guarantees}

We now state the complete oracle guarantees, including
initialization, budgets, and output rules.
Given $(x_0,y^{\rm init})\in X\times Y$, set $z_0=x_0$,
let $\Delta_\phi:=\phi(x_0)-\phi_{\inf}$, and define
\begin{equation}\label{eq:uniform-initial-bound}
	\Delta_0:=\Delta_\phi
	+D_Y\|\nabla_yf(x_0,y^{\rm init})\|
	+\frac{\|\nabla_xf(x_0,y^{\rm init})\|^2
		+2\|\nabla_yf(x_0,y^{\rm init})\|^2}{L}.
\end{equation}
The warm-up budget below uses a supplied bound
$\Delta>0$ with $\Delta_\phi\le\Delta$ and does not
depend on the initial gradient norms.
The quantity $\Delta_0$ is used in
Section~\ref{sec:initialization-bounds} to bound the
initial potential and the complexity without warm-up.

For GS, define the computable certificate
\begin{equation}\label{eq:certificate}
	S_t:=\left[
	\|\nabla_xf(x_{t+1},y_{t+1})+\xi_{t+1}\|^2
	+\|-\nabla_yf(x_{t+1},y_{t+1})
	+r_y y_{t+1}+n_{t+1}\|^2
	\right]^{1/2}.
\end{equation}
Its evaluation uses the gradient already queried at the
corrected iterate and requires no additional oracle call.
The best certificate and its associated iterate can be
maintained online.
The OS output uses a uniformly sampled index and requires
no evaluation of the Moreau-envelope gradient.

\begin{theorem}[End-to-end oracle complexity]
	\label{thm:end-to-end}
	Suppose Assumption~\ref{ass:main} holds.
	Initialize Algorithm~\ref{alg:warmup} at
	$(x_0,y^{\rm init})\in X\times Y$ with zero normal
	vectors, zero momentum, and center $z_0=x_0$.
	Let $\Delta>0$ satisfy $\Delta_\phi\le\Delta$.
	Choose $r_y$ as specified below and set
	\begin{equation}\label{eq:explicit-main-parameters}
		r_x=2L,\quad
		h=\frac{2}{15(3L+r_y)},\quad
		\rho=\frac{hL}{8},\quad
		\alpha=\frac{8}{8+h\sqrt{L(\mu+r_y)}},\quad
		\beta=\frac{1-\alpha}{50\alpha}.
	\end{equation}
	Run Algorithm~\ref{alg:warmup} for
	$K=1+\left\lceil
	\frac{1+\alpha}{1-\alpha}
	\log\!\left[\frac{28}{3h^2L^2}
	\left(1+\frac{LD_Y^2}{\Delta}\right)\right]
	\right\rceil$ steps.
	Then run Algorithm~\ref{alg:main} for $T$ steps,
	retaining the entire warm-up state and all parameters.
	Index the main phase by $t=0,\ldots,T-1$.
	For OS, return $z_\tau$ with $\tau$ sampled independently
	and uniformly from $\{0,\ldots,T-1\}$.
	For GS, return $(x_{t_\star+1},y_{t_\star+1})$ with
	$t_\star\in\argmin_{0\le t<T}S_t$, breaking ties
	deterministically.
	The following stationarity and total joint-gradient
	complexity guarantees hold.
	\begin{enumerate}[label=(\roman*),leftmargin=*]
		\item
		\textbf{NC--C OS.}
		Suppose $\mu=0$ and
		$0<\varepsilon\le\min\{4LD_Y,4\sqrt{L\Delta}\}$.
		Set $r_y=\varepsilon^2/(16LD_Y^2)$ and
		$T=\lceil\max\{1,32L\Delta/(\beta\varepsilon^2)\}\rceil$.
		Then $\mathbb E\|\nabla\Phi_{2L}(z_\tau)\|^2
		\le\varepsilon^2$, with complexity
		\begin{equation}\label{eq:full-ncc-os-complexity}
			O\!\left(
			1+\frac{LD_Y}{\varepsilon}
			\log\!\left(2+\frac{LD_Y^2}{\Delta}\right)
			+\frac{L^2D_Y\Delta}{\varepsilon^3}
			\right).
		\end{equation}
		
		\item
		\textbf{NC--C GS.}
		Suppose $\mu=0$ and
		$0<\varepsilon\le\min\{2LD_Y,2\Delta/D_Y\}$.
		Set $r_y=\varepsilon/(2D_Y)$ and
		$T=\lceil\max\{1,56L\Delta/(\beta\varepsilon^2)\}\rceil$.
		Then $R(x_{t_\star+1},y_{t_\star+1})\le\varepsilon$,
		with complexity
		\begin{equation}\label{eq:full-ncc-gs-complexity}
			O\!\left(
			1+\sqrt{\frac{LD_Y}{\varepsilon}}
			\log\!\left(2+\frac{LD_Y^2}{\Delta}\right)
			+\frac{L^{3/2}D_Y^{1/2}\Delta}{\varepsilon^{5/2}}
			\right).
		\end{equation}
		
		\item
		\textbf{NC--SC OS and GS.}
		Suppose $\mu>0$ and $\varepsilon>0$.
		Set $r_y=0$ and
		$T=\lceil\max\{1,14L\Delta/(\beta\varepsilon^2)\}\rceil$.
		The respective outputs satisfy
		$\mathbb E\|\nabla\Phi_{2L}(z_\tau)\|^2\le\varepsilon^2$
		and $R(x_{t_\star+1},y_{t_\star+1})\le\varepsilon$,
		with complexity
		\begin{equation}\label{eq:full-ncsc-complexity}
			O\!\left(
			1+\sqrt{\kappa}
			\log\!\left(2+\frac{LD_Y^2}{\Delta}\right)
			+\frac{\sqrt{\kappa}\,L\Delta}{\varepsilon^2}
			\right),
			\qquad \kappa=\frac L\mu.
		\end{equation}
	\end{enumerate}
	All implied constants are universal.
	With gradient reuse, each case requires at most
	$1+2(K+T)$ joint gradient evaluations.
\end{theorem}

\begin{proof}
	In each case, the accuracy conditions imply
	$r_yD_Y^2\le\Delta$ and $0<\mu+r_y\le L$.
	The parameters in \eqref{eq:explicit-main-parameters}
	therefore satisfy \eqref{eq:parameters}.
	Proposition~\ref{prop:warmup-potential function} gives
	\[
	V(w_K^{\mathrm{wu}};z_0)\le2\Delta
	\]
	for the prescribed warm-up budget.
	
	Use $w_K^{\mathrm{wu}}$ as the initial state of the
	main phase.
	Substituting the potential bound $2\Delta$ into
	the budgets of
	Theorems~\ref{thm:ncc-optimization-complexity}
	and~\ref{thm:ncsc-complexity} gives the stated
	values of $T$.
	Corollary~\ref{cor:implementable-outputs} then yields
	the expected squared OS guarantee and the
	deterministic GS guarantee.
	
	For the explicit parameters,
	\[
	\frac1{30}\le hL\le\frac2{45},\qquad
	\beta=\frac{h\sqrt{L(\mu+r_y)}}{400},
	\qquad
	\frac{1+\alpha}{1-\alpha}
	=1+\frac{16}{h\sqrt{L(\mu+r_y)}}.
	\]
	Consequently,
	\[
	K=O\!\left(
	1+\sqrt{\frac{L}{\mu+r_y}}
	\log\!\left(2+\frac{LD_Y^2}{\Delta}\right)
	\right),
	\qquad
	T=O\!\left(
	1+\sqrt{\frac{L}{\mu+r_y}}
	\frac{L\Delta}{\varepsilon^2}
	\right).
	\]
	Substituting the three choices of $r_y$ into
	$K+T$ proves
	\eqref{eq:full-ncc-os-complexity}--%
	\eqref{eq:full-ncsc-complexity}.
	
	After the initial query, each update requires
	two new joint gradient evaluations.
	The final warm-up gradient is reused in the
	first moving-center step, and neither output rule
	requires an additional query.
	The total oracle count is therefore at most
	$1+2(K+T)$.
\end{proof}

\begin{remark}[Initialization dependence and leading-order comparison]
	\label{rem:scope-main-optimality}
	The prescribed budgets use a supplied positive bound
	$\Delta\ge\Delta_\phi$, but do not require evaluation
	of the potential, a regularized saddle point, or the
	value-function infimum.
	They are independent of the initial gradient norms.
	
	When $\Delta_\phi>0$, the instance-dependent choice
	$\Delta=\Delta_\phi$ gives the leading terms in
	Table~\ref{tab:perturbed-comparison}.
	For fixed problem parameters and fixed $\Delta>0$,
	the additive warm-up cost is lower order as
	$\varepsilon\downarrow0$.
	The lower bounds match the leading terms in their
	common accuracy ranges.
	If $\Delta_\phi=0$, any positive supplied $\Delta$
	can still be used, subject to the stated accuracy
	conditions.
\end{remark}

\section{Complexity analysis}\label{sec:analysis}

The analysis separates two effects of Algorithm~\ref{alg:main}:
contraction of the primal--dual tracking error at a fixed center
and the perturbation caused by moving the center.
We control these effects through a joint potential function
combining the regularized envelope with a residual-based
tracking energy.
Its cross terms account for dual momentum and normal-cone
feedback, without requiring an inner stopping criterion.

\subsection{\rev{Regularized saddle-point geometry}}
\label{sec:regularized-geometry}

Recall $G$ and $p$ from
\eqref{eq:regularized-subproblem}--\eqref{eq:envelope}.
Throughout the next four lemmas, let $r_x>L$, $r_y\ge0$,
and $\mu+r_y>0$, and set
$\mu_x:=r_x-L$, $\mu_y:=\mu+r_y$, and $L_G:=L+r_x+r_y$.
Then $G(\cdot,\cdot;z)$ is $\mu_x$-strongly convex in $x$,
$\mu_y$-strongly concave in $y$, and has an
$L_G$-Lipschitz continuous gradient.
The constant $L_G$ is a convenient upper bound and need
not be the smallest Lipschitz constant.

The definitions of $G$ and $p$ extend to every
$z\in\R^{d_x}$.
For each such center, $\max_{y\in Y}G(\cdot,y;z)$ is
strongly convex and coercive on $X$, and therefore has a
unique minimizer $x^\star(z)$.
Strong concavity gives a unique maximizer $y^\star(z)$
at $x^\star(z)$.
Danskin's theorem and the first-order optimality conditions
then show that $(x^\star(z),y^\star(z))$ is the unique
saddle point.
The following inequalities control both fixed-center
tracking and changes of center.

\begin{lemma}\label{lem:G-geometry}
	For every $z\in\R^{d_x}$,
	$(x,y),(x',y')\in X\times Y$,
	$\xi\in N_X(x)$, and $n'\in N_Y(y')$, we have
	\begin{align}
		&G(x',y';z)-G(x,y;z)
		-\langle\nabla_yG(x',y';z)-n',y'-y\rangle\notag\\
		&\qquad\ge
		\langle\nabla_xG(x,y;z)+\xi,x'-x\rangle
		+\frac{\mu_x}{2}\|x'-x\|^2
		+\frac{\mu_y}{2}\|y'-y\|^2.
		\label{eq:A-function-change}
	\end{align}
\end{lemma}

\begin{proof}
	Strong convexity in $x$ and strong concavity in $y$ give
	\begin{align*}
		G(x',y;z)-G(x,y;z)
		&\ge\langle\nabla_xG(x,y;z),x'-x\rangle
		+\frac{\mu_x}{2}\|x'-x\|^2,\\
		G(x',y';z)-G(x',y;z)
		&\ge\langle\nabla_yG(x',y';z),y'-y\rangle
		+\frac{\mu_y}{2}\|y'-y\|^2.
	\end{align*}
	Adding these inequalities and using
	$\langle\xi,x'-x\rangle\le0$ and
	$\langle n',y'-y\rangle\ge0$
	proves \eqref{eq:A-function-change}.
\end{proof}

\begin{lemma}\label{lem:saddle-properties}
	For every $z\in\R^{d_x}$,
	$(x,y)\in X\times Y$, $\xi\in N_X(x)$, and $n\in N_Y(y)$,
	we have
	\begin{align}
		G(x,y;z)-p(z)
		&\ge\langle\nabla_yG(x,y;z)-n,y-y^\star(z)\rangle
		+\frac{\mu_x}{2}\|x-x^\star(z)\|^2
		+\frac{\mu_y}{2}\|y-y^\star(z)\|^2,
		\label{eq:saddle-lower}\\
		G(x,y;z)-p(z)
		&\le\langle\nabla_xG(x,y;z)+\xi,x-x^\star(z)\rangle
		-\frac{\mu_x}{2}\|x-x^\star(z)\|^2
		-\frac{\mu_y}{2}\|y-y^\star(z)\|^2.
		\label{eq:saddle-upper}
	\end{align}
	Moreover, for $z,z'\in\R^{d_x}$,
	\begin{equation}\label{eq:saddle-center-stability}
		\mu_x\|x^\star(z')-x^\star(z)\|^2
		+\mu_y\|y^\star(z')-y^\star(z)\|^2
		\le r_x\langle z'-z,x^\star(z')-x^\star(z)\rangle.
	\end{equation}
	Consequently,
	$\|x^\star(z')-x^\star(z)\|\le(r_x/\mu_x)\|z'-z\|$
	and
	$\|y^\star(z')-y^\star(z)\|
	\le r_x\|z'-z\|/(2\sqrt{\mu_x\mu_y})$.
\end{lemma}

\begin{proof}
	Saddle-point optimality gives normal vectors
	$\xi^\star(z)\in N_X(x^\star(z))$ and
	$n^\star(z)\in N_Y(y^\star(z))$ such that
	$\nabla_xG(x^\star(z),y^\star(z);z)+\xi^\star(z)=0$
	and
	$\nabla_yG(x^\star(z),y^\star(z);z)-n^\star(z)=0$.
	Applying Lemma~\ref{lem:G-geometry} from
	$(x^\star(z),y^\star(z))$ to $(x,y)$ and in the reverse
	order proves \eqref{eq:saddle-lower} and
	\eqref{eq:saddle-upper}, respectively.
	
	We next apply \eqref{eq:saddle-lower} at center $z$ to
	$(x^\star(z'),y^\star(z'))$ with $n=n^\star(z')$,
	and repeat the argument with $z$ and $z'$ interchanged.
	The dual residuals vanish because $\nabla_yG$ is
	independent of the center.
	Adding the resulting inequalities gives
	\begin{align*}
		&\mu_x\|x^\star(z')-x^\star(z)\|^2
		+\mu_y\|y^\star(z')-y^\star(z)\|^2\\
		&\quad\le
		G(x^\star(z'),y^\star(z');z)-p(z)
		+G(x^\star(z),y^\star(z);z')-p(z')\\
		&\quad=
		\frac{r_x}{2}\bigl(
		\|x^\star(z')-z\|^2-\|x^\star(z')-z'\|^2
		+\|x^\star(z)-z'\|^2-\|x^\star(z)-z\|^2\bigr)\\
		&\quad=
		r_x\langle z'-z,x^\star(z')-x^\star(z)\rangle,
	\end{align*}
	where we used $p(z)=G(x^\star(z),y^\star(z);z)$.
	This proves \eqref{eq:saddle-center-stability}.
	The primal bound follows from Cauchy--Schwarz.
	For the dual bound, Young's inequality gives
	$r_x\langle z'-z,x^\star(z')-x^\star(z)\rangle
	\le\mu_x\|x^\star(z')-x^\star(z)\|^2
	+r_x^2\|z'-z\|^2/(4\mu_x)$.
	Substitution into \eqref{eq:saddle-center-stability}
	completes the proof.
\end{proof}

\begin{lemma}\label{lem:envelope-properties}
	The function $p$ is continuously differentiable on
	$\R^{d_x}$, with
	\begin{equation}\label{eq:envelope-gradient}
		\nabla p(z)=r_x(z-x^\star(z)).
	\end{equation}
	Moreover, for $z,z'\in\R^{d_x}$ and $x\in X$,
	\begin{equation}\label{eq:A-combined-center-change}
		p(z')-p(z)
		+\frac{r_x}{2}\bigl(\|x-z\|^2-\|x-z'\|^2\bigr)
		\le r_x\langle x-x^\star(z),z'-z\rangle.
	\end{equation}
\end{lemma}

\begin{proof}
	Define
	$\phi_{r_y}(x):=\max_{y\in Y}
	\{f(x,y)-(r_y/2)\|y\|^2\}$
	on $X$ and extend it by $+\infty$ outside $X$.
	This function is proper, lower semicontinuous,
	$L$-weakly convex, and bounded below by
	$\phi_{\inf}-r_yD_Y^2/2$.
	Since $r_x>L$, its Moreau envelope $p$ is continuously
	differentiable, with unique proximal point $x^\star(z)$.
	The envelope-gradient formula therefore gives
	\eqref{eq:envelope-gradient}.
	Evaluating the minimization defining $p(z')$ at
	$x^\star(z)$ yields
	\begin{align*}
		p(z')-p(z)
		&\le\max_{y\in Y}G(x^\star(z),y;z')
		-\max_{y\in Y}G(x^\star(z),y;z)\\
		&=\frac{r_x}{2}\bigl(
		\|x^\star(z)-z'\|^2-\|x^\star(z)-z\|^2\bigr).
	\end{align*}
	Adding $r_x(\|x-z\|^2-\|x-z'\|^2)/2$ and expanding
	the squared norms proves
	\eqref{eq:A-combined-center-change}.
\end{proof}

The next lemma bounds the bias introduced by dual regularization.
When $r_x=2L$, it converts envelope-gradient estimates for $p$
into OS guarantees for the original problem.

\begin{lemma}\label{lem:regularized-envelope-bias}
	For every $z\in\R^{d_x}$, we have
	\begin{equation}\label{eq:regularization-bias}
		0\le\Phi_{r_x}(z)-p(z)\le\frac{r_yD_Y^2}{2},
	\end{equation}
	and
	\begin{equation}\label{eq:envelope-bias}
		\|\nabla\Phi_{r_x}(z)-\nabla p(z)\|
		=r_x\|x^\star(z)-\bar x_{r_x}(z)\|
		\le r_xD_Y\sqrt{\frac{r_y}{\mu_x}}.
	\end{equation}
	In particular, $p=\Phi_{r_x}$ when $r_y=0$.
\end{lemma}

\begin{proof}
	Since $\|y\|\le D_Y$ on $Y$, the function $\phi_{r_y}$
	defined above satisfies
	$0\le\phi(x)-\phi_{r_y}(x)\le r_yD_Y^2/2$ for $x\in X$.
	Adding the same proximal quadratic and minimizing gives
	\eqref{eq:regularization-bias}.
	Fix $z$, and set
	$F_z(x):=\phi(x)+(r_x/2)\|x-z\|^2$ and
	$H_z(x):=\phi_{r_y}(x)+(r_x/2)\|x-z\|^2$.
	Strong convexity of $F_z$ and optimality of
	$\bar x_{r_x}(z)$ and $x^\star(z)$ give
	\begin{align*}
		\frac{\mu_x}{2}
		\|x^\star(z)-\bar x_{r_x}(z)\|^2
		&\le F_z(x^\star(z))-F_z(\bar x_{r_x}(z))\\
		&\le H_z(x^\star(z))-H_z(\bar x_{r_x}(z))
		+\frac{r_yD_Y^2}{2}
		\le\frac{r_yD_Y^2}{2}.
	\end{align*}
	The two envelope-gradient formulas then prove
	\eqref{eq:envelope-bias}.
	When $r_y=0$, the defining minimization problems for
	$p$ and $\Phi_{r_x}$ coincide.
\end{proof}

\subsection{\rev{Tracking energy and fixed-center contraction}}
\label{sec:tracking-analysis}

We construct a potential function that combines envelope
descent with saddle-point tracking.
Write $w_t=(x_t,y_t,\xi_t,n_t,v_t)$.
A state $w=(x,y,\xi,n,v)$ is called \emph{admissible} if
$x\in X$, $y\in Y$, $\xi\in N_X(x)$,
$n\in N_Y(y)$, and $v\in\R^{d_y}$.
The projection updates and
\eqref{eq:algorithm-normal-inclusions} preserve admissibility.
For an admissible state $w$ and a center $z\in X$, define
the predictor directions
\begin{equation}\label{eq:analysis-coefficients}
	d^x(w;z):=\nabla_xG(x,y;z)+\xi,
	\qquad
	d^y(w;z):=\nabla_yG(x,y;z)-n+v.
\end{equation}
Throughout \rev{the tracking and descent analysis}, assume
\begin{align}
	&r_x=2L,\qquad
	0<\mu+r_y\le L,\qquad
	0<h\le\frac{2}{15(3L+r_y)},\qquad
	\rho=\frac{hL}{8},\notag\\
	&\frac{4}{4+h\sqrt{L(\mu+r_y)}}\le\alpha
	\le\frac{8}{8+h\sqrt{L(\mu+r_y)}},
	\qquad
	0<\beta\le\frac{1-\alpha}{50\alpha}.
	\label{eq:parameters}
\end{align}
These choices give $\mu_x=L$, $\mu_y=\mu+r_y$,
and $L_G=3L+r_y$.
For NC--C problems, the dual regularization parameter
$r_y>0$ will be chosen according to the target accuracy;
for NC--SC problems, we take $r_y=0$.
To simplify the potential function and its analysis, set
$\gamma:=(1-\alpha)/(2\alpha h)$ and
$\omega:=(1-\alpha+2\rho)/(2\alpha h)
=\gamma+L/(8\alpha)$.
Condition~\eqref{eq:parameters} implies
$\sqrt{L\mu_y}/16\le\gamma
\le\sqrt{L\mu_y}/8\le L/8$
and $L/8\le\omega\le13L/50$.
We will also use the consequences
$hL_G\le2/15$, $hL\le2/45$, $h\gamma\le1/180$,
$\beta\le h\gamma/25$, and $4\gamma^2\le\omega\mu_y$.

The joint potential function is
\begin{equation}\label{eq:lyapunov}
	V(w;z):=p(z)-p_{\inf}+\mathscr E(w;z),
\end{equation}
where
$p_{\inf}:=\inf_{z\in X}p(z)
\ge\phi_{\inf}-r_yD_Y^2/2>-\infty$,
and the tracking energy is defined by
\begin{align}
	\mathscr E(w;z)
	:={}&-\frac{2\omega}{L}\bigl(G(x,y;z)-p(z)\bigr)
	+\frac1L\|d^x(w;z)\|^2\notag\\
	&+\frac1L\|d^y(w;z)+\gamma(y-y^\star(z))\|^2
	-\frac{2\gamma}{L}
	\langle\nabla_yG(x,y;z)-n,y-y^\star(z)\rangle\notag\\
	&-\frac{\gamma^2}{L}\|y-y^\star(z)\|^2
	+\frac{\gamma}{8L^2}\|v\|^2.
	\label{eq:tracking-energy}
\end{align}
The first term of $V$ measures envelope suboptimality,
while $\mathscr E$ couples the saddle-value discrepancy,
the predictor residuals, and the dual momentum.
The cross term and the signed quadratic term account
for the interaction between the dual residual and the
dual tracking error.
Although $\mathscr E$ is not manifestly nonnegative,
the saddle-point geometry and the parameter conditions
above ensure its nonnegativity, as established below.

\begin{lemma}\label{lem:tracking-nonnegative}
	Under Assumption~\ref{ass:main} and
	\eqref{eq:parameters}, every admissible state $w$ and
	center $z\in X$ satisfy
	\begin{align}
		\mathscr E(w;z)
		\ge{}&
		\frac{\alpha hL-(1-\alpha+\rho)}
		{\alpha hL^2}\|d^x(w;z)\|^2
		+\frac1L\left\|
		d^y(w;z)+\frac{1-\alpha}{2\alpha h}
		(y-y^\star(z))\right\|^2\notag\\
		&+\frac{1-\alpha}{2\alpha h}
		\|x-x^\star(z)\|^2
		+\frac{1-\alpha}{16\alpha hL^2}\|v\|^2\notag\\
		&+\frac{
			2\alpha h[3(1-\alpha)+2\rho](\mu+r_y)
			-(1-\alpha)^2}{4\alpha^2h^2L}
		\|y-y^\star(z)\|^2
		\ge0.
		\label{eq:energy-nonnegative}
	\end{align}
	Consequently, $V(w;z)\ge0$.
\end{lemma}

\begin{proof}
	Set $a=G(x,y;z)-p(z)$,
	$e_x=x-x^\star(z)$, and $e_y=y-y^\star(z)$.
	Since $\mu_x=L$, \eqref{eq:saddle-upper} and
	Young's inequality give
	$a\le\|d^x(w;z)\|^2/(2L)-(\mu_y/2)\|e_y\|^2$.
	Meanwhile, \eqref{eq:saddle-lower} gives
	$\langle\nabla_yG(x,y;z)-n,e_y\rangle
	\le a-(L/2)\|e_x\|^2-(\mu_y/2)\|e_y\|^2$.
	Substituting the latter bound into
	\eqref{eq:tracking-energy} and then applying the
	upper bound on $a$ to $-2(\omega+\gamma)a/L$
	proves the first inequality in
	\eqref{eq:energy-nonnegative}.
	The coefficients of $\|d^x(w;z)\|^2$ and $\|e_y\|^2$
	equal $(L-\omega-\gamma)/L^2$ and
	$((\omega+2\gamma)\mu_y-\gamma^2)/L$, respectively.
	Both are nonnegative by the consequences of
	\eqref{eq:parameters}, and the remaining coefficients
	are positive.
	Thus, $\mathscr E(w;z)\ge0$.
	Since $z\in X$, we also have $p(z)\ge p_{\inf}$,
	which proves $V(w;z)\ge0$.
\end{proof}

We next establish contraction of the tracking energy at a
fixed center by analyzing lines
\ref{alg:core-start}--\ref{alg:core-end} of
Algorithm~\ref{alg:main}.
Define the corrected directions
\begin{align}
	\widetilde d_t^x
	&:=-\frac{x_{t+1}-x_t}{h}
	=\nabla_xG(\widetilde x_t,\widetilde y_t;z_t)
	+\xi_{t+1},\notag\\
	\widetilde d_t^y
	&:=\frac{y_{t+1}-y_t}{h}
	=\nabla_yG(\widetilde x_t,\widetilde y_t;z_t)
	+\widetilde v_t-(1+\rho)n_{t+1}\notag\\
	&=(1+\rho)[
	\nabla_yG(\widetilde x_t,\widetilde y_t;z_t)-n_{t+1}]
	+\alpha v_t.
	\label{eq:corrected-directions}
\end{align}
These identities follow from the correction and momentum
updates.
The next two lemmas control the inner products of these
directions with residual differences and their deviations
from the residuals at the corrected iterate.

\begin{lemma}\label{lem:fixed-center-geometry}
	Under Assumption~\ref{ass:main} and
	\eqref{eq:parameters}, one fixed-center step from an
	admissible state $w_t$ at $z_t\in X$ satisfies
	\begin{align}
		&\langle\widetilde d_t^x,
		d^x(w_{t+1};z_t)-d^x(w_t;z_t)\rangle
		+\langle\widetilde d_t^y,
		d^y(w_{t+1};z_t)-d^y(w_t;z_t)-v_{t+1}+v_t\rangle
	\le
		-hL\|\widetilde d_t^x\|^2
		-h\mu_y\|\widetilde d_t^y\|^2.
		\label{eq:correction-geometry}
	\end{align}
\end{lemma}

\begin{proof}
	Applying Lemma~\ref{lem:G-geometry} in both orders to
	$(x_t,y_t)$ and $(x_{t+1},y_{t+1})$, with their
	respective normal vectors, and adding the inequalities
	gives
	\begin{align*}
		L\|x_{t+1}-x_t\|^2+\mu_y\|y_{t+1}-y_t\|^2
		&\le
		\langle d^x(w_{t+1};z_t)-d^x(w_t;z_t),
		x_{t+1}-x_t\rangle\\
		&\quad-\langle
		d^y(w_{t+1};z_t)-d^y(w_t;z_t)-v_{t+1}+v_t,
		y_{t+1}-y_t\rangle,
	\end{align*}
	where we used $\mu_x=L$.
	Substituting \eqref{eq:corrected-directions} and
	dividing by $h$ proves \eqref{eq:correction-geometry}.
\end{proof}

\begin{lemma}\label{lem:fixed-center-mismatch}
	Under Assumption~\ref{ass:main} and
	\eqref{eq:parameters}, one fixed-center step from an
	admissible state $w_t$ at $z_t\in X$ satisfies
	\begin{align}
		&\|\widetilde d_t^x-d^x(w_{t+1};z_t)\|^2
		+\|\widetilde d_t^y-d^y(w_{t+1};z_t)\|^2\notag\\
		&\qquad\le\frac1{36}\bigl(
		\|d^x(w_{t+1};z_t)-d^x(w_t;z_t)\|^2
		+\|d^y(w_{t+1};z_t)-d^y(w_t;z_t)\|^2\bigr).
		\label{eq:mismatch-bound}
	\end{align}
\end{lemma}

\begin{proof}
	The correction and momentum updates give
	\begin{align}
		\widetilde d_t^x-d^x(w_{t+1};z_t)
		&=\nabla_xG(\widetilde x_t,\widetilde y_t;z_t)
		-\nabla_xG(x_{t+1},y_{t+1};z_t),\notag\\
		\widetilde d_t^y-d^y(w_{t+1};z_t)
		&=(1+\rho)[
		\nabla_yG(\widetilde x_t,\widetilde y_t;z_t)
		-\nabla_yG(x_{t+1},y_{t+1};z_t)].
		\label{eq:correction-identities}
	\end{align}
	Let $M$ denote the square root of the left-hand side
	of \eqref{eq:mismatch-bound}, and set
	$D^2:=\sum_{i\in\{x,y\}}
	\|d^i(w_{t+1};z_t)-d^i(w_t;z_t)\|^2$.
	Lipschitz continuity gives
	$M\le L_G(1+\rho)
	\|(\widetilde x_t-x_{t+1},\widetilde y_t-y_{t+1})\|$.
	Projection nonexpansiveness, feasibility of the corrected
	iterates, and \eqref{eq:corrected-directions} imply
	\begin{align*}
		\|\widetilde x_t-x_{t+1}\|
		&\le\|x_t-hd^x(w_t;z_t)-x_{t+1}\|
		=h\|\widetilde d_t^x-d^x(w_t;z_t)\|,\\
		\|\widetilde y_t-y_{t+1}\|
		&\le\|y_t+hd^y(w_t;z_t)-y_{t+1}\|
		=h\|\widetilde d_t^y-d^y(w_t;z_t)\|.
	\end{align*}
	Combining these bounds with the triangle inequality
	yields $M\le q(M+D)$, where $q:=hL_G(1+\rho)$.
	Since $\rho=hL/8\le1/180$,
	$q\le(2/15)(1+1/180)=181/1350<1/7$.
	Thus, $M\le qD/(1-q)\le D/6$, and squaring proves
	\eqref{eq:mismatch-bound}.
\end{proof}

These estimates allow the direction errors to be absorbed
into the fixed-center energy decrease.

\begin{lemma}\label{lem:tracking-contraction}
	\label{lem:fixed-center-descent}
	Suppose Assumption~\ref{ass:main} and
	\eqref{eq:parameters} hold.
	Let $w_{t+1}$ be obtained from an admissible state $w_t$
	by executing lines
	\ref{alg:core-start}--\ref{alg:core-end} of
	Algorithm~\ref{alg:main} at a fixed center $z_t\in X$.
	Then
	\begin{align}
		(1+h\gamma)\mathscr E(w_{t+1};z_t)
		-\mathscr E(w_t;z_t)
		&\le-\frac{3h}{10}\|d^x(w_{t+1};z_t)\|^2
		-\frac{2h\gamma}{19L}\|d^y(w_{t+1};z_t)\|^2
		\notag\\
		&\quad-h\gamma\omega
		\|x_{t+1}-x^\star(z_t)\|^2
		-\frac{h\gamma\omega}{8L^2}\|v_{t+1}\|^2.
		\label{eq:fixed-center-energy}
	\end{align}
	In particular,
	\begin{equation}\label{eq:tracking-one-step-contraction}
		\mathscr E(w_{t+1};z_t)
		\le\frac{2\alpha}{1+\alpha}\mathscr E(w_t;z_t).
	\end{equation}
\end{lemma}

\begin{proof}
	Fix $z=z_t$ and abbreviate
	$a_i=d^i(w_t;z)$,
	$b_i=d^i(w_{t+1};z)$,
	$q_i=\widetilde d_t^i$,
	$\delta_i=b_i-a_i$, and $e_i=q_i-b_i$
	for $i\in\{x,y\}$.
	Also set $u=v_{t+1}$, $v=v_t$,
	$X_+=x_{t+1}-x^\star(z)$,
	$Y_+=y_{t+1}-y^\star(z)$,
	$G_t=G(x_t,y_t;z)$, and
	$G_+=G(x_{t+1},y_{t+1};z)$.
	All residuals below are evaluated at the same center $z$.
	
	\rev{\emph{Step 1: exact energy balance.}}
	Expanding \eqref{eq:tracking-energy} gives
	\begin{align}
		\mathscr E(w_t;z_t)
		={}&\frac2L\left[
		-\omega(G(x_t,y_t;z_t)-p(z_t))
		+\frac12\|d^x(w_t;z_t)\|^2
		+\frac12\|d^y(w_t;z_t)\|^2\right]\notag\\
		&+\frac{2\gamma}{L}
		\langle v_t,y_t-y^\star(z_t)\rangle
		+\frac{\gamma}{8L^2}\|v_t\|^2.
		\label{eq:tracking-energy-expanded}
	\end{align}
	The momentum update and the definitions of $\gamma$
	and $\omega$ imply
	\begin{equation}\label{eq:feedback-identity}
		u-v+h\gamma u
		=h\omega(b_y-u)-h\gamma b_y.
	\end{equation}
	\rev{Indeed, $u=\alpha v+\rho(b_y-u)$ gives
		$v=(1+\rho)u/\alpha-\rho b_y/\alpha$.
		The identities
		$h(\omega+\gamma)=(1-\alpha+\rho)/\alpha$
		and $h(\omega-\gamma)=\rho/\alpha$
		then yield \eqref{eq:feedback-identity}.}
	
	\rev{For clarity, the momentum cross term changes according
		to
		$(1+h\gamma)\langle u,Y_+\rangle
		-\langle v,Y_+-hq_y\rangle
		=\langle u-v+h\gamma u,Y_+\rangle
		+h\langle v,q_y\rangle$.
		Substitute \eqref{eq:feedback-identity} into the first
		term and write $v=u-(u-v)$ in the second.}
	Using this identity, $y_{t+1}-y_t=hq_y$, and
	$\frac12(\|b_i\|^2-\|a_i\|^2)
	=\langle q_i,\delta_i\rangle
	-\langle e_i,\delta_i\rangle-\frac12\|\delta_i\|^2$,
	we obtain the exact energy balance
	\begin{align}
		\mathcal D
		:={}&(1+h\gamma)\mathscr E(w_{t+1};z)
		-\mathscr E(w_t;z)\notag\\
		={}&-\frac{2\omega}{L}
		[G_+-G_t-h\langle b_y-u,q_y\rangle]
		-\frac{2h\gamma\omega}{L}
		[G_+-p(z)-\langle b_y-u,Y_+\rangle]\notag\\
		&+\frac2L\langle q_x,\delta_x\rangle
		+\frac2L\langle q_y,\delta_y-u+v\rangle
		+\frac{h\gamma}{L}\|b_x\|^2
		-\frac{h\gamma}{L}\|b_y\|^2\notag\\
		&-\frac{2h\gamma}{L}
		\langle b_y,e_y+\gamma Y_+\rangle
		-\frac2L\sum_{i\in\{x,y\}}\langle e_i,\delta_i\rangle
		-\frac1L\sum_{i\in\{x,y\}}\|\delta_i\|^2\notag\\
		&-\frac{2h\gamma}{L}\langle u-v,q_y\rangle
		+\frac{\gamma}{8L^2}
		[(1+h\gamma)\|u\|^2-\|v\|^2].
		\label{eq:raw-energy-balance}
	\end{align}
	
	\rev{\emph{Step 2: saddle geometry and direction terms.}}
	Lemma~\ref{lem:G-geometry}, applied from the current
	iterate to the corrected iterate, and completion of the
	square give
	\begin{align}
		&-\omega[G_+-G_t-h\langle b_y-u,q_y\rangle]
		-\frac12\|\delta_x\|^2\notag\\
		\le &
		h\omega\|b_x\|^2+h\omega\langle b_x,e_x\rangle
		-\frac12\|\delta_x+h\omega q_x\|^2
		-\frac{h^2\omega}{2}
		[(L-\omega)\|q_x\|^2+\mu_y\|q_y\|^2].
		\label{eq:primal-square-completion}
	\end{align}
	The second bracket in \eqref{eq:raw-energy-balance}
	is bounded below by \eqref{eq:saddle-lower}, while the
	two corrected-direction products are bounded by
	Lemma~\ref{lem:fixed-center-geometry}.
	The remaining cross terms satisfy
	\begin{align}
		h(\omega+\gamma/2)\|b_x\|^2
		-hL\|q_x\|^2+h\omega\langle b_x,e_x\rangle
	&\le
		-h(L/2-\omega-\gamma/2)\|b_x\|^2
		+hL\|e_x\|^2,
		\label{eq:primal-absorption}\\
		-\frac{h\gamma}{2}\|b_y\|^2
		-h\gamma\langle b_y,e_y+\gamma Y_+\rangle
		-\frac{h\gamma\omega\mu_y}{2}\|Y_+\|^2
	&\le
		-\frac{h\gamma}{4}\|b_y\|^2
		+2h\gamma\|e_y\|^2
		-\frac{h\gamma}{2}
		(\omega\mu_y-4\gamma^2)\|Y_+\|^2.
		\label{eq:dual-absorption}
	\end{align}
	Indeed, \eqref{eq:primal-absorption} follows by expanding
	$q_x=b_x+e_x$ and applying Young's inequality with
	$(2L-\omega)^2\le4L^2$.
	For \eqref{eq:dual-absorption}, Young's inequality gives
	$|\langle b_y,e_y+\gamma Y_+\rangle|
	\le\|b_y\|^2/4+\|e_y+\gamma Y_+\|^2$, and
	$\|e_y+\gamma Y_+\|^2
	\le2\|e_y\|^2+2\gamma^2\|Y_+\|^2$.
	Since $\omega<L$ and $4\gamma^2\le\omega\mu_y$,
	combining these estimates and discarding the nonpositive
	terms involving $q_x$, $q_y$, and $Y_+$ yields
	\begin{align}
		\mathcal D
		\le{}&
		-\frac{2h}{L}(L/2-\omega-\gamma/2)\|b_x\|^2
		-\frac{h\gamma}{2L}\|b_y\|^2
		+2h\|e_x\|^2+\frac{4h\gamma}{L}\|e_y\|^2
		\notag\\
		&-h\gamma\omega\|X_+\|^2
		-\frac1L\|\delta_x+h\omega q_x\|^2
		-\frac1L\|\delta_y\|^2
		-\frac2L\sum_{i\in\{x,y\}}
		\langle e_i,\delta_i\rangle\notag\\
		&-\frac{2h\gamma}{L}\langle u-v,q_y\rangle
		+\frac{\gamma}{8L^2}
		[(1+h\gamma)\|u\|^2-\|v\|^2].
		\label{eq:energy-after-direction-bounds}
	\end{align}
	
	\rev{\emph{Step 3: shifted square and momentum feedback.}}
	Young's inequality gives
	\begin{equation}\label{eq:shifted-square-bound}
		-\frac12\|\delta_x+h\omega q_x\|^2
		\le-\frac38\|\delta_x\|^2
		+\frac32h^2\omega^2\|q_x\|^2.
	\end{equation}
	Monotonicity of $N_Y$ implies
	$\langle n_{t+1}-n_t,q_y\rangle\ge0$.
	Moreover, Lipschitz continuity of $\nabla f$ gives
	$\|\nabla_yG(x_{t+1},y_{t+1};z)
	-\nabla_yG(x_t,y_t;z)\|
	\le hL\|q_x\|+h(L+r_y)\|q_y\|$.
	Using the definition of $\delta_y$, these bounds yield
	\begin{align}
		\langle u-v,q_y\rangle
		&\ge-\|\delta_y\|\|q_y\|
		-hL\|q_x\|\|q_y\|
		-h(L+r_y)\|q_y\|^2\notag\\
		&\ge-\frac1{4h\gamma}\|\delta_y\|^2
		-\frac{hL_G}{2}\|q_x\|^2
		-h(L_G+\gamma)\|q_y\|^2,
		\label{eq:momentum-product-bound}
	\end{align}
	where the second inequality follows from Young's
	inequality and $L_G=3L+r_y$.
	Rewriting \eqref{eq:feedback-identity} as
	$v=(1+h(\omega+\gamma))u-h(\omega-\gamma)b_y$
	and completing the square gives
	\begin{equation}\label{eq:feedback-square}
		(1+h\gamma)\|u\|^2-\|v\|^2
		\le-h\omega\|u\|^2
		+\frac{h(\omega-\gamma)^2}{\omega+\gamma}\|b_y\|^2
		\le h\omega(\|b_y\|^2-\|u\|^2).
	\end{equation}
	The first inequality follows from
	$(1+b)\|u\|^2-\|(1+b)u-cb_y\|^2
	\le(c^2/b)\|b_y\|^2$
	with $b=h(\omega+\gamma)>0$ and $c=h(\omega-\gamma)$.
	The second uses $0<\gamma\le\omega$.
	Finally, Young's inequality and $q_i=b_i+e_i$ imply
	\begin{equation}\label{eq:displacement-bounds}
		\|q_x\|^2\le2\|b_x\|^2+2\|e_x\|^2,
		\qquad
		\|q_y\|^2\le\frac54\|b_y\|^2+5\|e_y\|^2.
	\end{equation}
	Substituting
	\eqref{eq:shifted-square-bound}--%
	\eqref{eq:displacement-bounds}
	into \eqref{eq:energy-after-direction-bounds} gives
	\begin{align}
		\mathcal D
		\le{}&
		-\frac{2h}{L}A_x\|b_x\|^2
		-\frac{2h}{L}A_y\|b_y\|^2
		+\frac{2h}{L}B_x\|e_x\|^2
		+\frac{2h}{L}B_y\|e_y\|^2\notag\\
		&-\frac{3}{4L}\|\delta_x\|^2
		-\frac1{2L}\|\delta_y\|^2
		-\frac2L\sum_{i\in\{x,y\}}
		\langle e_i,\delta_i\rangle
		-h\gamma\omega\|X_+\|^2
		-\frac{h\gamma\omega}{8L^2}\|u\|^2,
		\label{eq:energy-after-feedback-bounds}
	\end{align}
	where
	$A_x:=L/2-\omega-\gamma/2-h(3\omega^2+L_G\gamma)$,
	$A_y:=\gamma/4-5h\gamma(L_G+\gamma)/4
	-\gamma\omega/(16L)$,
	$B_x:=L+h(3\omega^2+L_G\gamma)$, and
	$B_y:=2\gamma+5h\gamma(L_G+\gamma)$.
	
	\rev{\emph{Step 4: parameter bounds and mismatch absorption.}}
	\rev{The consequences of \eqref{eq:parameters} give}
	\begin{align*}
		\frac{A_x}{L}
		&\ge\frac12-\frac{13}{50}-\frac1{16}
		-\frac{2}{45}\,3\left(\frac{13}{50}\right)^2
		-\frac{2}{15}\frac18>\frac3{20},\\
		\frac{A_y}{\gamma}
		&\ge\frac14-\frac54
		\left(\frac2{15}+\frac1{180}\right)
		-\frac{13}{800}>\frac1{19},\\
		2hB_x
		&\le\frac4{45}\left[
		1+\frac2{45}\,3\left(\frac{13}{50}\right)^2
		+\frac2{15}\frac18\right]<\frac1{10},\\
		2hB_y
		&\le\frac1{90}\left[
		2+5\left(\frac2{15}+\frac1{180}\right)\right]
		<\frac1{10}.
	\end{align*}
	Thus, the right-hand side of
	\eqref{eq:energy-after-feedback-bounds} is at most
	the right-hand side of \eqref{eq:fixed-center-energy}
	plus
	$\frac1{10L}\sum_i\|e_i\|^2
	-\frac1{2L}\sum_i\|\delta_i\|^2
	-\frac2L\sum_i\langle e_i,\delta_i\rangle$.
	By Lemma~\ref{lem:fixed-center-mismatch} and
	Cauchy--Schwarz, this remainder is bounded above by
	$L^{-1}(1/360-1/2+1/3)\sum_i\|\delta_i\|^2\le0$.
	This proves \eqref{eq:fixed-center-energy}.
	Dropping its nonpositive right-hand side and using
	$1/(1+h\gamma)=2\alpha/(1+\alpha)$ proves
	\eqref{eq:tracking-one-step-contraction}.
\end{proof}

\subsection{\rev{Moving-center potential descent}}
\label{sec:moving-center-descent}

We now combine fixed-center contraction with the envelope
decrease to control the change in the tracking target.
The resulting potential function descent bounds both the
tracking errors and the envelope gradient.

\begin{theorem}[\rev{Potential-function descent}]
	\label{thm:common-descent}
	Suppose Assumption~\ref{ass:main} and
	\eqref{eq:parameters} hold.
	Let $(w_t,z_t)_{t=0}^T$ be generated by
	Algorithm~\ref{alg:main} from an admissible state $w_0$
	and a center $z_0\in X$.
	Then, for $t=0,\ldots,T-1$,
	\begin{align}
		V(w_{t+1};z_{t+1})-V(w_t;z_t)
		&\le-\frac h4\|d^x(w_{t+1};z_t)\|^2
		-\frac{1-\alpha}{19\alpha L}
		\|d^y(w_{t+1};z_t)\|^2\notag\\
		&\quad-\frac{(1-\alpha)L}{35\alpha}
		\|x_{t+1}-x^\star(z_t)\|^2
		-\frac{1-\alpha}{160\alpha L}\|v_{t+1}\|^2
		-\frac{\beta}{4L}\|\nabla p(z_t)\|^2.
		\label{eq:V-descent}
	\end{align}
\end{theorem}

\begin{proof}
	The projection updates preserve admissibility, while
	$0<\beta<1$ and convexity of $X$ ensure that $z_t\in X$
	for every $t$.
	Fix $t$, and set $s=z_{t+1}-z_t$,
	$e=x_{t+1}-x^\star(z_t)$, and $g=\nabla p(z_t)$.
	The center update and \eqref{eq:envelope-gradient} imply
	$s=\beta(e-g/(2L))$.
	By Lemma~\ref{lem:envelope-properties},
	$p(z_{t+1})-p(z_t)\le\langle g,s\rangle+L\|s\|^2$.
	Substituting the identity
	$\langle g,s\rangle
	=L\beta\|e\|^2-\beta\|g\|^2/(4L)-L\|s\|^2/\beta$
	into this bound gives
	\begin{align}
		p(z_{t+1})-p(z_t)
		&\le L\beta\|x_{t+1}-x^\star(z_t)\|^2
		-\frac{\beta}{4L}\|\nabla p(z_t)\|^2
		-L\left(\frac1\beta-1\right)
		\|z_{t+1}-z_t\|^2\notag\\
		&\le L\beta\|x_{t+1}-x^\star(z_t)\|^2
		-\frac{\beta}{4L}\|\nabla p(z_t)\|^2
		-\frac{24L}{h\gamma}\|z_{t+1}-z_t\|^2,
		\label{eq:envelope-step}
	\end{align}
	where the second inequality follows from
	$\beta\le h\gamma/25$ and $h\gamma\le1$.
	
	We next bound the energy change with $w_{t+1}$ held fixed.
	Since $d^y$ is independent of the center and
	$d^x(w_{t+1};z_{t+1})=d^x(w_{t+1};z_t)-2Ls$,
	\eqref{eq:tracking-energy-expanded} and
	\eqref{eq:A-combined-center-change} give
	\begin{align*}
		\mathscr E(w_{t+1};z_{t+1})
		-\mathscr E(w_{t+1};z_t)
		&\le4\omega\langle e,s\rangle
		-4\langle d^x(w_{t+1};z_t),s\rangle
		+4L\|s\|^2\\
		&\quad+\frac{2\gamma}{L}
		\langle v_{t+1},
		y^\star(z_t)-y^\star(z_{t+1})\rangle.
	\end{align*}
	Lemma~\ref{lem:saddle-properties} gives
	$\|y^\star(z_{t+1})-y^\star(z_t)\|
	\le\sqrt{L/\mu_y}\|s\|$.
	Combining this bound with Young's inequality for the
	three cross terms yields
	\begin{align}
		\mathscr E(w_{t+1};z_{t+1})
		-\mathscr E(w_{t+1};z_t)
		&\le\frac{h\gamma\omega}{6}
		\|x_{t+1}-x^\star(z_t)\|^2
		+\frac h{20}\|d^x(w_{t+1};z_t)\|^2
		+\frac{h\gamma\omega}{48L^2}\|v_{t+1}\|^2
		\notag\\
		&\quad+\left(
		4L+\frac{24\omega}{h\gamma}+\frac{80}{h}
		+\frac{48L\gamma}{h\omega\mu_y}\right)
		\|z_{t+1}-z_t\|^2.
		\label{eq:center-energy-change}
	\end{align}
	
	Combining \eqref{eq:envelope-step},
	\eqref{eq:center-energy-change}, and
	\eqref{eq:fixed-center-energy} bounds the total change in $V$.
	The resulting term $-h\gamma\mathscr E(w_{t+1};z_t)$
	is nonpositive by Lemma~\ref{lem:tracking-nonnegative}.
	Discarding this term gives
	\begin{align}
		V(w_{t+1};z_{t+1})-V(w_t;z_t)
		&\le-\frac h4\|d^x(w_{t+1};z_t)\|^2
		-\frac{2h\gamma}{19L}\|d^y(w_{t+1};z_t)\|^2
		-\frac{\beta}{4L}\|\nabla p(z_t)\|^2\notag\\
		&\quad-\left(\frac{5h\gamma\omega}{6}-L\beta\right)
		\|x_{t+1}-x^\star(z_t)\|^2
		-\frac{5h\gamma\omega}{48L^2}\|v_{t+1}\|^2
		\notag\\
		&\quad+\left(
		4L+\frac{24\omega}{h\gamma}+\frac{80}{h}
		+\frac{48L\gamma}{h\omega\mu_y}
		-\frac{24L}{h\gamma}\right)
		\|z_{t+1}-z_t\|^2.
		\label{eq:lyapunov-after-young}
	\end{align}
	Let $C_s$ denote the coefficient of
	$\|z_{t+1}-z_t\|^2$ in
	\eqref{eq:lyapunov-after-young}.
	\rev{The consequences of \eqref{eq:parameters}},
	including $\gamma^2\le L\mu_y/64$ and $\omega\ge L/8$,
	imply
	\[
	\frac{h\gamma}{L}C_s
	=4h\gamma+\frac{24\omega}{L}+\frac{80\gamma}{L}
	+\frac{48\gamma^2}{\omega\mu_y}-24
	\le\frac1{45}+\frac{156}{25}+10+6-24<0.
	\]
	The same parameter bounds give
	$5h\gamma\omega/6-L\beta
	\ge h\gamma L(5/48-1/25)\ge2h\gamma L/35$
	and
	$5h\gamma\omega/(48L^2)\ge h\gamma/(80L)$.
	Substituting these estimates into
	\eqref{eq:lyapunov-after-young} and using
	$2h\gamma=(1-\alpha)/\alpha$ proves
	\eqref{eq:V-descent}.
\end{proof}

Theorem~\ref{thm:common-descent} provides a common descent
estimate for both stationarity criteria.
The residual, tracking, and momentum bounds yield GS
certificates, while the envelope-gradient bound and
Lemma~\ref{lem:regularized-envelope-bias} yield OS guarantees.
The fixed-center contraction will also be used to refine
the initialization dependence through a warm-up phase.

\subsection{\rev{Stationarity estimates and output guarantees}}
\label{sec:complexity}

We now use potential function descent to derive iteration
complexity bounds for the OS and GS criteria.
\rev{Recall the computable regularized residual $S_t$
	from \eqref{eq:certificate}.
	The next two theorems first state the estimates in terms
	of the initial potential.
	Corollary~\ref{cor:implementable-outputs} gives the
	corresponding output guarantees, while
	Theorem~\ref{thm:end-to-end} replaces the unknown potential
	by supplied initialization bounds.}

\begin{theorem}[NC--C stationarity]
	\label{thm:ncc-optimization-complexity}
	\label{thm:ncc-game-complexity}
	Suppose Assumption~\ref{ass:main} holds with $\mu=0$.
	Given $\varepsilon>0$, run Algorithm~\ref{alg:main}
	from an admissible state $w_0$ and a center $z_0\in X$,
	with parameters satisfying \eqref{eq:parameters}.
	\begin{enumerate}[label=(\roman*),leftmargin=*]
		\item
		For OS, choose
		\begin{equation}\label{eq:os-main-budget}
			r_y=\min\left\{L,
			\frac{\varepsilon^2}{16LD_Y^2}\right\},
			\qquad
			T=\left\lceil\max\left\{
			1,\frac{16LV(w_0;z_0)}{\beta\varepsilon^2}
			\right\}\right\rceil.
		\end{equation}
		Then
		\begin{equation}\label{eq:ncc-OS-guarantee}
			\min_{0\le t<T}\|\nabla\Phi_{2L}(z_t)\|
			\le\varepsilon.
		\end{equation}
		
		\item
		For GS, choose
		\begin{equation}\label{eq:gs-main-budget}
			r_y=\min\left\{L,\frac{\varepsilon}{2D_Y}\right\},
			\qquad
			T=\left\lceil\max\left\{
			1,\frac{28LV(w_0;z_0)}{\beta\varepsilon^2}
			\right\}\right\rceil.
		\end{equation}
		Then
		\begin{equation}\label{eq:ncc-GS-guarantee}
			\min_{0\le t<T}R(x_{t+1},y_{t+1})
			\le\varepsilon.
		\end{equation}
	\end{enumerate}
	\rev{Both conclusions remain valid if $T$ is increased
		or $V(w_0;z_0)$ in the budget is replaced by an upper bound.}
\end{theorem}

\begin{proof}
	We first derive a residual estimate valid for every
	$\mu\ge0$ under \eqref{eq:parameters}.
	Set $a_t=d^x(w_{t+1};z_t)$,
	$b_t=d^y(w_{t+1};z_t)$,
	$e_t=x_{t+1}-x^\star(z_t)$, and
	$g_t=\nabla p(z_t)$.
	The definitions of the residuals and
	\eqref{eq:envelope-gradient} give
	\begin{align*}
		\nabla_xf(x_{t+1},y_{t+1})+\xi_{t+1}
		&=a_t-2Le_t+g_t,\\
		-\nabla_yf(x_{t+1},y_{t+1})
		+r_y y_{t+1}+n_{t+1}
		&=-b_t+v_{t+1}.
	\end{align*}
	Weighted Cauchy--Schwarz and Young's inequality imply
	\begin{align*}
		\|a_t-2Le_t+g_t\|^2
		&\le\left(
		\frac4h+\frac{70L}{h\gamma}+\frac{4L}{\beta}\right)
		\left(
		\frac h4\|a_t\|^2
		+\frac{2h\gamma L}{35}\|e_t\|^2
		+\frac{\beta}{4L}\|g_t\|^2\right),\\
		\|-b_t+v_{t+1}\|^2
		&\le5\|b_t\|^2+\frac54\|v_{t+1}\|^2.
	\end{align*}
	\rev{The consequences of \eqref{eq:parameters} give}
	$\beta/(h\gamma)\le1/25$ and $\gamma/L\le1/8$.
	Hence
	$4\beta/(hL)+70\beta/(h\gamma)+4
	\le1/50+14/5+4<7$.
	The same bounds give
	$5\le(7L/\beta)(2h\gamma/(19L))$
	and
	$5/4\le(7L/\beta)(h\gamma/(80L))$.
	Combining these estimates with \eqref{eq:V-descent}
	and $2h\gamma=(1-\alpha)/\alpha$ yields
	\begin{align}
		S_t^2
		&\le\frac{7L}{\beta}\left(
		\frac h4\|a_t\|^2
		+\frac{2h\gamma}{19L}\|b_t\|^2
		+\frac{2h\gamma L}{35}\|e_t\|^2
		+\frac{h\gamma}{80L}\|v_{t+1}\|^2
		+\frac{\beta}{4L}\|g_t\|^2\right)\notag\\
		&\le\frac{7L}{\beta}
		\bigl(V(w_t;z_t)-V(w_{t+1};z_{t+1})\bigr).
		\label{eq:certificate-descent}
	\end{align}
	
	For OS, summing \eqref{eq:V-descent} and using
	Lemma~\ref{lem:tracking-nonnegative} gives
	$\sum_{t=0}^{T-1}\|\nabla p(z_t)\|^2
	\le4LV(w_0;z_0)/\beta$.
	Lemma~\ref{lem:regularized-envelope-bias}, with
	$r_x=2L$ and $\mu_x=L$, therefore implies
	\begin{align}
		\rev{\frac1T\sum_{t=0}^{T-1}
			\|\nabla\Phi_{2L}(z_t)\|^2}
		&\le\frac2T\sum_{t=0}^{T-1}\|\nabla p(z_t)\|^2
		+8Lr_yD_Y^2\notag\\
		&\le\frac{8LV(w_0;z_0)}{\beta T}
		+8Lr_yD_Y^2
		\le\frac{\varepsilon^2}{2}
		+\frac{\varepsilon^2}{2}
		=\varepsilon^2,
		\label{eq:os-final-bound}
	\end{align}
	where the last inequality uses \eqref{eq:os-main-budget}.
	This proves \eqref{eq:ncc-OS-guarantee}.
	
	For GS, admissibility and the triangle inequality give
	\begin{align}
		R(x_{t+1},y_{t+1})
		&\le\left[
		\|\nabla_xf(x_{t+1},y_{t+1})+\xi_{t+1}\|^2
		+\|-\nabla_yf(x_{t+1},y_{t+1})+n_{t+1}\|^2
		\right]^{1/2}\notag\\
		&\le S_t+r_y\|y_{t+1}\|
		\le S_t+r_yD_Y.
		\label{eq:pointwise-GS-transfer}
	\end{align}
	Summing \eqref{eq:certificate-descent} and using
	nonnegativity of $V$ yields
	\begin{align}
		\min_{0\le t<T}R(x_{t+1},y_{t+1})
		&\le\sqrt{\frac1T\sum_{t=0}^{T-1}S_t^2}
		+r_yD_Y\notag\\
		&\le\sqrt{\frac{7LV(w_0;z_0)}{\beta T}}
		+r_yD_Y
		\rev{\le\frac{\varepsilon}{2}
			+\frac{\varepsilon}{2}=\varepsilon},
		\label{eq:summation-gs}
	\end{align}
	where the last inequality follows from
	\eqref{eq:gs-main-budget}.
	This proves \eqref{eq:ncc-GS-guarantee}.
	\rev{The same estimates hold for any larger budget
		or any upper bound on the initial potential.}
\end{proof}

Under dual strong concavity, we take $r_y=0$, so the
regularization bias vanishes.

\begin{theorem}[NC--SC stationarity]
	\label{thm:ncsc-complexity}
	Suppose Assumption~\ref{ass:main} holds with $\mu>0$.
	Given $\varepsilon>0$, run Algorithm~\ref{alg:main}
	from an admissible state $w_0$ and a center $z_0\in X$,
	with parameters satisfying \eqref{eq:parameters} and
	\begin{equation}\label{eq:ncsc-main-budget}
		r_y=0,\qquad
		T=\left\lceil\max\left\{
		1,\frac{7LV(w_0;z_0)}{\beta\varepsilon^2}
		\right\}\right\rceil.
	\end{equation}
	Then
	\begin{equation}\label{eq:ncsc-guarantees}
		\min_{0\le t<T}\|\nabla\Phi_{2L}(z_t)\|\le\varepsilon,
		\qquad
		\min_{0\le t<T}R(x_{t+1},y_{t+1})\le\varepsilon.
	\end{equation}
	\rev{The conclusions remain valid if $T$ is increased
		or the initial potential in the budget is replaced
		by an upper bound.}
\end{theorem}

\begin{proof}
	Since $r_y=0$, Lemma~\ref{lem:regularized-envelope-bias}
	gives $\nabla\Phi_{2L}=\nabla p$.
	Summing \eqref{eq:V-descent} and using nonnegativity
	of $V$ therefore yields
	\rev{$T^{-1}\sum_{t=0}^{T-1}
		\|\nabla\Phi_{2L}(z_t)\|^2
		\le4LV(w_0;z_0)/(\beta T)\le\varepsilon^2$.}
	For GS, \eqref{eq:pointwise-GS-transfer} gives
	$R(x_{t+1},y_{t+1})\le S_t$.
	The estimate \eqref{eq:certificate-descent} remains
	valid for $\mu>0$, so summing it gives
	\rev{$T^{-1}\sum_{t=0}^{T-1}S_t^2
		\le7LV(w_0;z_0)/(\beta T)\le\varepsilon^2$.}
	Both final inequalities follow from
	\eqref{eq:ncsc-main-budget}.
	\rev{Taking minima proves \eqref{eq:ncsc-guarantees},
		and the same estimates apply to larger budgets or
		upper bounds on the initial potential.}
\end{proof}

\begin{corollary}[\rev{Implementable output guarantees}]
	\label{cor:implementable-outputs}
	\begin{revision}
		Under the respective assumptions and budgets of
		Theorems~\ref{thm:ncc-optimization-complexity}
		and~\ref{thm:ncsc-complexity}, the output rules in
		Algorithm~\ref{alg:main} satisfy the following guarantees.
		\begin{enumerate}[label=(\roman*),leftmargin=*]
			\item
			For OS, if $\tau$ is independent and uniform on
			$\{0,\ldots,T-1\}$, then
			$\mathbb E\|\nabla\Phi_{2L}(z_\tau)\|^2
			\le\varepsilon^2$.
			
			\item
			For GS, if
			$t_\star\in\argmin_{0\le t<T}S_t$,
			then
			$R(x_{t_\star+1},y_{t_\star+1})\le\varepsilon$.
		\end{enumerate}
		Neither output rule requires an additional gradient
		evaluation or an evaluation of the Moreau-envelope
		gradient.
	\end{revision}
\end{corollary}

\begin{proof}
	\begin{revision}
		Uniform sampling gives
		$\mathbb E\|\nabla\Phi_{2L}(z_\tau)\|^2
		=T^{-1}\sum_{t=0}^{T-1}\|\nabla\Phi_{2L}(z_t)\|^2$.
		The average is bounded by $\varepsilon^2$ in
		\eqref{eq:os-final-bound} for NC--C problems and in
		the proof of Theorem~\ref{thm:ncsc-complexity}
		for NC--SC problems.
		
		For GS, the selected certificate satisfies
		$S_{t_\star}^2\le
		T^{-1}\sum_{t=0}^{T-1}S_t^2
		\le7LV(w_0;z_0)/(\beta T)$.
		Combining this inequality with
		\eqref{eq:pointwise-GS-transfer} and the corresponding
		budget gives the claimed residual bound.
		The gradient needed to compute $S_t$ is already
		evaluated in the momentum update, and all remaining
		quantities are stored algorithmic variables.
	\end{revision}
\end{proof}

\subsection{\rev{Initialization bounds}}
\label{sec:initialization-bounds}

The preceding bounds depend on the initial potential function
$V(w_0;z_0)$.
To express this quantity in terms of the initial value gap
and gradients, we take $z_0=x_0$ and
$\xi_0=n_0=v_0=0$.

\begin{proposition}[\rev{Initial potential bound}]
	\label{prop:initial-energy-bound}
	Suppose Assumption~\ref{ass:main} and
	\eqref{eq:parameters} hold.
	For $(x_0,y_0)\in X\times Y$, initialize
	$z_0=x_0$ and $w_0=(x_0,y_0,0,0,0)$.
	\rev{Let $\Delta_\phi=\phi(x_0)-\phi_{\inf}$ and
		evaluate $\Delta_0$ in \eqref{eq:uniform-initial-bound}
		with $y^{\rm init}=y_0$}.
	Then
	\begin{equation}\label{eq:initial-bound-result}
		0\le V(w_0;z_0)
		\le\Delta_0+\left(r_y+\frac{2r_y^2}{L}\right)D_Y^2.
	\end{equation}
	In particular, if $r_yD_Y^2\le\Delta_0$, then
	$V(w_0;z_0)\le\Delta_0+3r_yD_Y^2\le4\Delta_0$.
\end{proposition}

\begin{proof}
	The initialization is admissible because
	$0\in N_X(x_0)$ and $0\in N_Y(y_0)$.
	Evaluating the minimization defining $p(x_0)$ at $x_0$
	gives
	$p(x_0)\le\max_{y\in Y}G(x_0,y;x_0)
	=\max_{y\in Y}\{f(x_0,y)-(r_y/2)\|y\|^2\}
	\le\phi(x_0)$.
	Concavity in $y$ and $\diam(Y)\le D_Y$ imply
	$\phi(x_0)-f(x_0,y_0)
	\le\max_{y\in Y}
	\langle\nabla_yf(x_0,y_0),y-y_0\rangle
	\le D_Y\|\nabla_yf(x_0,y_0)\|$.
	Together with
	$p_{\inf}\ge\phi_{\inf}-r_yD_Y^2/2$,
	these bounds give
	$p(x_0)-p_{\inf}\le\Delta_\phi+r_yD_Y^2/2$.
	
	Since $v_0=0$, the expanded energy
	\eqref{eq:tracking-energy-expanded} yields
	\begin{align*}
		\mathscr E(w_0;x_0)
		={}&\frac{2\omega}{L}\left[
		p(x_0)-f(x_0,y_0)+\frac{r_y}{2}\|y_0\|^2\right]
		+\frac{\|\nabla_xf(x_0,y_0)\|^2
			+\|\nabla_yf(x_0,y_0)-r_y y_0\|^2}{L}\\
		\le{}&
		D_Y\|\nabla_yf(x_0,y_0)\|
		+\frac{r_yD_Y^2}{2}
		+\frac{\|\nabla_xf(x_0,y_0)\|^2
			+2\|\nabla_yf(x_0,y_0)\|^2+2r_y^2D_Y^2}{L}.
	\end{align*}
	Here we first bounded the bracket by the nonnegative
	quantity
	$D_Y\|\nabla_yf(x_0,y_0)\|+r_yD_Y^2/2$,
	and then used $0<2\omega/L\le1$ and
	$\|a-b\|^2\le2\|a\|^2+2\|b\|^2$.
	Adding the bound on $p(x_0)-p_{\inf}$ gives
	$V(w_0;z_0)
	=p(x_0)-p_{\inf}+\mathscr E(w_0;x_0)
	\le\Delta_0+(r_y+2r_y^2/L)D_Y^2$.
	The lower bound follows from
	Lemma~\ref{lem:tracking-nonnegative}.
	Finally, $r_y\le L$ implies
	$r_y+2r_y^2/L\le3r_y$,
	which proves the last assertion.
\end{proof}

\begin{remark}\label{rem:regularization-choice}
	For $\Delta_0>0$, the choices in
	\eqref{eq:os-main-budget} and \eqref{eq:gs-main-budget}
	ensure $r_yD_Y^2\le\Delta_0$ whenever
	$0<\varepsilon\le4\sqrt{L\Delta_0}$ for OS and
	$0<\varepsilon\le2\Delta_0/D_Y$ for GS.
	In the NC--SC case, $r_y=0$ satisfies this condition
	for every $\varepsilon>0$.
\end{remark}

\begin{remark}\label{rem:base-query-complexities}
	Initialize $z_0=x_0$ and $w_0=(x_0,y_0,0,0,0)$.
	Choose the parameters as in
	Theorems~\ref{thm:ncc-optimization-complexity}
	and~\ref{thm:ncsc-complexity}, with
	$h=\Theta(L^{-1})$ and
	$\beta=\Theta(\sqrt{(\mu+r_y)/L})$.
	For fixed problem and initialization data with
	$\Delta_0>0$,
	Proposition~\ref{prop:initial-energy-bound} gives
	$V(w_0;z_0)\le4\Delta_0$ for all sufficiently small
	$\varepsilon$.
	Since each iteration uses a constant number of gradient
	oracle calls, the resulting gradient oracle complexities
	as $\varepsilon\downarrow0$ are
	$O(L^2D_Y\Delta_0\varepsilon^{-3})$ for NC--C OS,
	$O(L^{3/2}D_Y^{1/2}\Delta_0\varepsilon^{-5/2})$
	for NC--C GS, and
	$O(\sqrt{\kappa}\,L\Delta_0\varepsilon^{-2})$
	for NC--SC OS and GS, where $\kappa:=L/\mu$.
	\rev{These bounds use the output guarantees in
		Corollary~\ref{cor:implementable-outputs}.
		For a prescribed budget, the potential can be replaced
		by the upper bound in
		Proposition~\ref{prop:initial-energy-bound},
		using a supplied upper bound on $\Delta_\phi$.}
\end{remark}

\subsection{Warm-up initialization}\label{sec:warmup}

A one-time fixed-center phase reduces the initialization
factor in the leading complexity bounds from $\Delta_0$
to a supplied bound $\Delta\ge\Delta_\phi$.
It uses the primal--dual updates of Algorithm~\ref{alg:main}
while keeping the center fixed.
After a prescribed number of steps, the method resumes
its usual center updates, without solving nested subproblems or imposing an
inner stopping criterion.

\begin{algorithm}[H]
	\caption{One-time fixed-center initialization}
	\label{alg:warmup}
	\footnotesize
	\begin{algorithmic}[1]
		\Require $x_0\in X$, $y^{\mathrm{init}}\in Y$,
		$K\in\mathbb N_0$; parameters $r_x,r_y,h,\alpha,\rho$.
		\State Set $z_0=x_0$.
		\State Initialize
		$w^{\mathrm{wu}}_0=(x_0,y^{\mathrm{init}},0,0,0)$.
		\For{$k=0,\ldots,K-1$}
		\State Obtain $w^{\mathrm{wu}}_{k+1}$ from
		$w^{\mathrm{wu}}_k$ by executing lines
		\ref{alg:core-start}--\ref{alg:core-end} of
		Algorithm~\ref{alg:main} at the fixed center $z_0$.
		\EndFor
		\State Set $w_0^{\mathrm{main}}:=w_K^{\mathrm{wu}}$.
		\State \Return $(w_0^{\mathrm{main}},z_0)$ as the
		initial state and center for Algorithm~\ref{alg:main}.
	\end{algorithmic}
\end{algorithm}

The gap $\Delta_\phi$ remains defined at the original
point $x_0=z_0$; its evaluation is not required.
The warm-up preserves admissibility and contracts the
tracking energy without changing $p(z_0)-p_{\inf}$.

\begin{proposition}[Initial potential after warm-up]
	\label{prop:warmup-potential function}
	Suppose Assumption~\ref{ass:main} holds.
	Use the parameters in \eqref{eq:explicit-main-parameters}
	with $r_y\ge0$ and $0<\mu+r_y\le L$.
	Let $\Delta>0$ satisfy $\Delta_\phi\le\Delta$ and
	$r_yD_Y^2\le\Delta$.
	Initialize Algorithm~\ref{alg:warmup} at
	$w_0^{\mathrm{wu}}=(x_0,y^{\mathrm{init}},0,0,0)$
	with $z_0=x_0$.
	If
	\begin{equation}\label{eq:warmup-potential function-length}
		K\ge1+\left\lceil
		\frac{1+\alpha}{1-\alpha}
		\log\!\left[
		\frac{28}{3h^2L^2}
		\left(1+\frac{LD_Y^2}{\Delta}\right)
		\right]\right\rceil,
	\end{equation}
	then
	\begin{equation}\label{eq:warmup-potential function-bound}
		V(w_K^{\mathrm{wu}};z_0)\le2\Delta.
	\end{equation}
	The prescribed warm-up length is independent of the
	initial gradient norms.
	In particular, if $\Delta_\phi>0$ and
	$r_yD_Y^2\le\Delta_\phi$, one may take
	$\Delta=\Delta_\phi$.
\end{proposition}

\begin{proof}
	Write $w_k=w_k^{\mathrm{wu}}$,
	$u_k=(x_k,y_k)$, $\widetilde u_0=(\widetilde x_0,
	\widetilde y_0)$, and
	$u^\star=(x^\star(z_0),y^\star(z_0))$.
	Let $a=\|u_0-u^\star\|$ and
	$G(\cdot,\cdot)=G(\cdot,\cdot;z_0)$.
	The $L$-strong convexity of $\max_{y\in Y}G(\cdot,y)$
	and \eqref{eq:regularization-bias} give
	\[
	\frac L2\|x_0-x^\star(z_0)\|^2
	\le \max_{y\in Y}G(x_0,y)-p(z_0)
	\le\Delta_\phi+\frac{r_yD_Y^2}{2}.
	\]
	Since $r_y\le L$ and $\diam(Y)\le D_Y$,
	\[
	La^2\le2\Delta_\phi+(L+r_y)D_Y^2
	\le2(\Delta+LD_Y^2).
	\]
	
	\emph{First update.}
	The parameters satisfy $hL\le2/45$ and
	$\rho=hL/8\le1/180$, and $\nabla G$ is
	$3L$-Lipschitz.
	Saddle-point optimality and projection
	nonexpansiveness imply
	$\|\widetilde u_0-u^\star\|
	\le(1+3hL)a\le17a/15$.
	Since $\xi_0=n_0=v_0=0$, the first dual correction is
	$y_1=\Pi_Y(y_0+h(1+\rho)
	\nabla_yG(\widetilde u_0))$.
	Define the frozen-gradient correction
	\[
	\widehat u_1=
	\left(
	\Pi_X(x_0-h\nabla_xG(u^\star)),
	\Pi_Y(y_0+h(1+\rho)\nabla_yG(u^\star))
	\right).
	\]
	The translated projection defining $\widehat u_1$
	is firmly nonexpansive and fixes $u^\star$.
	Thus,
	\[
	\|\widehat u_1-u^\star\|^2
	+\|\widehat u_1-u_0\|^2\le a^2,
	\qquad
	\|u_1-\widehat u_1\|
	\le3h(1+\rho)L\|\widetilde u_0-u^\star\|.
	\]
	Consequently,
	\[
	\max\{\|u_1-u^\star\|,\|u_1-u_0\|\}
	\le\left[1+3h(1+\rho)L\frac{17}{15}\right]a
	\le\frac76a.
	\]
	
	Write $d_1=(d^x(w_1;z_0),d^y(w_1;z_0))$.
	Equations~\eqref{eq:corrected-directions}
	and~\eqref{eq:correction-identities} yield
	\begin{align*}
		\|d_1\|
		&\le h^{-1}\|u_1-u_0\|
		+3L(1+\rho)\|u_1-\widetilde u_0\|\\
		&\le\left[
		\frac7{6h}
		+3L(1+\rho)\left(\frac76+\frac{17}{15}\right)
		\right]a
		\le\frac3{2h}a.
	\end{align*}
	The momentum update gives
	$v_1=\rho d^y(w_1;z_0)/(1+\rho)$, so
	$\|v_1\|\le3La/16$.
	
	\emph{Energy bound and contraction.}
	By \eqref{eq:saddle-lower},
	\eqref{eq:tracking-energy-expanded}, and
	Cauchy--Schwarz,
	\begin{align*}
		\mathscr E(w_1;z_0)
		&\le\frac{\|d_1\|^2}{L}
		+\frac{2\omega}{L}\|d^y(w_1;z_0)\|
		\|y_1-y^\star(z_0)\|\\
		&\quad+\frac{2(\omega+\gamma)}{L}\|v_1\|
		\|y_1-y^\star(z_0)\|
		+\frac{\gamma}{8L^2}\|v_1\|^2\\
		&\le\left(\frac9{4Lh^2}+\frac7{6h}
		+\frac L4\right)a^2
		\le\frac7{3Lh^2}a^2
		\le\frac{14}{3h^2L^2}(\Delta+LD_Y^2).
	\end{align*}
	Here we used $\gamma\le L/8$, $\omega\le L/3$,
	and $hL\le2/45$.
	Applying \eqref{eq:tracking-one-step-contraction}
	for the remaining $K-1$ steps gives
	\begin{align*}
		\mathscr E(w_K;z_0)
		&\le
		\left(\frac{2\alpha}{1+\alpha}\right)^{K-1}
		\mathscr E(w_1;z_0)\\
		&\le\frac{14}{3h^2L^2}(\Delta+LD_Y^2)
		\exp\!\left(
		-\frac{1-\alpha}{1+\alpha}(K-1)\right)
		\le\frac{\Delta}{2}.
	\end{align*}
	Finally, \eqref{eq:regularization-bias} implies
	$p(z_0)\le\phi(z_0)$ and
	$p_{\inf}\ge\phi_{\inf}-r_yD_Y^2/2$.
	Hence
	\[
	V(w_K;z_0)
	=p(z_0)-p_{\inf}+\mathscr E(w_K;z_0)
	\le\Delta_\phi+\frac{r_yD_Y^2}{2}
	+\frac{\Delta}{2}\le2\Delta.
	\]
\end{proof}

After warm-up, initialize Algorithm~\ref{alg:main} with
$w_0^{\mathrm{main}}=w_K^{\mathrm{wu}}$ and the unchanged
center $z_0$ and parameters.
Retain all state components, including the normal vectors
and momentum; do not reset the center to the final
primal warm-up iterate.
The quantities $\Delta_\phi$ and $\Delta_0$ remain
defined at the original initialization.
Theorems~\ref{thm:ncc-optimization-complexity}
and~\ref{thm:ncsc-complexity} then apply with
$V(w_0^{\mathrm{main}};z_0)\le2\Delta$.
The total number of updates is $K+T$.

\begin{remark}\label{rem:query-complexities}
	Suppose $\Delta_\phi>0$ and use the parameters of
	Theorem~\ref{thm:end-to-end} with
	$\Delta=\Delta_\phi$.
	For NC--C OS, assume
	$0<\varepsilon\le
	\min\{4LD_Y,4\sqrt{L\Delta_\phi}\}$;
	for NC--C GS, assume
	$0<\varepsilon\le
	\min\{2LD_Y,2\Delta_\phi/D_Y\}$.
	These conditions ensure $r_yD_Y^2\le\Delta_\phi$
	and $0<r_y\le L$.
	For NC--SC problems, take $r_y=0$.
	
	With $\ell_0=\log(2+LD_Y^2/\Delta_\phi)$,
	the prescribed warm-up length is
	$K=O(\sqrt{L/(\mu+r_y)}\,\ell_0)$.
	Table~\ref{tab:query-complexities} summarizes the
	resulting parameter and update complexity orders.
	For fixed problem data and positive $\Delta_\phi$,
	the warm-up cost is lower order than the main-phase
	cost as $\varepsilon\downarrow0$.
	With gradient reuse, the full oracle count is at most
	$1+2(K+T)$.
\end{remark}

\begin{table}[!ht]
	\caption{Parameter and update complexity orders with
		one-time warm-up as $\varepsilon\downarrow0$ for
		fixed problem data and $\Delta_\phi>0$.
		Here $\ell_0=\log(2+LD_Y^2/\Delta_\phi)$ and
		$\kappa=L/\mu$ in the NC--SC case.
		The total gradient complexity is $O(1+K+T)$.}
	\label{tab:query-complexities}
	\centering
	\setlength{\tabcolsep}{5pt}
	\renewcommand{\arraystretch}{1.8}
	\begin{tabular}{@{}ccccc@{}}
		\toprule
		Curvature & Criterion & $1-\alpha,\ \beta$
		& Warm-up $K$ & Moving-center $T$\\
		\midrule
		$\mu=0$ & OS
		& $\Theta\!\left(\dfrac{\varepsilon}{LD_Y}\right)$
		& $O\!\left(\dfrac{LD_Y}{\varepsilon}\ell_0\right)$
		& $O\!\left(
		\dfrac{L^2D_Y\Delta_\phi}{\varepsilon^3}\right)$\\[5pt]
		& GS
		& $\Theta\!\left(
		\sqrt{\dfrac{\varepsilon}{LD_Y}}\right)$
		& $O\!\left(
		\sqrt{\dfrac{LD_Y}{\varepsilon}}\ell_0\right)$
		& $O\!\left(
		\dfrac{L^{3/2}D_Y^{1/2}\Delta_\phi}
		{\varepsilon^{5/2}}\right)$\\[5pt]
		\midrule
		$\mu>0$ & OS, GS
		& $\Theta(\kappa^{-1/2})$
		& $O(\sqrt{\kappa}\,\ell_0)$
		& $O\!\left(
		\dfrac{\sqrt{\kappa}\,L\Delta_\phi}
		{\varepsilon^2}\right)$\\
		\bottomrule
	\end{tabular}
\end{table}

\section{Lower bounds}\label{sec:lower-bound}
\label{sec:ncc-lower-bounds}

We establish lower bounds for OS and GS in both the NC--C
and NC--SC settings under Assumption~\ref{ass:main}.
\rev{These bounds have the same parameter dependence as
	the leading terms in Theorem~\ref{thm:end-to-end},
	within the projected zero-respecting oracle model defined
	below and in the common accuracy ranges.
	The comparison separates the additive initialization cost,
	as explained in Remark~\ref{rem:scope-main-optimality}.}
The lower bounds do not impose a single-loop restriction
on competing algorithms.

\subsection{Problem and algorithm classes}
\label{subsec:lb-gates}

We first specify the problem class and oracle model.
The problem class retains the smoothness, dual-domain
diameter, and initial value-gap parameters appearing
in the upper bounds.
The dimensions are unrestricted and may depend on the
target accuracy.

\begin{definition}\label{def:lb-problem-class}
	Fix $L,D_Y,\Delta>0$ and $0\le\mu\le L$.
	Let $\mathcal F_\mu(L,D_Y,\Delta)$ denote the class
	of minimax problems satisfying Assumption~\ref{ass:main},
	with joint gradient Lipschitz constant at most $L$,
	dual strong concavity parameter at least $\mu$,
	and $\diam(Y)\le D_Y$.
	The initial point is $(0,0)\in X\times Y$ and
	satisfies
	$\Delta_\phi=\phi(0)-\inf_{x\in X}\phi(x)\le\Delta$.
	When $\mu=0$, we write $\mathcal F(L,D_Y,\Delta)$.
\end{definition}

At each query point $(x,y)$, the first-order oracle returns
$(\nabla_xf(x,y),\nabla_yf(x,y))$.
We count joint gradient evaluations and allow projections
onto $X$ and $Y$ without charging additional oracle calls.
Following the zero-respecting framework
\cite{carmon2020lower,li2021lower},
we restrict each preprojection vector to coordinates
revealed by previous query points and oracle responses.
\rev{We also record coordinates revealed by intermediate
	projections, so that the model permits multiple internal
	projection stages between gradient queries.}
This is a support restriction, rather than a linear-span
restriction: the algorithm may otherwise combine the
available information arbitrarily.

\begin{definition}[\rev{Projected zero-respecting algorithms}]
	\label{def:lb-zero-respecting}
	\begin{revision}
		A first-order algorithm initialized at $(0,0)$ is
		\emph{projected zero-respecting} if its computation
		maintains revealed-coordinate sets
		$\mathcal I_x\subseteq\{1,\ldots,d_x\}$ and
		$\mathcal I_y\subseteq\{1,\ldots,d_y\}$, initially empty,
		and satisfies the following rules.
		\begin{enumerate}[label=(\roman*),leftmargin=*]
			\item
			Every primal or dual vector formed before projection
			is supported on $\mathcal I_x$ or $\mathcal I_y$,
			respectively.
			Within these supports, its entries may be chosen
			arbitrarily using the available information.
			
			\item
			The algorithm may evaluate $\Pi_X(u)$ and $\Pi_Y(v)$
			when $\supp(u)\subseteq\mathcal I_x$ and
			$\supp(v)\subseteq\mathcal I_y$.
			The supports of the resulting projection outputs
			are added to the corresponding revealed-coordinate
			sets.
			
			\item
			Every gradient query is made at a feasible pair
			$(x,y)$ with
			$\supp(x)\subseteq\mathcal I_x$ and
			$\supp(y)\subseteq\mathcal I_y$.
			The supports of the returned primal and dual
			gradients are then added to $\mathcal I_x$ and
			$\mathcal I_y$, respectively.
			
			\item
			Every reported output is feasible and obeys the
			same support restrictions, after any terminal
			projections satisfying part~(ii).
		\end{enumerate}
		Only finitely many internal operations and projections
		are performed between successive gradient queries
		and before an output is reported.
		For a randomized algorithm, all restrictions hold
		for every realization of its internal randomness.
	\end{revision}
\end{definition}

\begin{revision}
	The initial gradient query at the origin is permitted
	because the zero vector has empty support.
	Projection-generated normal vectors require no additional
	oracle: their supports are contained in the union of
	the supports of the projection input and output.
	The definition does not permit an independent normal-cone
	oracle that introduces unrevealed coordinates.
	
	For OS, the output restriction applies to the reported
	primal point; for GS, it applies to both components of
	the reported pair.
	The definition permits multiple internal stages, nested
	iterations, and adaptive stopping rules.
	The model is coordinate dependent and is not the class
	of all deterministic or randomized first-order algorithms.
\end{revision}

Our constructions reveal coordinates successively through
the first-order oracle.
Until sufficiently many coordinates have been revealed,
the relevant OS or GS residual remains above the target
accuracy.
The lower bounds therefore follow by controlling the
number of coordinates that can be revealed per oracle call.
\rev{On the constructed feasible sets, projections preserve
	the relevant chain prefixes, although a projection may
	reveal another coordinate within an already revealed group.
	Thus, arbitrarily many permitted internal projections
	cannot advance the chain beyond the current prefix.
	This property is imposed on the hard instances, not on
	the entire problem class.}

\begin{proposition}[\rev{Algorithm membership}]
	\label{prop:algorithm-membership}
	\begin{revision}
		On every problem in
		$\mathcal F_\mu(L,D_Y,\Delta)$, MC-EGM initialized with
		$x_0=y_0=z_0=0$ and $\xi_0=n_0=v_0=0$ is projected
		zero-respecting in the sense of
		Definition~\ref{def:lb-zero-respecting}.
		The same statement holds with fixed-center warm-up
		and with either output rule in Algorithm~\ref{alg:main}.
	\end{revision}
\end{proposition}

\begin{proof}
	\begin{revision}
		Proceed in the chronological order of gradient queries
		and internal projections.
		The initial vectors have empty support.
		Suppose all stored primal and dual vectors are supported
		on the current revealed-coordinate sets.
		The gradient of $G$ is obtained from that of $f$ by
		adding $r_x(x-z)$ in the primal block and subtracting
		$r_y y$ in the dual block.
		These explicit corrections introduce no new coordinates.
		
		The predictor and correction inputs are combinations
		of stored vectors and available gradients, so they
		obey the preprojection support restrictions.
		Each projection output is added to the revealed
		coordinates before it is used in a subsequent query.
		The normal vectors are combinations of projection
		inputs, outputs, and available gradients.
		The momentum and center updates are likewise
		combinations of supported vectors.
		Thus, the induction is preserved through every
		elementary operation.
		
		Holding the center fixed changes none of these support
		properties, and retaining the warm-up state preserves
		the induction.
		The GS output is a previously queried feasible pair.
		The OS output is a stored feasible center supported
		on revealed primal coordinates.
		Selecting either output, deterministically or by a
		random index, therefore satisfies the definition.
	\end{revision}
\end{proof}

\subsection{OS lower bound}\label{subsec:lb-os}

We construct a polynomial hard instance with $M$ successive
stages, each consisting of $n$ dual coordinates followed
by one primal coordinate.
The dual blocks delay the discovery of successive primal
coordinates, while the Moreau-envelope gradient remains
bounded away from zero whenever the last primal coordinate
is zero.
Scaling this construction yields the NC--C lower bound;
choosing the block length according to the condition number
yields the NC--SC lower bound.

Fix integers $M\ge1$ and $n\ge2$, and write
$y=(y^1,\ldots,y^M)$ with $y^i\in\R^n$.
Define
\begin{equation}\label{eq:lb-checkpoint-domain}
	\mathcal C_M:=\{x\in\R^M:1\ge x_1\ge\cdots\ge x_M\ge0\},
	\qquad
	X:=\mathcal C_M,\qquad
	Y:=\left\{y:\sum_{i=1}^M\|y^i\|_\infty
	\le\frac{\sqrt n}{4}\right\}.
\end{equation}
Projection onto $X$ preserves zero tails.
Indeed, if the input has a zero tail, replacing the
corresponding tail of a feasible point by zero preserves
feasibility and strictly decreases its distance to the
input unless that tail is already zero.
Projection onto $Y$ preserves every zero coordinate by
the same argument, since setting a coordinate to zero
preserves feasibility.

With the conventions $x_0=1$ and $x_{M+1}=0$, define
\begin{align}
	f_{M,n}(x,y)
	:={}&-\frac12\|x\|^2
	-\frac12\sum_{i=1}^M\left[
	\frac{(y_1^i)^2}{n}
	+\sum_{j=1}^{n-1}(y_{j+1}^i-y_j^i)^2\right]\notag\\
	&+\frac1{4\sqrt n}\sum_{i=1}^M(x_{i-1}-x_i)
	[x_{i-1}y_1^i-x_i(y_n^i-y_1^i)].
	\label{eq:lb-os-function}
\end{align}
The construction induces the propagation chain
$y_1^i\longrightarrow\cdots\longrightarrow y_n^i
\longrightarrow x_i\longrightarrow y_1^{i+1}$,
where the last transition applies when $i<M$.
Within each stage, the quadratic differences couple
neighboring dual coordinates.
The primal coordinate $x_i$ cannot be revealed before
$y_n^i$, and the next dual block cannot be revealed before
$x_i$.
The following lemma makes these dependencies precise.

\begin{lemma}\label{lem:lb-os-propagation}
	For the function \eqref{eq:lb-os-function},
	$\nabla_xf_{M,n}(0,0)=0$ and
	$\nabla_yf_{M,n}(0,0)=e_1/(4\sqrt n)$,
	where $e_1$ is the first coordinate vector in $\R^{Mn}$.
	Fix $1\le i\le M$, and suppose that
	$x_\ell=0$ and $y^\ell=0$ for every $\ell>i$.
	Then the following statements hold.
	\begin{enumerate}[label=(\roman*),leftmargin=*]
		\item
		If $x_i=0$ and $y_{j+1}^i=\cdots=y_n^i=0$
		for some $1\le j<n$, then
		$\partial_{y_{j+1}^i}f_{M,n}=y_j^i$
		and $\partial_{y_k^i}f_{M,n}=0$
		for $j+2\le k\le n$.
		
		\item
		If $x_i=0$, then
		$\partial_{x_i}f_{M,n}
		=-x_{i-1}y_n^i/(4\sqrt n)$.
		
		\item
		If $i<M$, then
		$\partial_{y_1^{i+1}}f_{M,n}
		=(x_i^2-x_{i+1}^2)/(4\sqrt n)$.
	\end{enumerate}
\end{lemma}

\begin{proof}
	The coupling term for block $i$ equals
	$[(x_{i-1}^2-x_i^2)y_1^i
	-x_i(x_{i-1}-x_i)y_n^i]/(4\sqrt n)$.
	Differentiating this expression and the quadratic
	dual terms proves the three identities.
	In part~(ii), the contribution from block $i+1$
	vanishes under the zero-tail assumption.
	In part~(iii), the quadratic contribution vanishes
	because $y^{i+1}=0$.
	Evaluating the same derivatives at the origin gives
	the initial gradient identities.
\end{proof}

Consider the coordinate order
$y_1^1,\ldots,y_n^1,x_1;\,
y_1^2,\ldots,y_n^2,x_2;\,\ldots;\,
y_1^M,\ldots,y_n^M,x_M$.
If a query point is supported on a prefix of this order,
its gradient is supported on that prefix and at most
the next coordinate.
This follows from Lemma~\ref{lem:lb-os-propagation};
when the current dual block is entirely zero, only its
first coordinate can be revealed.
The projection properties established above preserve
the corresponding zero tails.
\rev{Consequently, any finite sequence of permitted
	internal projections preserves the same prefix.}
Induction therefore shows that, after $K$ oracle calls,
all revealed coordinates lie in the first $K$ positions.
The same restriction applies to any output satisfying
Definition~\ref{def:lb-zero-respecting}.
In particular, $K<M(n+1)$ implies that the last primal
coordinate of the output is zero.

The extended value function has the explicit expression
\begin{equation}\label{eq:lb-os-value}
	\phi_{M,n}(x)=
	\begin{cases}
		-\dfrac12\|x\|^2+
		\dfrac1{32}\displaystyle\sum_{i=1}^M
		(x_{i-1}-x_i)^2
		\left[x_{i-1}^2+
		\left(1-\dfrac1n\right)x_i^2\right],
		&x\in X,\\[3pt]
		+\infty,&x\notin X.
	\end{cases}
\end{equation}
To verify this formula, define
\begin{equation}\label{eq:lb-os-ystar}
	y_j^{i,*}(x):=
	\frac{(x_{i-1}-x_i)[nx_{i-1}-(j-1)x_i]}{4\sqrt n},
	\qquad 1\le i\le M,\quad 1\le j\le n.
\end{equation}
The negative dual Hessian is positive definite, since
its quadratic form is
$\sum_i[(a_1^i)^2/n+
\sum_{j=1}^{n-1}(a_{j+1}^i-a_j^i)^2]$.
Direct substitution gives
$\nabla_yf_{M,n}(x,y^*(x))=0$.
Moreover, each block of $y^*(x)$ is nonnegative and
nonincreasing, and
$\sum_i\|y^{i,*}(x)\|_\infty
=(\sqrt n/4)\sum_i(x_{i-1}-x_i)x_{i-1}
\le(\sqrt n/4)(1-x_M)\le\sqrt n/4$.
Thus, $y^*(x)\in Y$ is the unique dual maximizer.
Substituting it into $f_{M,n}$ proves
\eqref{eq:lb-os-value}.

The propagation property controls how quickly an algorithm
can reach the end of the chain.
To obtain an OS lower bound, we also need a lower bound
on the Moreau-envelope gradient before that point.
The next proposition establishes this barrier and records
the quantitative properties needed for scaling.

\begin{proposition}[Properties of the OS hard instance]
	\label{prop:lb-os-properties}
	The instance \eqref{eq:lb-os-function} has the following
	properties.
	\begin{enumerate}[label=(\roman*),leftmargin=*]
		\item
		The function $f_{M,n}$ is jointly $5$-smooth on
		$X\times Y$ and $2/[n(3n-1)]$-strongly concave in $y$.
		Moreover, $\diam(Y)=n/2$.
		
		\item
		The value function attains its minimum on $X$, and
		\begin{equation}\label{eq:lb-os-scaled-gap}
			\phi_{M,n}(0)-\min_{x\in X}\phi_{M,n}(x)
			=\frac M2+\frac1{32}.
		\end{equation}
		
		\item
		Let $\Phi_{10}^{M,n}$ denote the Moreau envelope
		of $\phi_{M,n}$
		\rev{with quadratic penalty $5\|u-x\|^2$}.
		For every $x\in X$,
		\begin{equation}\label{eq:lb-os-scaled-barrier}
			x_M=0
			\quad\Longrightarrow\quad
			\|\nabla\Phi_{10}^{M,n}(x)\|>\frac1{16}.
		\end{equation}
	\end{enumerate}
\end{proposition}

\begin{proof}
	\emph{(i)}
	Direct differentiation gives
	\begin{align*}
		\partial_{x_i x_j}^2f_{M,n}
		&=
		\begin{cases}
			-1+\dfrac{-y_1^i+y_n^i+y_1^{i+1}}{2\sqrt n},
			&j=i,\\[3pt]
			-\dfrac{y_n^{\max\{i,j\}}}{4\sqrt n},
			&|i-j|=1,\\[3pt]
			0,&\text{otherwise},
		\end{cases}\\
		\partial_{x_i y_k^j}^2f_{M,n}
		&=
		\begin{cases}
			-\dfrac{x_i}{2\sqrt n},&j=i,\ k=1,\\[3pt]
			\dfrac{2x_i-x_{i-1}}{4\sqrt n},&j=i,\ k=n,\\[3pt]
			\dfrac{x_i}{2\sqrt n},&j=i+1,\ k=1,\\[3pt]
			-\dfrac{x_{i+1}}{4\sqrt n},&j=i+1,\ k=n,\\[3pt]
			0,&\text{otherwise},
		\end{cases}\\
		\nabla^2_{yy}f_{M,n}
		&=-\operatorname{diag}(Q_n,\ldots,Q_n),
	\end{align*}
	where only existing block indices are included and
	$y_1^{M+1}:=0$.
	Here $Q_n\in\R^{n\times n}$ is the symmetric
	tridiagonal matrix with diagonal entries
	$1+1/n,2,\ldots,2,1$ and off-diagonal entries $-1$.
	For $n=2$, its diagonal is $(1+1/n,1)$.
	
	Set $m_i=\|y^i\|_\infty$ and $m_{M+1}=0$.
	Since $\sum_i m_i\le\sqrt n/4$, the primal Hessian
	satisfies
	\begin{align*}
		\|\nabla^2_{xx}f_{M,n}\|_\infty
		&\le\max_i\left(
		1+\frac{2m_i+m_{i+1}}{2\sqrt n}
		+\frac{m_i+m_{i+1}}{4\sqrt n}\right)\\
		&\le1+\frac5{4\sqrt n}\sum_i m_i
		\le\frac{21}{16}.
	\end{align*}
	The inequalities
	$0\le x_{i+1}\le x_i\le x_{i-1}\le1$ also give
	\begin{align*}
		\|\nabla^2_{xy}f_{M,n}\|_\infty
		&\le\max_i
		\frac{5x_i+|2x_i-x_{i-1}|}{4\sqrt n}
		\le\frac3{2\sqrt n},\\
		\|\nabla^2_{xy}f_{M,n}\|_1
		&\le\frac1{4\sqrt n}\max_j
		\{2x_{j-1}+2x_j,\,
		x_j+|2x_j-x_{j-1}|\}
		\le\frac1{\sqrt n}.
	\end{align*}
	Including absent boundary terms in these sums only
	increases the upper bounds.
	
	Direct multiplication verifies
	$(Q_n^{-1})_{jk}=n+\min\{j,k\}-1$.
	Hence $\operatorname{tr}(Q_n^{-1})=n(3n-1)/2$, and
	$\lambda_{\min}(Q_n)
	=1/\lambda_{\max}(Q_n^{-1})
	\ge1/\operatorname{tr}(Q_n^{-1})
	=2/[n(3n-1)]$.
	This proves the claimed dual strong concavity.
	Moreover, $\|Q_n\|_\infty\le4$.
	Symmetry of the full Hessian therefore gives
	\begin{align*}
		\|\nabla^2f_{M,n}\|
		&\le\|\nabla^2f_{M,n}\|_\infty\\
		&\le\max\{
		\|\nabla^2_{xx}f_{M,n}\|_\infty+
		\|\nabla^2_{xy}f_{M,n}\|_\infty,\,
		\|\nabla^2_{xy}f_{M,n}\|_1+
		\|\nabla^2_{yy}f_{M,n}\|_\infty\}\\
		&\le\max\left\{
		\frac{21}{16}+\frac3{2\sqrt n},\,
		4+\frac1{\sqrt n}\right\}\le5.
	\end{align*}
	Since $X\times Y$ is convex, this uniform Hessian
	bound proves joint $5$-smoothness.
	
	For every $y\in Y$,
	$\|y\|\le\sqrt n\sum_i\|y^i\|_\infty\le n/4$.
	Equality is attained by setting all coordinates of
	one block equal to $\sqrt n/4$ and all other blocks
	to zero.
	Central symmetry of $Y$ then gives $\diam(Y)=n/2$.
	
	\emph{(ii)}
	For every $x\in X$, \eqref{eq:lb-os-value} gives
	$\phi_{M,n}(x)\ge-\|x\|^2/2
	\ge-M/2=\phi_{M,n}(\boldsymbol1_M)$.
	Thus, the minimum is attained at $\boldsymbol1_M$
	and equals $-M/2$.
	Since $\phi_{M,n}(0)=1/32$,
	\eqref{eq:lb-os-scaled-gap} follows.
	
	\emph{(iii)}
	Let $F_{M,n}$ denote the polynomial expression for
	$\phi_{M,n}$ on $X$ in \eqref{eq:lb-os-value}.
	We differentiate $F_{M,n}$ rather than the
	extended-valued function at boundary points.
	
	Fix $x\in X$ with $x_M<1$, and let
	$k=\min\{i:x_i<1\}$.
	Then $x_{k-1}=1$.
	Writing $a=1-1/n$, direct differentiation yields
	\begin{align}
		-\partial_{x_k}F_{M,n}(x)
		&\ge x_k+\frac{1-x_k}{16}
		(1-a x_k+2a x_k^2)\notag\\
		&\quad-\frac{x_k-x_{k+1}}{16}
		(2x_k^2-x_kx_{k+1}+a x_{k+1}^2)\notag\\
		&\ge x_k+\frac{(1-x_k)^2}{16}-\frac{x_k^3}{8}
		\ge\frac1{16}+\frac34x_k.
		\label{eq:lb-os-phi-grad}
	\end{align}
	The first inequality is an equality when $k<M$.
	When $k=M$, the contribution from block $k+1$
	is absent, and subtracting the displayed nonnegative
	quantity preserves the inequality.
	The second inequality uses
	$1-a x_k+2a x_k^2\ge1-x_k$ and
	$2x_k^2-x_kx_{k+1}+a x_{k+1}^2\le2x_k^2$.
	The last inequality follows from
	$2x_k+x_k^2-2x_k^3\ge0$.
	
	Now fix $x\in X$ with $x_M=0$, and define
	$u^\star:=\argmin_{u\in X}
	\{F_{M,n}(u)+5\|u-x\|^2\}$.
	Part~(i) implies that $\phi_{M,n}$ is $5$-weakly convex.
	Thus, $u^\star$ is unique and
	$\nabla\Phi_{10}^{M,n}(x)=10(x-u^\star)$.
	If $u^\star_M=1$, then
	$\|\nabla\Phi_{10}^{M,n}(x)\|\ge10$.
	Otherwise, let $k=\min\{i:u^\star_i<1\}$.
	Increasing $u^\star_k$ by a sufficiently small positive
	amount preserves feasibility: its predecessor equals
	$1$, and its successor is no larger than $u^\star_k$.
	First-order optimality in this direction gives
	$\partial_{x_k}F_{M,n}(u^\star)
	+10(u^\star_k-x_k)\ge0$.
	Combining this inequality with
	\eqref{eq:lb-os-phi-grad} yields
	$10(u^\star_k-x_k)
	\ge1/16+3u^\star_k/4>1/16$.
	The strict inequality holds because the preceding
	inequality implies $u^\star_k>x_k\ge0$.
	Consequently,
	$\|\nabla\Phi_{10}^{M,n}(x)\|
	\ge10|u^\star_k-x_k|>1/16$,
	which proves \eqref{eq:lb-os-scaled-barrier}.
\end{proof}

We now scale the instance to satisfy the prescribed
smoothness, dual diameter, and initial-gap bounds.
The block length $n$ determines the delay between
successive primal coordinates, while $M$ determines
the number of stages.
The scaling keeps the stationarity barrier above
$\varepsilon$ and the initial gap below $\Delta$.

\begin{theorem}[NC--C OS complexity lower bound]
	\label{thm:lb-os}
	Fix $L,D_Y,\Delta>0$, and let $\varepsilon$ satisfy
	\begin{equation}\label{eq:lb-os-range}
		0<\varepsilon\le
		\min\left\{\frac{LD_Y}{80},
		\frac{\sqrt{L\Delta}}{240}\right\}.
	\end{equation}
	Choose
	\begin{equation}\label{eq:lb-os-params}
		n=\left\lfloor\frac{LD_Y}{40\varepsilon}\right\rfloor,
		\qquad
		M=\left\lfloor
		\frac{5n^2\Delta}{2LD_Y^2}-\frac1{16}
		\right\rfloor.
	\end{equation}
	For the instance $f_{M,n}$ on $X\times Y$, define
	\begin{align}
		\widetilde f_{M,n}(x,y)
		&=\frac{4LD_Y^2}{5n^2}
		f_{M,n}\left(\frac{nx}{2D_Y},\frac{ny}{2D_Y}\right),
		\notag\\
		\widetilde X&=\frac{2D_Y}{n}X,
		\qquad
		\widetilde Y=\frac{2D_Y}{n}Y.
		\label{eq:lb-os-scaling}
	\end{align}
	Let $\widetilde\phi_{M,n}$ be its value function
	on $\widetilde X$, and let
	$\widetilde\Phi_{2L}^{M,n}$ be the Moreau envelope
	of its extension by $+\infty$ outside $\widetilde X$,
	with quadratic penalty $L\|u-x\|^2$.
	Then $n\ge2$, $M\ge1$, and the following statements hold.
	\begin{enumerate}[label=(\roman*),leftmargin=*]
		\item
		The scaled problem belongs to
		$\mathcal F(L,D_Y,\Delta)$.
		
		\item
		Consider any projected zero-respecting algorithm
		initialized at $(0,0)$.
		If it makes $K<M(n+1)$ oracle calls, then every
		feasible primal output $x^{\rm out}\in\widetilde X$
		satisfying the output restriction in
		Definition~\ref{def:lb-zero-respecting} obeys
		\begin{equation}\label{eq:lb-os-query-barrier}
			\|\nabla\widetilde\Phi_{2L}^{M,n}
			(x^{\rm out})\|>\varepsilon.
		\end{equation}
	\end{enumerate}
	Consequently, the worst-case number of joint gradient
	evaluations required by this algorithm class to produce
	an $\varepsilon$-OS point over
	$\mathcal F(L,D_Y,\Delta)$ is at least
	$M(n+1)\ge L^2D_Y\Delta/(409600\,\varepsilon^3)$.
	In particular, it is
	$\Omega(L^2D_Y\Delta\varepsilon^{-3})$.
\end{theorem}

\begin{proof}
	We first verify the dimension bounds.
	Condition~\eqref{eq:lb-os-range} gives
	$LD_Y/(40\varepsilon)\ge2$, hence
	$n\ge2$ and $n\ge LD_Y/(80\varepsilon)$.
	Set $A:=5n^2\Delta/(2LD_Y^2)$.
	The second accuracy condition yields
	$A\ge L\Delta/(2560\varepsilon^2)\ge45/2$.
	Therefore,
	$M=\lfloor A-1/16\rfloor
	\ge A-17/16\ge A/2\ge1$.
	
	To prove (i), set $s:=2D_Y/n$ and
	$c:=4LD_Y^2/(5n^2)$.
	Then $\widetilde f_{M,n}(x,y)=c f_{M,n}(x/s,y/s)$
	and $c/s^2=L/5$.
	Proposition~\ref{prop:lb-os-properties} gives
	$\|\nabla^2\widetilde f_{M,n}(x,y)\|
	=(L/5)\|\nabla^2f_{M,n}(x/s,y/s)\|\le L$
	and $\diam(\widetilde Y)=sn/2=D_Y$.
	Positive scaling preserves concavity in $y$.
	The scaled objective is polynomial, and the scaled
	feasible sets are nonempty, compact, convex, and
	contain the origin.
	Since $\widetilde\phi_{M,n}(x)=c\phi_{M,n}(x/s)$
	on $\widetilde X$, its initial gap is
	$c(M/2+1/32)=2LD_Y^2(M+1/16)/(5n^2)\le\Delta$,
	where the inequality follows from
	\eqref{eq:lb-os-params}.
	Thus, the scaled problem belongs to
	$\mathcal F(L,D_Y,\Delta)$.
	
	To prove (ii), fix an admissible algorithm and an
	output produced after $K<M(n+1)$ oracle calls.
	Positive coordinate scaling preserves supports,
	and the scaled gradient is a positive scalar
	multiple of the original gradient at the corresponding
	point.
	Moreover, $\Pi_{sC}(u)=s\Pi_C(u/s)$ for every
	nonempty closed convex set $C$.
	The propagation and projection properties of the
	original construction therefore remain valid for
	the scaled instance.
	Together with the output restriction, they imply
	$x_M^{\rm out}=0$.
	
	A change of variables in the envelope definition,
	using $Ls^2/c=5$, gives
	$\widetilde\Phi_{2L}^{M,n}(x)
	=c\Phi_{10}^{M,n}(x/s)$.
	Consequently,
	$\|\nabla\widetilde\Phi_{2L}^{M,n}(x^{\rm out})\|
	=(c/s)\|\nabla\Phi_{10}^{M,n}(x^{\rm out}/s)\|$.
	Since $x^{\rm out}/s\in X$ and its $M$th coordinate
	vanishes, \eqref{eq:lb-os-scaled-barrier} shows that
	this quantity is strictly greater than
	$LD_Y/(40n)\ge\varepsilon$.
	The last inequality uses
	$n\le LD_Y/(40\varepsilon)$.
	This proves \eqref{eq:lb-os-query-barrier}.
	
	Finally, the bounds $M\ge A/2$ and
	$n\ge LD_Y/(80\varepsilon)$ yield
	$M(n+1)\ge5n^3\Delta/(4LD_Y^2)
	\ge L^2D_Y\Delta/(409600\,\varepsilon^3)$.
	The claimed worst-case oracle lower bound follows.
\end{proof}

The same hard instance is already strongly concave in $y$.
Choosing $n$ according to the condition number preserves
the propagation mechanism while ensuring the required
dual curvature.

\begin{theorem}[NC--SC OS complexity lower bound]
	\label{thm:lb-ncsc-os}
	Fix $L,D_Y,\Delta>0$ and $0<\mu\le L$, and set
	$\kappa:=L/\mu$.
	Let $\varepsilon$ satisfy
	\begin{equation}\label{eq:lb-ncsc-os-range}
		0<\varepsilon\le
		\min\left\{
		\frac{\sqrt{L\Delta}}{130},
		\frac{12LD_Y}{85(1+\sqrt{1+24\kappa/5})}
		\right\}.
	\end{equation}
	For $\kappa\ge25$, choose
	\begin{equation}\label{eq:lb-ncsc-os-params}
		n=\left\lfloor
		\frac{1+\sqrt{1+24\kappa/5}}6\right\rfloor,
		\qquad
		M=\left\lfloor
		\frac{2L\Delta}{1445\varepsilon^2}-\frac1{16}
		\right\rfloor,
	\end{equation}
	and define
	\begin{align}
		\widetilde f_{M,n}^{\mathrm{SC}}(x,y)
		&=\frac{1445\varepsilon^2}{L}
		f_{M,n}\left(\frac{Lx}{85\varepsilon},
		\frac{Ly}{85\varepsilon}\right),\notag\\
		\widetilde X&=\frac{85\varepsilon}{L}X,
		\qquad
		\widetilde Y=\frac{85\varepsilon}{L}Y.
		\label{eq:lb-ncsc-os-scaling}
	\end{align}
	Let $\widetilde\phi_{M,n}$ be the corresponding value
	function on $\widetilde X$, and let
	$\widetilde\Phi_{2L}^{M,n}$ be the Moreau envelope
	of its extension by $+\infty$ outside $\widetilde X$,
	with quadratic penalty $L\|u-x\|^2$.
	Then $n\ge2$, $M\ge1$, and the following statements hold.
	\begin{enumerate}[label=(\roman*),leftmargin=*]
		\item
		The scaled problem belongs to
		$\mathcal F_\mu(L,D_Y,\Delta)$.
		
		\item
		Consider any projected zero-respecting algorithm
		initialized at $(0,0)$.
		If it makes $K<M(n+1)$ oracle calls, then every
		feasible primal output $x^{\rm out}\in\widetilde X$
		satisfying the output restriction in
		Definition~\ref{def:lb-zero-respecting} obeys
		\begin{equation}\label{eq:lb-ncsc-os-query-barrier}
			\|\nabla\widetilde\Phi_{2L}^{M,n}
			(x^{\rm out})\|>\varepsilon.
		\end{equation}
	\end{enumerate}
	For $\kappa\ge25$, the worst-case number of joint
	gradient evaluations required by this algorithm class
	to produce an $\varepsilon$-OS point over
	$\mathcal F_\mu(L,D_Y,\Delta)$ is therefore at least
	$M(n+1)\ge
	\sqrt\kappa\,L\Delta/(1445\sqrt{30}\,\varepsilon^2)$.
	For $1\le\kappa<25$, a reduction from nonconvex
	minimization gives the same asymptotic order.
	Thus, for every $\kappa\ge1$, the worst-case oracle
	complexity is
	$\Omega(\sqrt\kappa\,L\Delta\varepsilon^{-2})$.
\end{theorem}

\begin{proof}
	First suppose $\kappa\ge25$, and set
	$r:=(1+\sqrt{1+24\kappa/5})/6$.
	Then $r\ge2$, so $n=\lfloor r\rfloor\ge2$
	and $n\ge r/2$.
	Since $r(3r-1)=2\kappa/5$ and
	$t\mapsto t(3t-1)$ is increasing for $t\ge2$,
	we have $n(3n-1)\le2\kappa/5$.
	Set $B:=2L\Delta/(1445\varepsilon^2)$.
	Condition~\eqref{eq:lb-ncsc-os-range} gives
	$B\ge33800/1445$, and consequently
	$M=\lfloor B-1/16\rfloor
	\ge B-17/16\ge B/2\ge1$.
	
	To prove (i), set $s:=85\varepsilon/L$
	and $c:=1445\varepsilon^2/L$.
	Then $\widetilde f_{M,n}^{\mathrm{SC}}(x,y)
	=c f_{M,n}(x/s,y/s)$ and $c/s^2=L/5$.
	By Proposition~\ref{prop:lb-os-properties},
	the joint Hessian has norm at most $L$, while
	$\nabla^2_{yy}\widetilde f_{M,n}^{\mathrm{SC}}(x,y)
	\preceq-[2L/(5n(3n-1))]I_{Mn}
	\preceq-\mu I_{Mn}$.
	Moreover,
	$\diam(\widetilde Y)=85n\varepsilon/(2L)
	\le85r\varepsilon/(2L)\le D_Y$,
	where the last inequality follows from
	\eqref{eq:lb-ncsc-os-range}.
	The scaled value function satisfies
	$\widetilde\phi_{M,n}(x)=c\phi_{M,n}(x/s)$
	on $\widetilde X$, so its initial gap equals
	$c(M/2+1/32)
	=1445\varepsilon^2(M+1/16)/(2L)\le\Delta$.
	The remaining class conditions follow from the
	polynomial objective and the nonempty compact convex
	feasible sets containing the origin.
	This proves (i).
	
	To prove (ii), fix an admissible algorithm and an
	output produced after $K<M(n+1)$ oracle calls.
	As in the proof of Theorem~\ref{thm:lb-os},
	positive scaling preserves gradient supports and
	the projection propagation properties.
	The output restriction therefore gives
	$x_M^{\rm out}=0$.
	Since $Ls^2/c=5$, a change of variables in the
	envelope definition yields
	$\widetilde\Phi_{2L}^{M,n}(x)
	=c\Phi_{10}^{M,n}(x/s)$.
	Using $c/s=17\varepsilon$, we obtain
	$\|\nabla\widetilde\Phi_{2L}^{M,n}(x^{\rm out})\|
	=17\varepsilon
	\|\nabla\Phi_{10}^{M,n}(x^{\rm out}/s)\|$.
	Because $x^{\rm out}/s\in X$ and its $M$th coordinate
	vanishes, \eqref{eq:lb-os-scaled-barrier} makes this
	quantity strictly greater than
	$17\varepsilon/16>\varepsilon$.
	This proves \eqref{eq:lb-ncsc-os-query-barrier}.
	
	Finally, $n\ge r/2\ge\sqrt{\kappa/30}$ and
	$M\ge B/2=L\Delta/(1445\varepsilon^2)$ imply
	$M(n+1)\ge
	\sqrt\kappa\,L\Delta/(1445\sqrt{30}\,\varepsilon^2)$.
	This establishes the claimed lower bound for
	$\kappa\ge25$.
	
	For $1\le\kappa<25$, take the $L$-smooth nonconvex
	minimization hard instance $g:\R^d\to\R$ from
	\cite[Theorem~1]{carmon2020lower}, scaled so that
	$g(0)-\inf g\le\Delta$, with gradient accuracy
	$3\varepsilon/2$.
	Define $\widehat f(x,y):=g(x)-\mu y^2/2$
	on $\widehat X:=\R^d$ and
	$\widehat Y:=[-D_Y/2,D_Y/2]$.
	Its joint gradient Lipschitz constant is at most
	$\max\{L,\mu\}=L$, and it is $\mu$-strongly
	concave in $y$.
	Its value function is $g$, so the initial gap is
	at most $\Delta$.
	Thus, the problem belongs to
	$\mathcal F_\mu(L,D_Y,\Delta)$.
	
	Each joint gradient evaluation can be simulated by
	one gradient evaluation of $g$.
	Projection onto $\widehat X$ is the identity, and
	$N_{\widehat X}(x)=\{0\}$.
	The dual variable is decoupled from the primal
	variable and cannot reveal additional primal
	coordinates.
	Hence, the primal queries and outputs satisfy the
	zero-respecting restrictions for nonconvex
	minimization.
	
	Let $\widehat\Phi_{2L}$ be the Moreau envelope of $g$.
	For any $x$, the function
	$u\mapsto g(u)+L\|u-x\|^2$ is $L$-strongly convex
	and has a unique minimizer $u^\star$.
	Its optimality condition and the envelope gradient
	identity give
	$\nabla g(u^\star)=2L(x-u^\star)
	=\nabla\widehat\Phi_{2L}(x)$.
	By $L$-smoothness,
	$\|\nabla g(x)\|
	\le\|\nabla g(u^\star)\|+L\|x-u^\star\|
	=(3/2)\|\nabla\widehat\Phi_{2L}(x)\|$.
	Every $\varepsilon$-OS output therefore satisfies
	$\|\nabla g(x)\|\le3\varepsilon/2$.
	The nonconvex minimization lower bound gives
	$\Omega(L\Delta\varepsilon^{-2})$ joint gradient
	evaluations.
	Since $1\le\sqrt\kappa<5$, this is also
	$\Omega(\sqrt\kappa\,L\Delta\varepsilon^{-2})$,
	with a universal implied constant.
\end{proof}

\rev{Theorems~\ref{thm:lb-os}
	and~\ref{thm:lb-ncsc-os} have the same parameter dependence
	as the leading OS terms obtained with one-time warm-up.
	This comparison is within the stated oracle model and
	the intersection of the upper- and lower-bound accuracy
	ranges.
	Allowing nested iterations or additional internal
	projections in this model does not bypass the
	coordinate-propagation barrier.}

\subsection{GS lower bound}
\label{subsec:lb-gs-result}\label{subsec:lb-gs}

For GS, we need a lower bound on the joint constrained
residual at every feasible primal--dual pair.
We construct an alternating chain of primal and dual
coordinates whose residual remains bounded away from zero
until the last primal stage is revealed.

Fix integers $M\ge1$ and $n\ge2$.
Let $p,v,z\in\R^M$ and
$q=(q^1,\ldots,q^M)$, $y=(y^1,\ldots,y^M)$, where
$q^i,y^i\in\R^n$.
Define
\begin{align}
	X&:=\left\{(p,q,v):
	p\in\mathcal C_M,\quad 0\le v_i\le1,\quad
	v_i\le q_1^i\ (1\le i\le M),\quad
	\sum_{i=1}^M(q_1^i)^2\le n^2\right\},\notag\\
	Y&:=\{(y,z):z\ge0,\ \|(y,z)\|\le2n^2\}.
	\label{eq:lb-gs-domains}
\end{align}
With the conventions $p_0=1$ and $v_{M+1}=0$, define
\begin{align}
	f_{M,n}^{\mathrm{GS}}(p,q,v;y,z)
	:={}&-\|p\|^2-\frac14\sum_{i=1}^M(q_1^i)^2
	-\frac12\sum_{i=1}^M(p_{i-1}^2-p_i^2)v_i
	-\frac18\langle p,z\rangle\notag\\
	&+2\sum_{i=1}^M\left[
	\frac{q_1^iy_1^i}{\sqrt n}
	+\sum_{j=1}^{n-1}q_{j+1}^i(y_{j+1}^i-y_j^i)
	\right]\notag\\
	&-\frac1n\sum_{i=1}^M(y_n^i)^2
	-\frac1{4n^2}\|y\|^2-\frac12\|z\|^2
	+\frac1{2\sqrt n}\sum_{i=1}^M y_n^iz_i.
	\label{eq:lb-gs-function}
\end{align}
The propagation order within and between stages is
\[
(v_i,q_1^i)\longrightarrow y_1^i
\longrightarrow q_2^i\longrightarrow y_2^i
\longrightarrow\cdots\longrightarrow q_n^i
\longrightarrow y_n^i\longrightarrow z_i
\longrightarrow p_i\longrightarrow(v_{i+1},q_1^{i+1}),
\]
where the last transition applies when $i<M$.
The pair $(v_i,q_1^i)$ is treated as one group because
projection can reveal $q_1^i$ when $v_i$ is revealed.

\begin{revision}
	The variables have separate roles in the construction.
	The entrance pair $(v_i,q_1^i)$ transfers activation from
	the preceding stage while remaining compatible with
	projection.
	The alternating $q$ and $y$ coordinates delay information
	propagation.
	The terminal variable $z_i$ transfers the dual signal to
	the primal coordinate $p_i$, which activates the next stage.
	The two norm constraints bound the entrance coordinates
	and the dual domain.
	In the residual-barrier proof, a sufficiently small
	residual forces these norm constraints to be inactive,
	preventing their normal vectors from canceling the
	terminal signal.
\end{revision}

The projections preserve prefixes of this order.
For the $p$ coordinates, this follows from the zero-tail
property of $\Pi_{\mathcal C_M}$.
The coordinates $q_j^i$ with $j\ge2$ are unconstrained.
For an unrevealed entrance group, setting both $v_i$ and
$q_1^i$ to zero preserves feasibility and decreases the
projection objective unless both are already zero.
Finally, projection onto $Y$ preserves zero coordinates.
Thus, projections cannot bypass an unrevealed stage.
\rev{These statements apply to every internal projection,
	so they remain valid for any finite sequence of free
	projections between oracle calls.}

\begin{lemma}\label{lem:lb-gs-propagation}
	At the origin, the only nonzero gradient coordinate
	of $f_{M,n}^{\mathrm{GS}}$ is
	$\partial_{v_1}f_{M,n}^{\mathrm{GS}}=-1/2$.
	Fix $1\le i\le M$, and suppose that
	$(p_\ell,q^\ell,v_\ell,y^\ell,z_\ell)=0$
	for every $\ell>i$.
	Then the following statements hold.
	\begin{enumerate}[label=(\roman*),leftmargin=*]
		\item
		If $p_i=z_i=0$,
		$q_{j+1}^i=\cdots=q_n^i=0$, and
		$y_j^i=\cdots=y_n^i=0$ for some $1\le j\le n$,
		then
		$\partial_{y_1^i}f_{M,n}^{\mathrm{GS}}
		=2q_1^i/\sqrt n$ when $j=1$, and
		$\partial_{y_j^i}f_{M,n}^{\mathrm{GS}}
		=2q_j^i$ when $j\ge2$.
		If $j<n$, $p_i=z_i=0$,
		$q_{j+1}^i=\cdots=q_n^i=0$, and
		$y_{j+1}^i=\cdots=y_n^i=0$, then
		$\partial_{q_{j+1}^i}f_{M,n}^{\mathrm{GS}}
		=-2y_j^i$.
		In either case, all derivatives beyond the newly
		revealed coordinate in the propagation order vanish.
		
		\item
		If $p_i=z_i=0$, then
		$\partial_{z_i}f_{M,n}^{\mathrm{GS}}
		=y_n^i/(2\sqrt n)$.
		If $p_i=0$, then
		$\partial_{p_i}f_{M,n}^{\mathrm{GS}}=-z_i/8$.
		
		\item
		If $i<M$, then
		$\partial_{v_{i+1}}f_{M,n}^{\mathrm{GS}}
		=-(p_i^2-p_{i+1}^2)/2$.
	\end{enumerate}
\end{lemma}

\begin{proof}
	Each identity follows by differentiating
	\eqref{eq:lb-gs-function}.
	The bilinear terms couple only consecutive $q$ and
	$y$ coordinates within a block.
	The only coupling from $y_n^i$ to the remaining
	variables is through $z_i$, and the only nonzero
	contribution to $\partial_{p_i}f_{M,n}^{\mathrm{GS}}$
	when $p_i=0$ is $-z_i/8$.
	The zero-tail assumptions therefore also imply the
	stated vanishing of subsequent derivatives.
	Evaluating the derivatives at the origin gives the
	initial gradient identity.
\end{proof}

By induction, each oracle call reveals at most one new
group in the propagation order.
There are $2n+2$ groups per stage, including the entrance
pair $(v_i,q_1^i)$.
In particular, the weaker condition $K<M(2n+1)$ suffices
to ensure that $p_M=0$ at every admissible output after
$K$ oracle calls.
We use this conservative count below.

Define the extended value function
\begin{equation}\label{eq:lb-gs-value}
	\phi_{M,n}^{\mathrm{GS}}(p,q,v):=
	\begin{cases}
		\displaystyle\max_{(y,z)\in Y}
		f_{M,n}^{\mathrm{GS}}(p,q,v;y,z),&(p,q,v)\in X,\\[3pt]
		+\infty,&(p,q,v)\notin X.
	\end{cases}
\end{equation}
For $(p,q,v)\in X$, write
\begin{equation}\label{eq:lb-gs-ystar}
	(y^*(p,q,v),z^*(p,q,v))
	:=\arg\max_{(y,z)\in Y}
	f_{M,n}^{\mathrm{GS}}(p,q,v;y,z).
\end{equation}
Existence follows from compactness of $Y$.
Uniqueness follows from the dual strong concavity proved
below.
Unlike the OS barrier, the GS barrier must hold for every
feasible dual point, not only the dual maximizer.

\begin{proposition}[Properties of the GS hard instance]
	\label{prop:lb-gs-hard}
	For integers $M\ge1$ and $n\ge2$, the instance
	\eqref{eq:lb-gs-function} on $X\times Y$ has the
	following properties.
	\begin{enumerate}[label=(\roman*),leftmargin=*]
		\item
		The joint gradient of $f_{M,n}^{\mathrm{GS}}$
		is Lipschitz continuous on $X\times Y$ with
		constant at most $33/8$.
		For each fixed $(p,q,v)\in X$, the function
		$(y,z)\mapsto f_{M,n}^{\mathrm{GS}}(p,q,v;y,z)$
		is $1/(2n^2)$-strongly concave on $Y$.
		Moreover, $\diam(Y)=4n^2$.
		
		\item
		The value function $\phi_{M,n}^{\mathrm{GS}}$
		attains a finite minimum on $X$, and
		\begin{equation}\label{eq:lb-gs-gap}
			\phi_{M,n}^{\mathrm{GS}}(0,0,0)
			-\min_{(p,q,v)\in X}
			\phi_{M,n}^{\mathrm{GS}}(p,q,v)=M.
		\end{equation}
		
		\item
		Every feasible point
		$(p,q,v;y,z)\in X\times Y$ with $p_M=0$
		satisfies
		\begin{equation}\label{eq:lb-gs-barrier}
			R(p,q,v;y,z)>\frac1{80}.
		\end{equation}
	\end{enumerate}
\end{proposition}

\begin{proof}
	\emph{(i)}
	Write $f=f_{M,n}^{\mathrm{GS}}$.
	The nonzero Hessian entries, up to symmetry, are
	\begin{align*}
		\partial_{p_i p_i}^2f
		&=-2+v_i-v_{i+1},
		&
		\partial_{p_i v_i}^2f&=p_i,
		&
		\partial_{p_i v_{i+1}}^2f&=-p_i,\\
		\partial_{q_1^i q_1^i}^2f
		&=-\frac12,
		&
		\partial_{p_i z_i}^2f&=-\frac18,
		&
		\partial_{z_i z_i}^2f&=-1,\\
		\partial_{q_1^i y_1^i}^2f
		&=\frac2{\sqrt n},
		&
		\partial_{q_j^i y_j^i}^2f&=2,
		&
		\partial_{q_j^i y_{j-1}^i}^2f&=-2
		\quad(2\le j\le n),\\
		\partial_{y_j^i y_j^i}^2f
		&=-\frac1{2n^2}\quad(j<n),
		&
		\partial_{y_n^i y_n^i}^2f
		&=-\frac2n-\frac1{2n^2},
		&
		\partial_{y_n^i z_i}^2f&=\frac1{2\sqrt n}.
	\end{align*}
	Only existing indices are included.
	For reference, the $qy$ block is block diagonal
	with blocks $2P_n$, where
	\begin{equation}\label{eq:lb-gs-P}
		P_n=
		\begin{pmatrix}
			1/\sqrt n&0&\cdots&0\\
			-1&1&\ddots&\vdots\\
			&\ddots&\ddots&0\\
			0&\cdots&-1&1
		\end{pmatrix}.
	\end{equation}
	In particular, $P_n^\top P_n=Q_n$ and $\|P_n\|\le2$.
	To bound the full Hessian, assign weight $1/2$
	to each $v$ coordinate and weight $1$ to every
	other coordinate.
	The absolute weighted row sums are at most $33/8$
	for the $p$ rows and at most $4$ for the $v$ and
	$q$ rows.
	For the $y$ rows they are at most $4+1/(2n^2)$:
	the terminal row satisfies
	$2+2/n+1/(2n^2)+1/(2\sqrt n)\le4+1/(2n^2)$.
	The $z$ rows have sum at most
	$1+1/8+1/(2\sqrt n)$.
	Since the Hessian is symmetric, its spectral norm
	equals its spectral radius, which is bounded by any
	induced norm of a diagonally similar matrix.
	Consequently,
	$\|\nabla^2f\|
	\le\max\{33/8,4,4+1/(2n^2),
	1+1/8+1/(2\sqrt n)\}\le33/8$.
	Convexity of $X\times Y$ then gives joint smoothness.
	
	All nonterminal dual coordinates have negative
	curvature $1/(2n^2)$.
	For each terminal pair $(y_n^i,z_i)$, the negative
	dual Hessian satisfies
	\[
	\begin{pmatrix}
		2/n+1/(2n^2)&-1/(2\sqrt n)\\
		-1/(2\sqrt n)&1
	\end{pmatrix}
	\succeq\frac1{2n^2}I.
	\]
	Indeed, after subtracting $I/(2n^2)$, both diagonal
	entries are positive and the determinant equals
	$n^{-1}(7/4-1/n^2)>0$.
	This proves dual strong concavity.
	Finally, the radius bound gives $\diam(Y)\le4n^2$,
	with equality at $(y,z)=(\pm2n^2e,0)$ for any unit
	vector $e\in\R^{Mn}$.
	
	\emph{(ii)}
	For any $(p,q,v)\in X$, the point
	$\bar y_j^i=\sqrt n\,q_1^i$, $\bar z=0$ is feasible,
	because
	$\|\bar y\|=n(\sum_i(q_1^i)^2)^{1/2}\le n^2$.
	Evaluating $f$ at this point gives
	\begin{align*}
		\phi_{M,n}^{\mathrm{GS}}(p,q,v)
		&\ge-\|p\|^2+\frac12\sum_i(q_1^i)^2
		-\frac12\sum_i(p_{i-1}^2-p_i^2)v_i\\
		&\ge-\|p\|^2-\frac{1-p_M^2}{2}
		\ge-M.
	\end{align*}
	The last inequality uses
	$\|p\|^2\le M-1+p_M^2$.
	Equality is attained at
	$(p,q,v)=(\boldsymbol1_M,0,0)$.
	At this point and at the origin, the dual maximizer
	is $(y,z)=0$, so the corresponding values are
	$-M$ and $0$.
	This proves \eqref{eq:lb-gs-gap}.
	
	\emph{(iii)}
	Suppose $p_M=0$, and abbreviate
	$R=R(p,q,v;y,z)$.
	Let $k=\min\{i:p_i<1\}$.
	Then $p_{k-1}=1$, and increasing $p_k$ slightly
	preserves feasibility.
	Every normal vector to $\mathcal C_M$ therefore
	has a nonpositive $k$-th component.
	Since $z_k\ge0$, it follows that
	\begin{equation}\label{eq:lb-gs-q-gradient}
		R\ge-\partial_{p_k}f
		=2p_k-p_k(v_k-v_{k+1})+\frac{z_k}{8}
		\ge p_k+\frac{z_k}{8}.
	\end{equation}
	Assume, for a contradiction, that $R\le1/80$.
	
	\rev{\emph{Step 1: normal representation and recovery
			of the dual chain.}}
	Choose normal vectors attaining the distances
	defining $R$.
	The constraint sets satisfy Slater's condition,
	so their normal vectors admit the usual multiplier
	representation.
	Write $a_i=q_1^i$ and $d_i=p_{i-1}^2-p_i^2$.
	There exist $\lambda,\zeta\ge0$, $\theta_i\ge0$,
	and $n_p\in N_{\mathcal C_M}(p)$ such that
	$\lambda(\|a\|^2-n^2)=0$,
	$\zeta(\|(y,z)\|^2-4n^4)=0$, and
	$\theta_i(v_i-a_i)=0$.
	The factors of $2$ in the quadratic constraints
	are absorbed into $\lambda$ and $\zeta$.
	Using the ascent sign for the dual residual blocks,
	write
	\begin{align*}
		\rho_p&=\nabla_p f+n_p,
		&
		\rho_q^i&=\nabla_{q^i}f+
		(\lambda a_i-\theta_i)e_1,\\
		\rho_y&=\nabla_y f-\zeta y,
		&
		\rho_z&\in\nabla_z f-\zeta z-N_{\R_+^M}(z),\\
		(\rho_v)_i&\in-d_i/2+\theta_i+N_{[0,1]}(v_i).
	\end{align*}
	These blocks satisfy
	\begin{equation}\label{eq:lb-gs-joint-residual}
		\|\rho_p\|^2+\|\rho_q\|^2+\|\rho_v\|^2
		+\|\rho_y\|^2+\|\rho_z\|^2=R^2.
	\end{equation}
	The nonentrance $q$ equations give
	\begin{equation}\label{eq:lb-gs-y-recovery}
		y_j^i=y_n^i-\frac12
		\sum_{\ell=j+1}^n(\rho_q^i)_\ell,
		\qquad 1\le j\le n.
	\end{equation}
	Set $t_i=y_n^i/\sqrt n$ and $S=\|t\|$.
	Summing the $y$ equations within each block and
	using \eqref{eq:lb-gs-y-recovery} gives
	\begin{equation}\label{eq:lb-gs-normal-equation}
		\left\|a-\left(\frac54+\frac{\zeta n^2}{2}\right)t
		+\frac z4\right\|
		\le\left(\frac n2+\frac18+
		\frac{\zeta n^2}{4}\right)R.
	\end{equation}
	\begin{revision}
		To verify the normalization, define
		$E_i:=\sum_{j=1}^n(y_j^i-y_n^i)$.
		Equation~\eqref{eq:lb-gs-y-recovery} gives
		$E_i=-\frac12\sum_{\ell=2}^n
		(\ell-1)(\rho_q^i)_\ell$ and
		$\|E\|\le n^{3/2}R/2$.
		The summed $y$ equations, multiplied by $\sqrt n/2$,
		become
		\[
		a-\left(\frac54+\frac{\zeta n^2}{2}\right)t+\frac z4
		=\frac{\sqrt n}{2}
		\left(\sum_{j=1}^n(\rho_y^i)_j\right)_{i=1}^M
		+\frac{\sqrt n}{2}
		\left(\frac1{2n^2}+\zeta\right)E.
		\]
		The first term has norm at most $nR/2$, and the
		second has norm at most
		$(1/8+\zeta n^2/4)R$.
		This proves \eqref{eq:lb-gs-normal-equation}.
	\end{revision}
	
	\rev{\emph{Step 2: inactivity of the dual norm constraint.}}
	For $z_i>0$, the nonnegativity normal vanishes, and
	$(1+\zeta)z_i=t_i/2-p_i/8-(\rho_z)_i$.
	For $z_i=0$, the corresponding upper bound is
	automatic.
	Thus,
	\begin{equation}\label{eq:lb-gs-z-bound}
		\|z\|\le\frac S2+R.
	\end{equation}
	Combining this with
	\eqref{eq:lb-gs-normal-equation} and $\|a\|\le n$
	yields
	$(9/8+\zeta n^2/2)S
	\le n+(n/2+3/8+\zeta n^2/4)R$.
	For $n\ge2$ and $R\le1/80$, the right-hand side
	is strictly smaller than
	$n(9/8+\zeta n^2/2)$, so $S<n$.
	Equation~\eqref{eq:lb-gs-y-recovery} also gives
	$\|y\|\le nS+nR$.
	Consequently,
	$\|(y,z)\|\le\|y\|+\|z\|
	\le(n+1/2)S+(n+1)R
	<n^2+n/2+(n+1)/80<2n^2$.
	Complementarity therefore implies $\zeta=0$.
	
	\rev{\emph{Step 3: inactivity of the entrance norm constraint.}}
	The $v$ equations and
	$\langle\nu,v\rangle\ge0$
	for $\nu\in N_{[0,1]^M}(v)$ imply
	$\langle a,\theta\rangle
	=\langle v,\theta\rangle
	\le\frac12\sum_i d_iv_i+\langle v,\rho_v\rangle
	\le1/2+nR$.
	We used $\sum_i d_i=1$, $0\le v_i\le1$, and
	$\|v\|\le\|a\|\le n$.
	Define $e=a-\frac54t+\frac14z$.
	By \eqref{eq:lb-gs-normal-equation},
	$\|e\|\le(n/2+1/8)R$.
	The entrance $q$ equations and
	\eqref{eq:lb-gs-y-recovery} give
	$\|(\lambda-\frac12)a+2t-\theta\|\le2R$.
	\rev{Indeed,
		$(\lambda-\frac12)a_i+2t_i-\theta_i
		=(\rho_q^i)_1+
		n^{-1/2}\sum_{\ell=2}^n(\rho_q^i)_\ell$,
		whose blockwise operator norm is at most $\sqrt2<2$.}
	Substituting
	$t=\frac45a+\frac15z-\frac45e$ yields
	$\|(\lambda+\frac{11}{10})a+\frac25z-\theta\|
	\le(\frac{4n}{5}+\frac{11}{5})R$.
	Taking the inner product with $a$, and using
	$\lambda\ge0$ and $a,z\ge0$, gives
	\[
	\frac{11}{10}\|a\|^2
	\le\frac12+nR+
	n\left(\frac{4n}{5}+\frac{11}{5}\right)R
	\le\frac12+4n^2R<n^2.
	\]
	Thus, $\|a\|<n$, and complementarity implies
	$\lambda=0$.
	
	\rev{\emph{Step 4: entrance forcing and the terminal barrier.}}
	The entrance $q$ equations now give
	$t_i\ge a_i/4+\theta_i/2-R$.
	If $a_i<1$, then $v_i<1$ and the $v$ normal
	condition implies $\theta_i\ge d_i/2-R$.
	If $a_i\ge1$, then $a_i/4\ge d_i/4$ because
	$0\le d_i\le1$.
	In either case,
	\begin{equation}\label{eq:lb-gs-entrance-floor}
		\frac{y_n^i}{\sqrt n}
		\ge\frac{a_i}{4}+\frac{\theta_i}{2}-R
		\ge\frac{d_i}{4}-\frac32R.
	\end{equation}
	Since $\zeta=0$, the $z_k$ normal condition gives
	$z_k\ge y_n^k/(2\sqrt n)-p_k/8-R
	\ge(1-p_k^2-p_k)/8-7R/4$.
	Combining this with \eqref{eq:lb-gs-q-gradient}
	yields
	$R\ge p_k+z_k/8
	\ge(63p_k+1-p_k^2)/64-7R/32$,
	and thus
	$R\ge(63p_k+1-p_k^2)/78\ge1/78>1/80$.
	This contradicts $R\le1/80$ and proves
	\eqref{eq:lb-gs-barrier}.
\end{proof}

We now scale the instance.
Its dual diameter grows as $n^2$, whereas the unscaled
GS barrier is independent of $n$.
After fixing $L$ and $D_Y$, the residual scale is therefore
$LD_Y/n^2$, leading to
$n=\Theta(\sqrt{LD_Y/\varepsilon})$.
Together with $M=\Theta(L\Delta/\varepsilon^2)$,
this gives the $\varepsilon^{-5/2}$ lower bound.

\begin{theorem}[NC--C GS complexity lower bound]
	\label{thm:lb-gs}
	Fix $L,D_Y,\Delta>0$, and let $\varepsilon$ satisfy
	\begin{equation}\label{eq:lb-gs-range}
		0<\varepsilon\le
		\min\left\{\frac{LD_Y}{7040},
		\frac{\sqrt{L\Delta}}{90\sqrt{22}}\right\}.
	\end{equation}
	Choose
	\begin{equation}\label{eq:lb-gs-params}
		n=\left\lfloor
		\sqrt{\frac{LD_Y}{1760\varepsilon}}\right\rfloor,
		\qquad
		M=\left\lfloor\frac{88n^4\Delta}{LD_Y^2}\right\rfloor.
	\end{equation}
	For the instance $f_{M,n}^{\mathrm{GS}}$ on
	$X\times Y$, define
	\begin{align}
		\widetilde f_{M,n}^{\mathrm{GS}}(p,q,v;y,z)
		&=\frac{LD_Y^2}{88n^4}
		f_{M,n}^{\mathrm{GS}}\left(
		\frac{4n^2}{D_Y}(p,q,v);
		\frac{4n^2}{D_Y}(y,z)\right),\notag\\
		\widetilde X&=\frac{D_Y}{4n^2}X,
		\qquad
		\widetilde Y=\frac{D_Y}{4n^2}Y.
		\label{eq:lb-gs-scaling}
	\end{align}
	Let $\widetilde R$ denote the constrained GS residual
	of the scaled problem.
	Then $n\ge2$, $M\ge1$, and the following statements hold.
	\begin{enumerate}[label=(\roman*),leftmargin=*]
		\item
		The scaled problem belongs to
		$\mathcal F(L,D_Y,\Delta)$.
		
		\item
		Consider any projected zero-respecting algorithm
		initialized at the origin.
		If it makes $K<M(2n+1)$ oracle calls, then every
		feasible output
		$U^{\rm out}=(p^{\rm out},q^{\rm out},v^{\rm out};
		y^{\rm out},z^{\rm out})
		\in\widetilde X\times\widetilde Y$
		satisfying the output restriction in
		Definition~\ref{def:lb-zero-respecting} obeys
		\begin{equation}\label{eq:lb-gs-query-barrier}
			\widetilde R(U^{\rm out})>\varepsilon.
		\end{equation}
	\end{enumerate}
	Consequently, the worst-case number of joint gradient
	evaluations required by this algorithm class to produce
	an $\varepsilon$-GS point over
	$\mathcal F(L,D_Y,\Delta)$ is at least
	$M(2n+1)\ge
	(88/3960^{5/2})L^{3/2}D_Y^{1/2}
	\Delta\varepsilon^{-5/2}$.
	In particular, it is
	$\Omega(L^{3/2}D_Y^{1/2}\Delta\varepsilon^{-5/2})$.
\end{theorem}

\begin{proof}
	We first verify the dimension bounds.
	By \eqref{eq:lb-gs-range},
	$t:=\sqrt{LD_Y/(1760\varepsilon)}\ge2$.
	Using $\lfloor t\rfloor\ge2t/3$ for $t\ge2$,
	we obtain $n\ge2$ and
	$\sqrt{LD_Y/(3960\varepsilon)}
	\le n\le\sqrt{LD_Y/(1760\varepsilon)}$.
	Set $A:=88n^4\Delta/(LD_Y^2)$.
	The second accuracy condition gives
	$A\ge L\Delta/(178200\varepsilon^2)\ge1$.
	Hence $M=\lfloor A\rfloor\ge1$ and
	$A/2\le M\le A$.
	
	To prove (i), set $s:=D_Y/(4n^2)$ and
	$c:=LD_Y^2/(88n^4)$.
	Then $\widetilde f_{M,n}^{\mathrm{GS}}(U)
	=c f_{M,n}^{\mathrm{GS}}(U/s)$
	and $c/s^2=2L/11$.
	Proposition~\ref{prop:lb-gs-hard} therefore gives
	a joint gradient Lipschitz constant at most
	$3L/4\le L$, concavity in the dual variables,
	and $\diam(\widetilde Y)=D_Y$.
	The initial value-function gap is
	$\Delta_\phi=cM\in[\Delta/2,\Delta]$,
	since $cA=\Delta$.
	The objective remains polynomial, positive scaling
	preserves the required feasible-set properties,
	and the value function remains bounded below.
	Thus, the scaled problem belongs to
	$\mathcal F(L,D_Y,\Delta)$.
	
	To prove (ii), fix an admissible algorithm and an
	output produced after $K<M(2n+1)$ oracle calls.
	Positive scaling preserves gradient supports, and
	$\Pi_{sC}(u)=s\Pi_C(u/s)$ for every nonempty closed
	convex set $C$.
	The propagation and projection properties of the
	original construction therefore remain valid.
	Together with the output restriction, they give
	$p_M^{\rm out}=0$.
	
	For $u\in C$, the normal cones satisfy
	$N_{sC}(su)=N_C(u)$.
	Since each normal cone is a cone,
	$(c/s)N_C(u)=N_C(u)$.
	Combining these identities with
	$\nabla\widetilde f_{M,n}^{\mathrm{GS}}(U)
	=(c/s)\nabla f_{M,n}^{\mathrm{GS}}(U/s)$
	shows that
	$\widetilde R(U)=(c/s)R(U/s)$
	for every feasible $U$.
	In particular,
	$\widetilde R(U^{\rm out})
	=[LD_Y/(22n^2)]R(U^{\rm out}/s)$.
	Because $U^{\rm out}/s\in X\times Y$ and its
	$p_M$ coordinate vanishes,
	\eqref{eq:lb-gs-barrier} makes this quantity strictly
	greater than $LD_Y/(1760n^2)\ge\varepsilon$.
	The last inequality follows from the upper bound
	on $n$.
	This proves \eqref{eq:lb-gs-query-barrier}.
	
	Finally, $M\ge A/2$ gives
	$M(2n+1)\ge88n^5\Delta/(LD_Y^2)$.
	Using $n\ge\sqrt{LD_Y/(3960\varepsilon)}$, we obtain
	$M(2n+1)\ge
	(88/3960^{5/2})L^{3/2}D_Y^{1/2}
	\Delta\varepsilon^{-5/2}$.
	The claimed worst-case oracle lower bound follows.
\end{proof}

The GS instance has dual curvature of order $n^{-2}$.
As in the OS construction, choosing $n$ according to
$\kappa=L/\mu$ yields the strongly concave lower bound.

\begin{theorem}[NC--SC GS complexity lower bound]
	\label{thm:lb-ncsc-gs}
	Fix $L,D_Y,\Delta>0$ and $0<\mu\le L$, and set
	$\kappa:=L/\mu$.
	Let $\varepsilon$ satisfy
	\begin{equation}\label{eq:lb-ncsc-gs-range}
		0<\varepsilon\le
		\min\left\{
		\frac{\sqrt{L\Delta}}{20\sqrt{66}},
		\frac{LD_Y}{160\kappa}\right\}.
	\end{equation}
	For $\kappa\ge33$, choose
	\begin{equation}\label{eq:lb-ncsc-gs-params}
		n=\left\lfloor\sqrt{\frac{4\kappa}{33}}\right\rfloor,
		\qquad
		M=\left\lfloor
		\frac{L\Delta}{26400\varepsilon^2}\right\rfloor,
	\end{equation}
	and define
	\begin{align}
		\widetilde f_{M,n}^{\mathrm{GS,SC}}(p,q,v;y,z)
		&=\frac{26400\varepsilon^2}{L}
		f_{M,n}^{\mathrm{GS}}\left(
		\frac{L}{330\varepsilon}(p,q,v);
		\frac{L}{330\varepsilon}(y,z)\right),\notag\\
		\widetilde X&=\frac{330\varepsilon}{L}X,
		\qquad
		\widetilde Y=\frac{330\varepsilon}{L}Y.
		\label{eq:lb-ncsc-gs-scaling}
	\end{align}
	Let $\widetilde R$ denote the constrained GS residual
	of the scaled problem.
	Then $n\ge2$, $M\ge1$, and the following statements hold.
	\begin{enumerate}[label=(\roman*),leftmargin=*]
		\item
		The scaled problem belongs to
		$\mathcal F_\mu(L,D_Y,\Delta)$.
		
		\item
		Consider any projected zero-respecting algorithm
		initialized at the origin.
		If it makes $K<M(2n+1)$ oracle calls, then every
		feasible output
		$U^{\rm out}\in\widetilde X\times\widetilde Y$
		satisfying the output restriction in
		Definition~\ref{def:lb-zero-respecting} obeys
		\begin{equation}\label{eq:lb-ncsc-gs-query-barrier}
			\widetilde R(U^{\rm out})>\varepsilon.
		\end{equation}
	\end{enumerate}
	For $\kappa\ge33$, the worst-case number of joint
	gradient evaluations required by this algorithm class
	to produce an $\varepsilon$-GS point over
	$\mathcal F_\mu(L,D_Y,\Delta)$ is therefore at least
	$M(2n+1)\ge
	\sqrt\kappa\,L\Delta/
	(19800\sqrt{33}\,\varepsilon^2)$.
	For $1\le\kappa<33$, a reduction from nonconvex
	minimization gives the same asymptotic order.
	Thus, for every $\kappa\ge1$, the worst-case oracle
	complexity is
	$\Omega(\sqrt\kappa\,L\Delta\varepsilon^{-2})$.
\end{theorem}

\begin{proof}
	First suppose $\kappa\ge33$.
	Then $t:=\sqrt{4\kappa/33}\ge2$, so
	$n=\lfloor t\rfloor\ge2$ and
	$4\sqrt\kappa/(3\sqrt{33})\le n
	\le2\sqrt\kappa/\sqrt{33}$.
	Set $B:=L\Delta/(26400\varepsilon^2)$.
	Condition~\eqref{eq:lb-ncsc-gs-range} gives $B\ge1$.
	Thus $M=\lfloor B\rfloor\ge1$ and
	$B/2\le M\le B$.
	
	To prove (i), set $s:=330\varepsilon/L$ and
	$c:=26400\varepsilon^2/L$.
	Then $\widetilde f_{M,n}^{\mathrm{GS,SC}}(U)
	=c f_{M,n}^{\mathrm{GS}}(U/s)$ and
	$c/s^2=8L/33$.
	Proposition~\ref{prop:lb-gs-hard} gives a joint
	gradient Lipschitz constant at most $L$ and a dual
	strong concavity modulus at least
	$4L/(33n^2)\ge\mu$.
	Moreover,
	$\diam(\widetilde Y)=1320n^2\varepsilon/L
	\le160\kappa\varepsilon/L\le D_Y$.
	The initial value-function gap is
	$\Delta_\phi=cM\in[\Delta/2,\Delta]$,
	since $cB=\Delta$.
	The objective remains polynomial, positive scaling
	preserves the required feasible-set properties,
	and the value function remains bounded below.
	Hence, the scaled problem belongs to
	$\mathcal F_\mu(L,D_Y,\Delta)$.
	
	To prove (ii), fix an admissible algorithm and an
	output produced after $K<M(2n+1)$ oracle calls.
	Positive scaling preserves gradient supports and
	the projection propagation properties, as in the
	proof of Theorem~\ref{thm:lb-gs}.
	Together with the output restriction, these properties
	give $p_M^{\rm out}=0$.
	The normal-cone scaling identities used in that
	proof also give
	$\widetilde R(U^{\rm out})
	=(c/s)R(U^{\rm out}/s)
	=80\varepsilon R(U^{\rm out}/s)$.
	Since $U^{\rm out}/s$ is feasible and its $p_M$
	coordinate vanishes, \eqref{eq:lb-gs-barrier}
	yields $R(U^{\rm out}/s)>1/80$.
	Therefore $\widetilde R(U^{\rm out})>\varepsilon$,
	proving \eqref{eq:lb-ncsc-gs-query-barrier}.
	
	Finally, $M\ge L\Delta/(52800\varepsilon^2)$ and
	$n\ge4\sqrt\kappa/(3\sqrt{33})$ imply
	$M(2n+1)\ge2Mn
	\ge\sqrt\kappa\,L\Delta/
	(19800\sqrt{33}\,\varepsilon^2)$.
	This establishes the claimed lower bound for
	$\kappa\ge33$.
	
	For $1\le\kappa<33$, take the $L$-smooth nonconvex
	minimization hard instance $g:\R^d\to\R$ from
	\cite[Theorem~1]{carmon2020lower}, scaled so that
	$g(0)-\inf g\le\Delta$, with gradient accuracy
	$\varepsilon$.
	Define $\widehat f(x,y):=g(x)-\mu y^2/2$
	on $\widehat X:=\R^d$ and
	$\widehat Y:=[-D_Y/2,D_Y/2]$.
	As in the proof of Theorem~\ref{thm:lb-ncsc-os},
	this problem belongs to
	$\mathcal F_\mu(L,D_Y,\Delta)$.
	Each joint gradient evaluation can be simulated
	by one gradient evaluation of $g$, and the primal
	queries and outputs retain the zero-respecting
	restrictions for nonconvex minimization.
	
	Since $N_{\widehat X}(x)=\{0\}$, the constrained
	GS residual satisfies
	$\widehat R(x,y)\ge\|\nabla g(x)\|$.
	Every $\varepsilon$-GS output therefore gives a
	point with $\|\nabla g(x)\|\le\varepsilon$.
	The nonconvex minimization lower bound yields
	$\Omega(L\Delta\varepsilon^{-2})$ joint gradient
	evaluations.
	Because $1\le\sqrt\kappa<\sqrt{33}$, this is also
	$\Omega(\sqrt\kappa\,L\Delta\varepsilon^{-2})$,
	with a universal implied constant.
\end{proof}


\begin{corollary}[\rev{Fixed-budget randomized outputs}]
	\label{cor:lb-randomized-outputs}
	\begin{revision}
		The explicit chain barriers in
		Theorems~\ref{thm:lb-os}, \ref{thm:lb-ncsc-os},
		\ref{thm:lb-gs}, and \ref{thm:lb-ncsc-gs}
		hold for every realization of a randomized algorithm
		satisfying Definition~\ref{def:lb-zero-respecting}.
		In particular, below the corresponding query threshold,
		no such fixed-budget algorithm can satisfy
		$\mathbb E\|\nabla\Phi_{2L}(x^{\rm out})\|^2
		\le\varepsilon^2$ for OS or
		$\mathbb E R(U^{\rm out})^2\le\varepsilon^2$ for GS
		on the respective hard instance.
	\end{revision}
\end{corollary}

\begin{proof}
	\begin{revision}
		The propagation argument depends only on supports
		and the number of oracle calls, not on the numerical
		choices made within the revealed coordinates.
		It therefore applies to every realization of the
		internal randomness.
		Below the threshold, every admissible output has
		residual strictly greater than $\varepsilon$.
		The squared residual minus $\varepsilon^2$ is thus
		a strictly positive random variable, so its expectation
		is strictly positive, possibly infinite.
	\end{revision}
\end{proof}

\begin{remark}[\rev{Leading-order comparison in the oracle model}]
	\label{rem:lb-optimality}
	\begin{revision}
		The lower bounds have the same parameter dependence
		as the leading terms in Theorem~\ref{thm:end-to-end}:
		$L^2D_Y\Delta\varepsilon^{-3}$ for NC--C OS,
		$L^{3/2}D_Y^{1/2}\Delta\varepsilon^{-5/2}$
		for NC--C GS, and
		$\sqrt\kappa\,L\Delta\varepsilon^{-2}$
		for both NC--SC criteria.
		The explicitly scaled chain instances satisfy
		$\Delta_\phi=\Theta(\Delta)$, and the comparison
		applies in the intersection of the upper- and
		lower-bound accuracy ranges.
		
		The model permits nested iterations and finite
		sequences of free internal projections.
		Proposition~\ref{prop:algorithm-membership} places
		zero-initialized MC-EGM, including warm-up and its
		output rules, within this model.
		The randomized comparison uses fixed query budgets.
		The pathwise barrier also rules out successful
		termination before the threshold for an algorithm
		required to return an $\varepsilon$-stationary
		point almost surely; it does not by itself
		characterize expected stopping time under an
		expected-residual criterion.
		No lower bound for unrestricted first-order
		algorithms is claimed.
	\end{revision}
\end{remark}

\section{Numerical experiments}\label{sec:experiments}

We evaluate MC-EGM and its warm-up variant on two problems:
a dense quartic minimax problem and a fixed-candidate
learning problem constructed from MNIST.
The experiments examine whether the proposed single-loop
method translates its theoretical guarantees into
competitive gradient complexity and runtime.

We compare against four single-loop methods:
GDA~\cite{lin2020gda}, AGP~\cite{xu2023agp},
Smoothed GDA~\cite{zhang2020smoothed,li2026smoothing},
and Perturbed Smoothed GDA~\cite{li2026smoothing},
as well as two nested methods:
Perturbed Smoothed FOAM~\cite{li2026smoothing,kovalev2022optimal}
and Minimax-PPA~\cite{lin2020ppa}.
Convergence is measured by the constrained
game-stationarity residual $R(x,y)$, denoted by GS in
the figures.
Computational cost is reported in full-gradient oracle
calls and elapsed computation time; one oracle call
returns both gradient blocks.
For each problem, all methods use the same initial point.
Each method is evaluated using the best-performing
configuration identified in its respective parameter
search.

\begin{revision}
	The comparisons below are descriptive results for the
	reported instances and configurations.
	They do not estimate worst-case complexity exponents,
	and they do not establish a uniform runtime advantage.
	In particular, the experiments evaluate GS rather than
	the randomized Moreau-envelope OS guarantee.
	The nonsmooth MNIST experiment is also outside the
	assumptions of the complexity analysis.
\end{revision}

%
%
%
%
%

\subsection{Synthetic experiment}\label{subsec:exp-dense}

We consider the box-constrained minimax problem
\begin{equation}\label{eq:exp-dense}
	\min_{x\in[-2,2]^{500}}\max_{y\in[-10,10]^{500}}
	\sum_{i=1}^{500}\left(\frac{x_i^4}{4}-\frac{x_i^2}{2}\right)
	+y^\top Ax,
\end{equation}
where $A\in\R^{500\times500}$ is a fixed dense matrix
with smallest and largest singular values $2$ and $12$,
respectively.
The objective is nonconvex in $x$ and affine in $y$,
placing this problem in the NC--C setting.
The dense bilinear term couples the primal and dual
variables. All methods use the same stored matrix and feasible initial pair,
and each run terminates when $\mathrm{GS}\le10^{-6}$.
Figure~\ref{fig:exp-dense} and Table~\ref{tab:exp-dense}
show that MC-EGM reaches this tolerance with fewer gradient
calls and less measured runtime than any baseline that
also reaches it. 
Specifically, MC-EGM requires $217$ calls and
$23.316$~ms, compared with $284$ calls and $24.414$~ms
for Perturbed Smoothed GDA, the strongest baseline on
both metrics.
This corresponds to a $23.6\%$ reduction in gradient
calls and a more modest $4.5\%$ reduction in measured time.
The warm-up variant further reduces the cost to
$209$ calls and $22.729$~ms, improving on the same
baseline by $26.4\%$ and $6.9\%$, respectively.

GDA does not reach the prescribed tolerance in the
reported run.
Neither nested method reaches it within the
$500$-call budget: the final GS residuals are
approximately $6.26\times10^{-3}$ for Perturbed
Smoothed FOAM and $5.45\times10^{-2}$ for Minimax-PPA.
On this instance, MC-EGM therefore achieves the target
accuracy more efficiently without the need to solve nested
subproblems, while the one-time warm-up provides an additional,
albeit smaller, improvement.
The gains are clearest in oracle calls; the runtime
differences are comparatively modest.
\rev{The reported nested-method endpoints use their
	last completed outer outputs, as specified in the figure
	caption; the comparison does not assert that every
	intermediate point generated by an inner solver has a
	larger GS residual.}

\begin{figure}[!htbp]
	\centering
	\includegraphics[width=\linewidth]{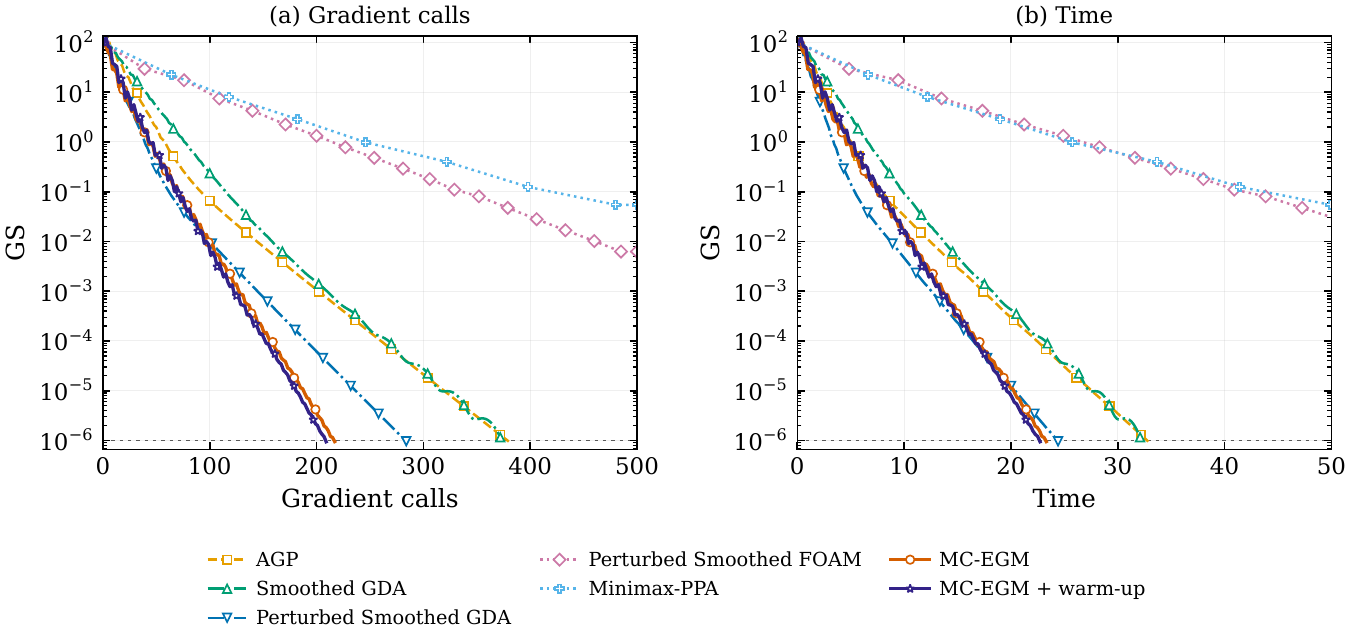}
	\caption{Convergence on the dense quartic problem:
		GS residual versus full-gradient calls (left) and computation
		time (right). Both vertical axes use a logarithmic scale;
		dotted horizontal lines indicate the target $\mathrm{GS}=10^{-6}$.
		The GDA trajectory is omitted because its iterates cycle
		without reaching the target tolerance.
		Times are pointwise medians over $15$ repeated runs,
		excluding diagnostic and residual-evaluation costs;
		the right panel is restricted to $[0,50]$~ms.
		Both cost metrics include algorithm-specific warm-up
		and inner-loop work.
		If the oracle budget is exhausted during the solution of an inner subproblem,
		the final reported residual is that of the last completed
		outer output.}
	\label{fig:exp-dense}
\end{figure}

\begin{table}[!htbp]
	\centering
	\caption{Cost of reaching $\mathrm{GS}\le10^{-6}$ on
		the dense quartic problem.
		Times are medians over $15$ repeated runs, with the
		cost accounting described in Figure~\ref{fig:exp-dense}.
		A dash indicates that the method did not reach the
		target within the $500$-call budget.
		Bold entries indicate the lowest cost in each column.}
	\label{tab:exp-dense}
	\small
	\renewcommand{\arraystretch}{1.08}
	\begin{tabular*}{\linewidth}{@{\extracolsep{\fill}}lrr@{}}
		\toprule
		Method & Gradient calls & Time (ms)\\
		\midrule
		\multicolumn{3}{@{}l}{\textit{Single-loop baselines}}\\
		GDA                      & --- & ---\\
		AGP                      & 380 & 32.863\\
		Smoothed GDA             & 374 & 32.266\\
		Perturbed Smoothed GDA   & 284 & 24.414\\
		\addlinespace
		\multicolumn{3}{@{}l}{\textit{Nested baselines}}\\
		Perturbed Smoothed FOAM  & --- & ---\\
		Minimax-PPA              & --- & ---\\
		\midrule
		\multicolumn{3}{@{}l}{\textit{Proposed methods}}\\
		MC-EGM                  & 217 & 23.316\\
		MC-EGM + warm-up         & \textbf{209} & \textbf{22.729}\\
		\bottomrule
	\end{tabular*}
\end{table}

\subsection{Robust neural network training on MNIST}
\label{subsec:exp-mnist}


\rev{We next evaluate the methods on a finite-candidate
	robust training problem.}
We use $n=20$ MNIST examples with digit labels $b_i$,
each associated with $m=10$ fixed candidate images
$a_{ij}\in\R^{1\times28\times28}$.
Unlike a formulation with separate adversarial weights
for each sample, we use a shared vector of simplex
weights across all samples.
Let $h_\theta$ be a neural-network classifier and
$\ell$ the cross-entropy loss.
The training problem is
\begin{equation}\label{eq:exp-mnist}
	\min_{\theta\in\R^{11978}}\max_{y\in\Delta_{10}}
	\frac1{20}\sum_{i=1}^{20}\sum_{j=1}^{10}
	y_j\ell\bigl(h_\theta(a_{ij}),b_i\bigr),
\end{equation}
where $\Delta_{10}$ is the probability simplex.
For fixed $\theta$, maximizing over $y$ selects the
candidate index with the largest average loss across
the samples.
Thus, the model minimizes the worst average loss among
the ten candidate groups, rather than allowing an
independent candidate choice for each sample.

The classifier has two convolutional layers with
$8$ and $16$ output channels and $5\times5$ kernels,
each followed by ReLU activation and $2\times2$ max pooling.
The resulting $16\times4\times4$ feature tensor is
passed through a fully connected layer of width $32$
with ReLU activation and a linear output layer with
ten logits.
The network has $11978$ parameters, including biases.
Because ReLU and max pooling are nonsmooth, this
experiment lies beyond the smooth setting of our theory.
The reported GS diagnostic is computed using the
derivatives returned by automatic differentiation;
it is not a certificate of smooth stationarity.

\begin{revision}
	The simplex does not contain the origin, but it can be
	translated by the uniform vector
	$\bar y=\boldsymbol1_{10}/10$.
	For $y\in\Delta_{10}$,
	$\|y-\bar y\|^2=\|y\|^2-1/10$.
	Thus, dual quadratic regularization centered at $\bar y$
	and at the origin differs only by a constant on the
	simplex.
	This observation relates the regularized objectives;
	it does not place the nonsmooth network within
	Assumption~\ref{ass:main}.
\end{revision}

All methods use the same initial network parameters and
uniform initial dual weights $y_j^0=1/10$.
Reported costs include all algorithm-specific warm-up
and inner-loop work.
Unlike the synthetic experiment, MNIST timings are measured
in a single run and include residual evaluations.
This experiment evaluates optimization efficiency on the
fixed training problem, not test accuracy or robustness
to unrestricted image perturbations.

Figure~\ref{fig:exp-mnist} and Table~\ref{tab:exp-mnist}
show that all eight methods reach $\mathrm{GS}\le10^{-2}$.
Among the baselines, Minimax-PPA uses the fewest gradient
calls ($6882$), whereas Perturbed Smoothed GDA has the
lowest measured time ($84.954$~s).
MC-EGM improves on both baseline benchmarks, reaching
the target in $6099$ calls and $71.668$~s.
These results correspond to $11.4\%$ fewer calls than
Minimax-PPA and $15.6\%$ less measured time than
Perturbed Smoothed GDA.

The warm-up variant achieves the lowest cost on both
metrics, requiring $4807$ calls and $58.108$~s.
Relative to MC-EGM without warm-up, this reduces
gradient calls by $21.2\%$ and measured time by $18.9\%$.
Relative to the best baseline on each metric, the
reductions are $30.2\%$ and $31.6\%$, respectively.
\rev{On this instance and for the selected configurations,
	the warm-up overhead is more than offset by subsequent
	savings.}
They also show that MC-EGM attains the target residual
without the need to solve nested subproblems while outperforming
the tested nested methods in both cost metrics.
The runtime comparisons remain descriptive, as they
are based on single runs.

\begin{figure}[!htbp]
	\centering
	\includegraphics[width=\linewidth]{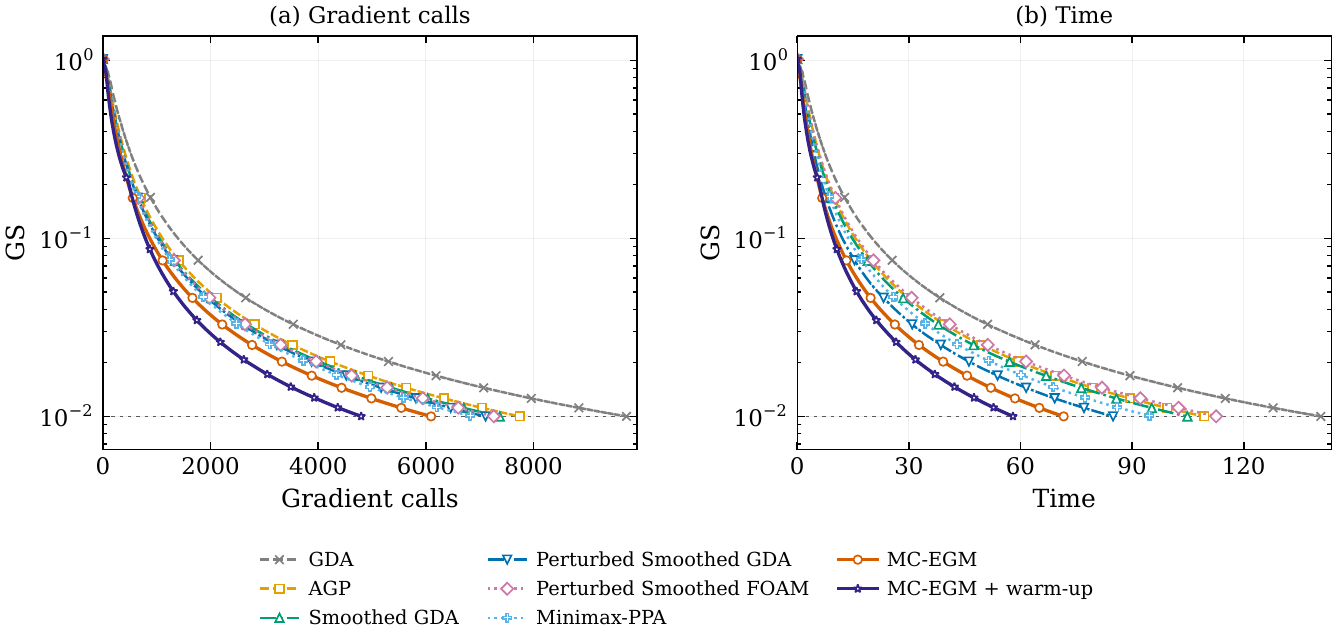}
	\caption{Convergence on the MNIST finite-candidate training problem:
		GS diagnostic versus full-gradient calls (left) and computation
		time (right). Both vertical axes use a logarithmic scale;
		dotted horizontal lines indicate the target $\mathrm{GS}=10^{-2}$.
		Times are single-run measurements in seconds and include
		residual evaluations.
		Both cost metrics include algorithm-specific warm-up
		and inner-loop work.}
	\label{fig:exp-mnist}
\end{figure}

\begin{table}[!htbp]
	\centering
	\caption{Cost of reaching $\mathrm{GS}\le10^{-2}$
		on the MNIST finite-candidate training problem.
		Times are single-run measurements, with the cost
		accounting described in Figure~\ref{fig:exp-mnist}.
		Bold entries indicate the lowest cost in each column.}
	\label{tab:exp-mnist}
	\small
	\renewcommand{\arraystretch}{1.08}
	\begin{tabular*}{\linewidth}{@{\extracolsep{\fill}}lrr@{}}
		\toprule
		Method & Gradient calls & Time (s)\\
		\midrule
		\multicolumn{3}{@{}l}{\textit{Single-loop baselines}}\\
		GDA                     & 9725 & 140.757\\
		AGP                     & 7748 & 109.424\\
		Smoothed GDA            & 7374 & 104.929\\
		Perturbed Smoothed GDA  & 7110 & 84.954\\
		\addlinespace
		\multicolumn{3}{@{}l}{\textit{Nested baselines}}\\
		Perturbed Smoothed FOAM & 7275 & 112.821\\
		Minimax-PPA             & 6882 & 95.474\\
		\midrule
		\multicolumn{3}{@{}l}{\textit{Proposed methods}}\\
		MC-EGM                 & 6099 & 71.668\\
		MC-EGM + warm-up        & \textbf{4807} & \textbf{58.108}\\
		\bottomrule
	\end{tabular*}
\end{table}

\subsection{\rev{Scope of the numerical evidence}}
\label{subsec:experimental-scope}

\begin{revision}
	The two experiments illustrate the behavior of the
	single-loop framework and its one-time warm-up, but
	they address only a limited part of the theoretical claims.
	The quartic example satisfies the smooth NC--C assumptions
	on its compact domains.
	The MNIST example instead measures an automatic-differentiation
	diagnostic for a small nonsmooth training problem.
	Neither experiment tests the NC--SC condition-number
	dependence, the OS guarantee, or the worst-case accuracy
	exponents.
	
	The comparisons also combine differences in update rules,
	parameter choices, and output mechanisms.
	They should therefore be interpreted as performance of
	the tested implementations rather than as an empirical
	ordering of the corresponding algorithm classes.
	A broader evaluation of parameter sensitivity, independent
	initializations, and individual algorithmic components
	would be needed to assess robustness of these observations.
\end{revision}

\section{Conclusions}\label{sec:conclusions}
This work closes the leading-order complexity gap between
single-loop and nested first-order methods for smooth
nonconvex--concave minimax optimization under both
optimization stationarity (OS) and game stationarity (GS).
The proposed MC-EGM combines projected extragradient steps,
dual momentum, and proximal-center updates within a single
loop, without the need to solve nested auxiliary problems.
A unified Lyapunov analysis controls optimization progress
and saddle-point tracking, while a one-time fixed-center
warm-up makes the leading complexity terms depend on
$\Delta_\phi$ rather than $\Delta_0$, with only a lower-order
initialization cost.

In the nonconvex--concave setting, the resulting leading
oracle complexities are
$O(L^2D_Y\Delta_\phi\varepsilon^{-3})$ for OS and
$O(L^{3/2}D_Y^{1/2}\Delta_\phi\varepsilon^{-5/2})$ for GS.
Thus, a single-loop method attains both rates without a
logarithmic factor in the leading terms.
Under dual strong concavity, the leading complexity
improves to
$O(\sqrt{\kappa}\,L\Delta_\phi\varepsilon^{-2})$
for both criteria.
These statements separate the additive initialization
cost and describe the leading behavior for fixed problem
and initialization quantities with $\Delta_\phi>0$
as $\varepsilon\downarrow0$.
The end-to-end theorem provides budgets based on
	supplied upper bounds, a randomized OS output satisfying
	an expected squared Moreau-gradient guarantee, and a
	deterministic GS output selected by a computable certificate.

Matching lower bounds establish the optimality of all four
leading complexity orders within the projected zero-respecting
oracle model.
In particular, the constructions certify the dependence on
the displayed problem parameters, rather than only the
accuracy exponents.
The lower bounds allow competing methods with nested
iterations or additional internal stages, as well as
randomized methods satisfying the support restrictions
for every realization.
The randomized output comparison is made at a fixed
	query budget.
	The model explicitly permits free internal projections,
	and the algorithm-membership result places the
	zero-initialized MC-EGM framework within this model.
The optimality claim is therefore not restricted to
single-loop competitors, although it remains specific
to the stated oracle model and accuracy ranges.

Experiments on a dense quartic minimax problem and a
finite-candidate MNIST training problem complement the
theory.
MC-EGM reaches the prescribed residual thresholds with
fewer gradient calls and lower measured runtime than the
tested baselines, and the one-time warm-up further reduces
both costs.
The MNIST experiment also illustrates empirical performance
beyond the smooth setting of the analysis, without extending
the theoretical guarantees to nonsmooth networks.

Together, these results show that optimal leading-order
oracle complexity for both OS and GS can be attained
through coordinated single-loop updates rather than
nested solvers.
Natural directions for future work include stochastic
extensions, unbounded dual domains, and lower bounds
for broader first-order oracle classes.
A further question is whether the complete
	initialization dependence can be characterized or
	improved under the same basic problem assumptions.

\bmhead{Acknowledgements}

This work was supported by the National Key R\&D Program
of China (Grant Nos.~2025YFA1017801 and 2025YFA1017800)
and the National Natural Science Foundation of China
(Grant No.~12471294).


\end{document}